\documentclass[11pt, reqno]{amsart}
\usepackage{graphicx, amssymb, amsmath, amsthm}
\usepackage{esint}
\usepackage[utf8]{inputenc}
\usepackage[
  backend=biber,
  style=numeric,
  maxbibnames=99,
  giveninits=true
]{biblatex}
\usepackage{epsfig}
\usepackage{hyperref}
\usepackage{comment}
\usepackage{mathrsfs}
\numberwithin{equation}{section}
\usepackage{amsthm}
\usepackage{subfigure}
\usepackage{tikz}
\usepackage{here}
\usepackage{float}
\usepackage{pifont}
\usepackage{enumitem}
\usetikzlibrary{matrix,arrows}
\usetikzlibrary{shapes}
\usetikzlibrary{calc}
\usetikzlibrary{arrows}
\usetikzlibrary{decorations.pathreplacing,decorations.markings}
\usepackage[all]{xy}

\usepackage{tikz-cd}

\usetikzlibrary{patterns}

\newtheorem{theorem}{Theorem}[section]

\newtheorem{lemma}[theorem]{Lemma}

\newtheorem{corollary}[theorem]{Corollary}

\newtheorem{proposition}[theorem]{Proposition}

\theoremstyle{definition}
\newtheorem{definition}[theorem]{Definition}
\newtheorem{remark}[theorem]{Remark}

\makeatletter
\newcommand{\Extend}[5]{\ext@arrow0099{\arrowfill@#1#2#3}{#4}{#5}}
\makeatother

\DeclareMathOperator{\Ric}{Ric}
\DeclareMathOperator{\dist}{dist}

\DeclareMathOperator{\ddiv}{div}
\DeclareMathOperator{\tr}{tr}

\newcommand{\dd}{\mathrm d}

\newcommand{\ve}{\varepsilon}

\begin{document}
\title{Riemannian Penrose inequality in all dimensions}

\author{Yuchen Bi}
\address{Mathematical Institute, Department of Pure Mathematics,
University of Freiburg, Ernst-Zermelo-Straße 1,
D-79104 Freiburg im Breisgau, Germany}
\email{yuchen.bi@math.uni-freiburg.de}

\author{Jintian Zhu}
\address{Institute for Theoretical Sciences, Westlake University,
600 Dunyu Road, 310030, Hangzhou, Zhejiang,
People's Republic of China}
\email{zhujintian@westlake.edu.cn}

\begin{abstract}
We are able to prove the Riemannian Penrose inequality in all dimensions for smooth complete asymptotically flat manifolds
with nonnegative scalar curvature and compact almost-minimizing minimal Caccioppoli boundary allowing mild singularities. Moreover, the equality holds exactly when the manifold is isometric to a Riemannian Schwarzschild exterior. Our proof extends Bray's conformal-flow method to
dimensions higher than seven, where the outermost minimizing enclosures along the flow may be singular.
\end{abstract}
\maketitle

\section{Introduction}

Let \((M^n,g)\), \(n\geq3\), be a smooth complete Riemannian manifold.  We say that \((M,g)\) is asymptotically flat if
\begin{itemize}
    \item there is a compact set $K$ such that $M\setminus K$ is diffeomorphic to \(\mathbb R^n\setminus B_1\);
\item in the corresponding Euclidean coordinates on $M\setminus K$ we have
\[
    |g_{ij}-\delta_{ij}|
    +|x|\,|\partial g_{ij}|
    +|x|^2|\partial^2 g_{ij}|
    =
    O(|x|^{-\tau})
\]
for some \(\tau>(n-2)/2\);\item and \(R_g\in L^1(M,g)\).
\end{itemize}
The ADM mass of $(M,g)$ is defined by
\[
    m_{\mathrm{ADM}}(M,g)
    =
    \frac{1}{2n(n-1)\omega_n}
    \lim_{\rho\to\infty}
    \int_{S_\rho}
    (\partial_jg_{ij}-\partial_i g_{jj})\nu^i\,d\sigma ,
\]
where \(S_\rho=\{|x|=\rho\}\), \(\nu\) is the Euclidean outward unit normal, and \(\omega_n\) denotes the volume of the unit Euclidean \(n\)-ball.

Let \((M^n,g)\) be a smooth complete asymptotically flat manifold with nonnegative scalar curvature and compact outermost minimal boundary
\(\Sigma\).

The Riemannian Penrose inequality
asserts that
\[
    m_{\mathrm{ADM}}(M,g)
    \geq
    \frac12
    \left(
        \frac{\mathcal H_g^{n-1}(\Sigma)}{n\omega_n}
    \right)^{\frac{n-2}{n-1}},
\]
where \(n\omega_n\) is the area of the unit sphere \(\mathbb S^{n-1}\). This mass inequality is sharp, with equality realized exactly by the Riemannian Schwarzschild exteriors given by
$$(M_m,g_m)=\left(\mathbb R^n\setminus B_{r_m},\left(1+\frac{m}{2}|x|^{2-n}\right)^{\frac{4}{n-2}}g_{euc}\right),$$
where
$$m>0\mbox{ and }r_m=\left(\frac{m}{2}\right)^{\frac{1}{n-2}}.$$
In dimension three, the Riemannian Penrose inequality was independently proved by Huisken--Ilmanen \cite{HuiskenIlmanen} and Bray \cite{Bray} with different flow methods.  Huisken and Ilmanen introduced a weak formulation of the inverse mean curvature flow and established the Geroch monotonicity for the Hawking mass along the weak flow.  This compares its value on the initial surface with the ADM mass at infinity and, in particular, gives a proof of the  positive mass theorem and the Riemannian Penrose inequality. Bray introduced a different method, based on a novel conformal flow of metrics.  Along this flow the area of the outermost minimal surface is kept fixed, the ADM mass is monotone non-increasing, and the limiting metric is Schwarzschild. The monotonicity of the ADM mass is proved using the positive mass theorem.  Bray--Lee \cite{Bray-Lee} later extended this conformal-flow approach to dimensions \(3\leq n\leq 7\).

Extending Bray's conformal-flow method to dimensions higher than seven requires two key ingredients. One is a positive mass theorem in the relevant dimension, since the mass monotonicity relies on it, and  the other is a way to handle the outer minimizing enclosures appearing in the conformal flow, which may be singular in dimensions greater than seven.

The positive mass theorem is now known to be true in all dimensions, and we briefly sketch the history.  The positive mass theorem was first proved by Schoen--Yau \cite{SchoenYau1979,SchoenYau1981}, using minimal hypersurfaces, for \(3\leq n\leq 7\), and by Witten \cite{Witten}, using spinors, for spin manifolds in all dimensions.  Schoen--Yau later developed a general-dimensional extension of the minimal-hypersurface approach \cite{SchoenYau2022}.  Along another related line, progress on the regularity theory of area-minimizing hypersurfaces has enlarged the range in which minimal-hypersurface arguments can be applied within the classical regularity framework: Smale \cite{Smale1993} first established a generic regularity theorem in dimension \(8\), and Chodosh--Mantoulidis--Schulze \cite{ChodoshMantoulidisSchulze2023} and
Chodosh--Mantoulidis--Schulze--Wang \cite{ChodoshMantoulidisSchulzeWang2025} subsequently extended the results in dimensions \(9,10\), and 
dimension \(11\), respectively. More recently, He--Shi--Yu \cite{HeShiYuPMT8} gave a new proof of the positive mass theorem in dimension \(8\) by using conformal blow-up to eliminate the singular set to infinity, in combination with a use of the torical symmetrization.  Building on this strategy, Bi--Hao--He--Shi--Zhu \cite{BiHaoHeShiZhuPMT19} proved the Riemannian positive mass theorem up to dimension \(19\), which in fact provides a way to overcome the technical limitations imposed by the regularity theory in geometric measure theory.  Shortly thereafter, Brendle--Wang \cite{BrendleWangPMT} introduced a new dimension-descent scheme involving this conformal blow-up perspective, yielding a proof of the positive mass theorem in arbitrary dimension.

The present paper addresses the second ingredient.  We show that Bray's conformal-flow argument can be carried out in arbitrary dimension, even though the outer minimizing enclosures arising in the flow may be singular.  More precisely, we construct the flow in this generality and prove that, for each \(t>0\), the hypersurface \(\Sigma_t\) admits a decomposition
\[
    \Sigma_t=\mathcal R_t\cup\mathcal S_t,
\]
where \(\mathcal R_t\) is a smooth \(g_t\)-minimal hypersurface and
\[
    \dim_{\mathcal H}\mathcal S_t\leq n-8 .
\]
We then develop the regularity theory needed to work with these singular enclosures, and prove a mass-capacity inequality for asymptotically flat manifolds with such boundaries.  This gives the monotonicity of the ADM mass along the flow and yields the Riemannian Penrose inequality in all dimensions.

With singularity under consideration, we work in the framework of {\it smooth Riemannian manifolds with Caccioppoli boundary}.
\begin{definition}
    $(M,g)$ is a smooth Riemannian manifold with Caccioppoli boundary if $(M,g)$ can be extended to $(\overline M,\overline g)$, where $\overline M$ is a smooth manifold and $\overline g$ is a H\"older Riemannian metric on $\overline M$, and $M$ is the support of a Caccioppoli set of $\overline M$, the regular part of $\partial M$ is smooth, and $g$ is smooth up to the regular part of $\partial M$ from the interior of $M$.
\end{definition}

Our main theorem is
\begin{theorem}\label{Thm: intro main}
Let \((M^n,g)\), \(n\geq3\), be a smooth complete asymptotically flat manifold with nonnegative scalar curvature and compact almost-minimizing Caccioppoli boundary \(\partial M\).  Assume that $\partial M$ is the outermost $g$-minimizing enclosure of itself and that its regular part $\mathcal R(\partial M)$ is $g$-minimal. Assume also that $g$ admits a $C^{2,\alpha}$ Riemannian extension across $\partial M$. Then we have
\[
    m_{\mathrm{ADM}}(M,g)
    \geq
    \frac12
    \left(
        \frac{\mathcal H_g^{n-1}(\partial M)}{n\omega_n}
    \right)^{\frac{n-2}{n-1}} .
\]
Moreover, equality holds if and only if \((M,g)\) is isometric to a Riemannian Schwarzschild exterior.
\end{theorem}

We finish the introduction by outlining the proof and the organization of the paper.

In Section 2, we construct Bray's conformal flow in arbitrary dimension from time-discrete approximations.  Boundary H\"older estimates provide compactness for the harmonic functions defining the flow, and the Almgren--Bombieri--Tamanini theory gives the regular-singular decomposition of
the time slices.  We also prove a local separation criterion for singular almost-minimizing boundaries and use it to show that distinct time slices are disjoint.

In Section 3, we prove the monotonicity of the ADM mass by establishing a mass-capacity inequality for singular boundaries. The separation criterion from Section 2 is used after a harmonic conformal perturbation to obtain a strictly separated boundary whose regular part has
mean curvature bounded below by a positive constant. We then use mean-convex smoothing, in the form formulated by Gromov \cite{GromovPlateauStein}, to reduce the argument to the smooth boundary case, where Bray's doubling trick and the positive mass theorem apply.

In Section 4, we prove the convergence of the normalized conformal flow. This part follows the strategy of Bray--Lee \cite{Bray-Lee} closely; we include the details for completeness. The positive mass theorem enters here through Lee's quantitative positive mass theorem for harmonically flat ends, which follows from the positive mass theorem in the corresponding dimension.  

Finally, in Section 5, we combine the mass monotonicity with the convergence of the flow to prove the Riemannian Penrose inequality.  We also establish the rigidity statement in the equality case.

\subsection*{Acknowledgements}
The second-named author is partially supported by National Key R\&D Program of China 2023YFA1009900, NSFC Grant No. 12401072, Zhejiang Provincial Natural Science Foundation of China under Grant No. LQKWL26A0101 and the start-up fund from Westlake University. The authors used ChatGPT (OpenAI) for copyediting and to assist in reviewing the proofs for errors in calculations or notation. The mathematical results and arguments were developed by the authors, who take full responsibility for the content of this paper.

\section{The conformal flow in the presence of singularity}\label{Sec: 2}

In this section, we always assume $(M^n,g)$ to be a smooth and complete asymptotically flat manifold with compact  Caccioppoli boundary $\partial M$. Assume that $g$ admits a $C^{2,\alpha}$ Riemannian extension across $\partial M$.

Denote
$$\mathcal E=\left\{E:
\mbox{Caccioppoli sets } 
E\mbox{ of }M\mbox{ such that }M\setminus E\mbox{ is bounded}
\right\}$$
and
$$\mathcal C=\{\Sigma:\Sigma=\partial E\mbox{ with }E\in \mathcal E\},$$
where $\partial E$ denotes the measure-theoretic boundary of $E$ given by
$$\partial E:=\{x\in M:\mathcal H^n_g(B_r^g(x)\cap E)>0\mbox{ and }\mathcal H^n_g(B_r^g(x)\setminus E)>0\mbox{ for all }r>0\}.$$ 
For convenience, we call $\Sigma\in \mathcal C$ a hypersurface in $M$ although it may not be smooth, and we always assume $E\in\mathcal E$ whenever we write $\Sigma=\partial E$.

\vspace{2mm} \noindent{\bf Convention I.}  We always omit the notation for reduced boundary just
 for simplicity. Namely,  we denote 
$$\mathcal H^{n-1}_g(\Sigma\cap U):=\mathcal H^{n-1}_g(\partial^*E\cap U),\mbox{ where }\Sigma=\partial E,$$ 
for any open subset $U$ of $M$.
\vspace{2mm}

Throughout, the Caccioppoli sets in $\mathcal E$, their boundaries in
$\mathcal C$, and their complements are understood in the extended manifold.
Extend each discrete conformal factor constantly by one across the
original boundary into the added interior region, and extend each
discrete velocity by zero there.

In order to characterize mild regularity of hypersurfaces in $\mathcal C$, we need the following definition of almost minimizing property.

\begin{definition}\label{Defn: almost minimizing}
       Given any  constants $\beta\geq 0$, $\mu>0$ and $r_0>0$, a boundary $\Sigma=\partial E$ in $\mathcal C$ is $(\beta,\mu,r_0,g)$-almost-minimizing if the following property holds: if $F\in\mathcal E$ is an open subset of $M$ 
       such that $E\Delta F$ is contained in some geodesic ball $B^g_{r}(x)$ 
       with $r\leq r_0$, then we have
    \[
    \mathcal H^{n-1}_g(\Sigma)\leq \mathcal H^{n-1}_g(\partial F)+\beta r^{n-1+\mu}.
    \]
\end{definition}

 {\noindent\bf Convention II.} Whenever $\Sigma=\partial E$ is almost-minimizing, we identify $E$ with its canonical open representative
 $$E^{op}:=\{x\in M:\mathcal H^n_g(B_r^g(x)\setminus E)=0\text{ for some }r>0\}.$$
 By the two-sided density estimate established in the proof of Lemma \ref{Lem: density lower bound}, $E^{op}$ agrees with $E$ up to a set of measure zero, and its topological boundary agrees with the measure-theoretic boundary $\partial E$.

\vspace{2mm}
To state the definition of a conformal flow, we need the following notion of minimizing enclosure.

\begin{definition}
Let $h$ be a continuous Riemannian metric on $M$ and $\Sigma\in \mathcal C$ be a hypersurface with $\Sigma=\partial E$. By $h$-minimizing enclosure of $\Sigma$ we mean a hypersurface $\Sigma_*\in \mathcal C$ with $\Sigma_*=\partial E_*$
 satisfying $E_*\subset E$ and 
$$\mathcal H_h^{n-1}(\Sigma_*)\leq \mathcal H_h^{n-1}(\partial E')$$ for any competitor $E'\in \mathcal E$ with $E'\subset E$.
An $h$-minimizing enclosure $\Sigma_*$ of $\Sigma$ is said to be outermost if there is no $h$-minimizing enclosure of $\Sigma$ enclosing $\Sigma_*$ other than $\Sigma_*$ itself.
\end{definition}
\begin{remark}
 On asymptotically flat manifolds $(M,h)$, the outermost $h$-minimizing enclosure of $\Sigma$ always exists and is unique. 
\end{remark}

In the following, we introduce the definition of {\it a conformal flow}, which is equivalent to the definition raised by Bray in \cite{Bray-Lee}.
\begin{definition}\label{Defn: conformal flow}
    A conformal flow on $(M,g)$ with initial hypersurface $\Sigma_0\in\mathcal C$ means a one-parameter family of metric-hypersurface pairs $$\{(g_t,\Sigma_t)\}_{t\geq 0}$$ satisfying the following properties:
    \begin{itemize}
    \item[(i)] $\Sigma_t\in\mathcal{C}$ is the outermost $g_t$-minimizing enclosure of $\Sigma_0$ in $M$;
        \item[(ii)]  there is a continuous function $v_t$ on $M$ such that
        $$\Delta_{g}v_t=0\mbox{ outside }\Sigma_t,$$
$$v_t\equiv 0\mbox{ inside and along }\Sigma_t,$$
and
        $$v_t(x)\to -e^{-t}\mbox{ as }x\to\infty;$$
        \item[(iii)]  the metrics $g_t$ are all conformal to $g$ and we have 
        $$g_t=u_t^{\frac{4}{n-2}}g\mbox{ with }u_t(x)=1+\int_0^tv_s(x)\,\mathrm ds.$$
    \end{itemize}
\end{definition}

\begin{remark}\label{Rmk: new conformal flow}
    From the discussion in \cite[Appendix A]{Bray}, we actually know that if $\{(g_t,\Sigma_t)\}_{t\geq 0}$ is a conformal flow on $(M,g)$ with initial hypersurface $\Sigma_0$, then for any $s\geq 0$ the subfamily 
    $\{(g_{t+s},\Sigma_{t+s})\}_{t\geq 0}$ is a conformal flow on $(M,g_s)$ with initial hypersurface $\Sigma_s$, where the new evolution function is
    $$\tilde v_t=\frac{v_{t+s}}{u_{s}}.$$
\end{remark}

\subsection{Existence} In this subsection, we are going to establish the following existence result.
\begin{theorem}\label{Prop: existence}
    Assume that $(M,g)$ is a smooth complete asymptotically flat manifold with compact Caccioppoli boundary, and that $\Sigma_0\in\mathcal C$ is a $(\beta,\mu,r_0,g)$-almost-minimizing hypersurface, which is the outermost $g$-minimizing enclosure of itself. Then there is a conformal flow $\{(g_t,\Sigma_t)\}_{t\geq 0}$ on $(M,g)$ with initial hypersurface $\Sigma_0$. Furthermore, for any constant $T>0$, there exist constants
\[
\alpha=\alpha(n,T,g,\beta,\mu,r_0)\in(0,1)
\quad\text{and}\quad
C=C(n,T,g,\beta,\mu,r_0)>0
\]
such that
\[
    \sup_{0\leq t\leq T}\|u_t\|_{C^\alpha(M,g)}\leq C.
\]
\end{theorem}

The proof for Theorem \ref{Prop: existence} was essentially given by Bray in \cite{Bray}, where he established the existence of a conformal flow on any asymptotically flat $3$-manifold. The existence of a conformal flow is well-known to be extended on asymptotically flat manifolds up to dimension seven. In higher dimensions, minimizing hypersurfaces may have singularities, therefore we have to make necessary modifications to ensure the existence.

\subsubsection{The construction of approximation flows} As in \cite{Bray}, for each constant $$0<\epsilon<1/2$$ we construct approximation flows $\{(g_t^\epsilon,\Sigma_t^\epsilon)\}_{t\geq 0}$ by induction such that the desired conformal flow eventually comes from the limit of the approximation flows as $\epsilon\to 0$. 

Given any constant $\epsilon$ we define
$$(g_t^\epsilon,\Sigma_t^\epsilon)\equiv (g,\Sigma_0)\mbox{ for all }t<0.$$
If the metric-hypersurface pairs $(g_t^\epsilon,\Sigma_t^\epsilon)$ are already defined on the interval $(-\infty,k\epsilon)$ for $k\in \mathbb N$, then we can define $(g_t^\epsilon,\Sigma_t^\epsilon)$ on the interval $[k\epsilon,(k+1)\epsilon)$ as follows.
\begin{itemize}
    \item[(i)] Define
    $$g_{k\epsilon}^\epsilon=\lim_{t\uparrow k\epsilon}g_t^\epsilon.$$
    \item[(ii)] For all $t\in [k\epsilon,(k+1)\epsilon)$ define $\Sigma_t^\epsilon$ to be the outermost $g_{k\epsilon}^\epsilon$-minimizing enclosure of $\Sigma^\epsilon_{(k-1)\epsilon}$ in $M$.
    \item[(iii)] Define
    \begin{equation*}
    g^\epsilon_t=(u^\epsilon_t)^{\frac{4}{n-2}}g,
\end{equation*}
with
\begin{equation*}
    u^\epsilon_t(x)=1+\int_0^tv_s^\epsilon(x)\,\mathrm ds,
\end{equation*}
where $v_t^\epsilon$ is the function determined by 
\begin{equation}\label{Eq: v1}
    \Delta_{g}v^\epsilon_t=0\mbox{ outside }\Sigma^\epsilon_t,
\end{equation}
\begin{equation}\label{Eq: v2}
    v^\epsilon_t\equiv 0\mbox{ inside and along }\Sigma^\epsilon_t,
\end{equation}
and
     \begin{equation}\label{Eq: v3}
v^\epsilon_t(x)\to-(1-\epsilon)^{\lfloor\frac{t}{\epsilon}\rfloor}\mbox{ as }x\to\infty,
     \end{equation}  
\end{itemize}
where $\lfloor \frac{t}{\epsilon}\rfloor$ denotes the largest integer no greater than $\frac{t}{\epsilon}$. 

The construction above works smoothly up to dimension seven. To make the construction valid in all dimensions, we need the following existence result for functions $v_t^\epsilon$. \begin{proposition}\label{Prop: v_t^epsilon}
    For each $t\geq 0$ the equations \eqref{Eq: v1}-\eqref{Eq: v3} have a unique solution $v_t^\epsilon$. Moreover, for each $T>0$ there exist constants $\alpha=\alpha(n,T)\in(0,1)$ and $C=C(n,T,g,\beta,\mu,r_0)>0$ such that
    \[
    \sup_{0\leq t\leq T}\|v^\epsilon_t\|_{C^\alpha(M,g)}\leq C.
    \]
\end{proposition}


The existence of $v_t^\epsilon$ will be proved by induction. Assume that $v_t^\epsilon$ already exists for $t\in[0,k\epsilon)$, then we are going to show that $v_t^\epsilon$ exists for $t\in [k\epsilon,(k+1)\epsilon)$. For a priori estimates, let us assume $t\leq T$.

\begin{lemma}\label{Lem: equivalent metric}
    We have
    $$\lambda g\leq g_{t}^\epsilon\leq g,$$
    for some constant $\lambda=\lambda(T)>0$ for all $t\in[0,k\epsilon]$.
\end{lemma}
\begin{proof}
     Since $v_t^\epsilon$ exists for $t\in[0,k\epsilon)$, it follows from the maximum principle that
    $$-(1-\epsilon)^{\lfloor\frac{t}{\epsilon}\rfloor}\leq v_t^\epsilon\leq 0\mbox{ for }t\in[0,k\epsilon).$$
    By integration we have
    $$ (1-\epsilon)^{k}\leq u_{t}^\epsilon\leq 1\mbox{ for all }t\in [0,k\epsilon).$$
    From the fact $\epsilon<1/2$ we see
    $$(1-\epsilon)^k\geq \left[(1-\epsilon)^{\frac{1}{\epsilon}}\right]^{k\epsilon}\geq 4^{-T}.$$
    This completes the proof.
\end{proof}

Let $E_t^\epsilon$ denote the region outside $\Sigma_t^\epsilon$ and $S^g_r(x)$ denote the boundary sphere of geodesic ball $B^g_r(x)$. Since the metric $g$ has a $C^{2,\alpha}$ extension across the original boundary and is asymptotically flat, we can take $\iota$ small enough depending on $g$ such that the metric $g$ is uniformly equivalent to the Euclidean metric, i.e.
$$c_{\iota}^{-2}g_{euc}\leq g\leq c_\iota^2 g_{euc},$$
in all geodesic balls with radius less than $\iota$, where $c_\iota>1$ denotes a constant depending on $\iota$ with $c_\iota\to 1$ as $\iota\to 0$. In particular, we have the isoperimetric inequality
$$\mathcal H^n_g(\Omega)^{\frac{n-1}{n}}\leq c_{iso}(n)\cdot\mathcal H^{n-1}_g(\partial^*\Omega) $$
 when $\Omega$ is contained in some geodesic ball with radius less than $\iota$.

At grid times, $\Sigma_{k\epsilon}^\epsilon$ is the outermost
$g_{k\epsilon}^\epsilon$-minimizing enclosure of $\Sigma_0$, by the union
and intersection argument of \cite[Lemma 3]{Bray}. In particular, its
$g_{k\epsilon}^\epsilon$-area is at most
$\mathcal H_g^{n-1}(\Sigma_0)$.
 \begin{lemma}\label{Lem: almost minimizing}
    Let $F\in \mathcal E$ be an open subset of $M$ such that $E_t^\epsilon\Delta F$ is contained in some geodesic ball $B^g_r(x)$ with $r\leq r_0$. Then we have
    $$\mathcal H^{n-1}_g(\Sigma_t^\epsilon\cap B^g_r(x))\leq \lambda^{-1}\left(\mathcal H^{n-1}_g(\partial F\cap B^g_r(x))+\beta r^{n-1+\mu}\right).$$
 \end{lemma}
 \begin{proof}
     Write $\Sigma_0=\partial E_0$. Note that $F\setminus E_0\subset F\setminus E_t^\epsilon$, so we know that $(E_0\cup F)\Delta E_0$ is contained in some geodesic ball $B_r^g(x)$ with $r\leq r_0$. Since $\Sigma_0$ is $(\beta,\mu,r_0,g)$-almost-minimizing, we have
     $$\mathcal H^{n-1}_g(\partial E_0\cap B^g_r(x))\leq \mathcal H^{n-1}_g(\partial (E_0\cup F)\cap B^g_r(x))+\beta r^{n-1+\mu}.$$
     Recall that $\Sigma_t^\epsilon$ is $g^\epsilon_{k\epsilon}$-minimizing outside $\Sigma_0$. Take $k=\lfloor t/\epsilon\rfloor$ and decreasing $\lambda$ if necessary,  Lemma \ref{Lem: equivalent metric} gives the following area comparison
     $$\mathcal H^{n-1}_g(\Sigma_t^\epsilon\cap B^g_r(x))\leq \lambda^{-1} \mathcal H^{n-1}_g(\partial (E_0\cap F)\cap B^g_r(x)).$$
     From the strong subadditivity we have 
     \[
     \begin{split}
         \mathcal H^{n-1}_g(\partial (E_0\cup F)&\cap B^g_r(x))+\mathcal H^{n-1}_g(\partial(E_0\cap F)\cap B^g_r(x))\\
         &\leq \mathcal H^{n-1}_g(\partial E_0\cap B^g_r(x))+\mathcal H^{n-1}_g(\partial F\cap B^g_r(x)).
     \end{split}\]
    The proof is completed by combining these inequalities.
 \end{proof}

\begin{lemma}\label{Lem: sphere lower area}
   By taking $\iota$ small enough, there is a constant $c=c(n,T)>0$ such that at any point $x\in \Sigma_t^\epsilon$ we have
    $$\mathcal H^{n-1}_g(S^g_\rho(x)\setminus E_t^\epsilon)\geq c\rho^{n-1}$$
    for almost every $\rho \in (0,\min\{r_0,\iota\})$.
\end{lemma}
\begin{proof}
  First fix $x\in \partial^*E_t^\epsilon$, so that $f(\rho)/\rho^n\to \omega_n/2$ as $\rho\downarrow 0$, where 
    $$f(\rho)=\mathcal H^n_g(B^g_\rho(x)\setminus E_t^\epsilon)\mbox{ with }\rho\in (0,\min\{r_0,\iota\}).$$
    Clearly, $f(\rho)$ is a Lipschitz function and we have
    $$f'(\rho)=\mathcal H^{n-1}_g(S^g_\rho(x)\setminus E_t^\epsilon)$$
    at differentiable points. It follows from Lemma \ref{Lem: almost minimizing} with the choice $$F=E_t^\epsilon\cup B^g_\rho(x)$$ and the isoperimetric inequality that we have
    \begin{equation}\label{Eq: iso}
        \mathcal H^n_g(B^g_\rho(x)\setminus E_t^\epsilon)^{\frac{n-1}{n}}\leq c_{iso}(1+\lambda^{-1})\mathcal H^{n-1}_g(S^g_\rho(x)\setminus E_t^\epsilon)+c_{iso}\lambda^{-1} \beta \rho^{n-1+\mu}.
    \end{equation}
    In particular, we have $$f(\rho)^{\frac{n-1}{n}}\leq c_{iso}(1+\lambda^{-1})f'(\rho)+c_{iso}\lambda^{-1} \beta \rho^{n-1+\mu}$$
    at differentiable points. By taking $\iota$ small enough we can guarantee
    $$c_{iso}\lambda^{-1} \beta \rho^\mu\leq\left(\frac{1}{c_{iso}(1+\lambda^{-1})(n+1)}\right)^nc_{iso}(1+\lambda^{-1}).$$
    The small-radius density and the preceding differential inequality
prevent a first crossing below the following comparison function. Thus
    $$f(\rho)\geq \left(\frac{\rho}{c_{iso}(1+\lambda^{-1})(n+1)}\right)^n$$
    for all $\rho \in (0,\min\{r_0,\iota\})$. Approximating any point of $\Sigma_t^\epsilon$ by reduced-boundary
points and using balls of half the radius gives the same volume lower
bound, with a smaller constant, at every boundary point. Substitution
in \eqref{Eq: iso}, followed by decreasing $\iota$, gives the desired
sphere estimate. The choice $F=E_t^\epsilon\setminus B^g_\rho(x)$
gives the volume lower bound for the other phase as well.
\end{proof}

Let $\chi$ denote the characteristic function of the region $E_t^\epsilon$ given by $\chi\equiv 1$ in $E_t^\epsilon$ and $\chi\equiv 0$ in $M\setminus E_t^\epsilon$. Let $x_*$ be a point on $\Sigma_t^\epsilon$ and $\rho$ be a small regular value in $(0,\min\{r_0,\iota\})$ such that
\begin{equation}\label{Eq: transversal}
    \mathcal H^{n-1}_g(\Sigma_t^\epsilon\cap S^g_{\rho}(x_*))=0.
\end{equation}
Choose $\rho$ also to satisfy the sphere estimate of
Lemma \ref{Lem: sphere lower area}.
Following Perron's method we define
\begin{equation}\label{Eq: w}
    w(x)=\sup_{\mathcal S}u(x),
\end{equation}
where $\mathcal S$ denotes the collection of all continuous subharmonic functions on $\overline{B^g_\rho(x_*)}$ with boundary value no greater than $\chi$. It is easy to see that $w$ is harmonic, nonnegative and continuous up to the boundary portion $$S^g_\rho(x_*)\setminus \Sigma_t^\epsilon.$$
\begin{lemma}\label{Lem: comparison}
  By taking $\iota$ small enough, we can find constants $\delta=\delta(n,T)\in (0,1)$ and $\tau=\tau(n,T)\in (0,1)$ such that
    $$\sup_{B^g_{\delta\rho}(x_*)}w\leq \tau.$$
\end{lemma}
\begin{proof}
    Since $w$ is harmonic and nonnegative, after taking $\iota$ small enough we have
    $$\frac{\mathrm d}{\mathrm dr}\int_{S^g_r(x_*)}w\,\mathrm d\sigma_r=\int_{S^g_r(x_*)}wH_r\,\mathrm d\sigma_r\geq \left(\frac{n-1}{r}-Cr\right)\int_{S^g_r(x_*)}w\,\mathrm d\sigma_r,$$
    where $C$ is a fixed constant depending only on the bound of the curvatures of $g$ and $H_r=\ddiv_{S_r^g(x_*)}\partial_r$.
     From this we see
    $$\frac{1}{n\omega_nr^{n-1}}\int_{S^g_r(x_*)}w\,\mathrm d\sigma_r\geq e^{-\frac{C}{2}r^2}w(x_*).$$
    Let $r\to \rho$. Using the continuity of $w$, Lemma \ref{Lem: sphere lower area}, and the equivalence of $g$ and $g_{euc}$, we obtain
    $$w(x_*)\leq e^{\frac{C}{2}\rho^2}\left(c_\iota^{n-1}-\frac{c}{n\omega_n}\right).$$
   By taking $\iota$ to be small enough such that $c_\iota\approx 1$, we can guarantee
    $w(x_*)\leq 2\tau-1$ for some constant $\tau=\tau(n,c)\in (0,1)$. It follows from \cite[Corollary 6.2]{PeterLi} that we have the Harnack inequality
    $$\sup_{B^g_{r/4}(x_*)}u\leq C_*\inf_{B^g_{r/4}(x_*)}u$$
   for any nonnegative harmonic function $u$ on $B^g_r(x_*)$, where $C_*=C_*(n)$ after decreasing $\iota$ so that
$\iota^2\sup|\operatorname{Rm}_g|\leq1$. In the following, we consider the oscillation function $\omega(r)$ associated to $w$ defined by
    $$\omega(r)=\sup_{B^g_r(x_*)}w-\inf_{B^g_r(x_*)}w.$$
    From the Harnack inequality we can derive
    $$\omega(\delta\rho)\leq (4\delta)^\alpha\omega(\rho)\mbox{ with }\alpha=\log_4\left(\frac{C_*+1}{C_*-1}\right).$$
    Using the fact $\omega(\rho)\leq 1$ and taking $\delta=\delta(\tau)$ small enough, we can guarantee
$$\sup_{B^g_{\delta\rho}(x_*)}w\leq \tau.$$
    This completes the proof.
\end{proof}
\begin{remark}
    By carefully tracking the dependence of $\iota$, we know that the constant $\iota<r_0$ can be chosen small enough depending only on $n$, $T$, $g$, $\beta$, $\mu$ and $r_0$. For later use, the same zero-Dirichlet boundary H\"older estimate
holds for uniformly elliptic divergence-form operators under the
exterior volume density bound.
\end{remark}

Denote $$K_t^\epsilon=M\setminus E_t^\epsilon.$$ Notice that $K_t^\epsilon$ is a closed subset of $M$. In particular, we can construct a smooth nonnegative function 
$$f:M\to \mathbb R$$ such that $K_t^\epsilon=\{f\leq 0\}$ and $f(x)\to+\infty$ as $x\to\infty$. Take a sequence of regular values $\epsilon_i\to 0$ of $f$ as $i\to+\infty$ and denote $V_i=\{f\leq \epsilon_i\}$. It is not difficult to verify $V_i\to K_t^\epsilon$ as $i\to+\infty$ in the sense of Hausdorff distance. Let $v_i$ be the function solving the following equation:
$$\Delta_g v_i=0\mbox{ in }M\setminus V_i,\,v_i=0\mbox{ on }\partial V_i,\,v_i\to-(1-\epsilon)^{\lfloor\frac{t}{\epsilon}\rfloor}\mbox{ as }x\to\infty.$$
\begin{lemma}\label{Lem: Holder v_i}
There are constants $\Lambda=\Lambda(n,T,g,\beta,\mu,r_0)>0$ and $\alpha=\alpha(n,T)\in (0,1)$ such that for any $x\in M\setminus V_i$ with $$\dist_g(x,\Sigma_t^\epsilon)<\min\{r_0,\iota\},$$ we have
    $$|v_i(x)|\leq \Lambda\dist_g(x,\Sigma_t^\epsilon)^\alpha.$$
\end{lemma}
\begin{proof}
  The maximum principle gives $|v_i|\leq 1$. Fix a point $x\in M\setminus V_i$ and let $x_*$ be the nearest point of $x$ on $\Sigma_t^\epsilon$. Denote the function from \eqref{Eq: w} by $w_\rho$. Through a comparison of the function $-v_i$ and the functions $\tau^{j}\cdot w_{\delta^j\iota}$ (approximations are made if $\delta^j\iota$ is not a regular value satisfying \eqref{Eq: transversal}), we have
   $$|v_i(x)|\leq \tau^k$$
   if
$$\delta^{k+1}\iota<\dist_g(x,\Sigma_t^\epsilon)\leq\delta^{k}\iota.$$
Now it is standard to pick constants $\Lambda=\Lambda(\delta,\tau,\iota)$ and $\alpha=\alpha(\delta,\tau)$.
\end{proof}

The grid-time area bound, metric comparison, and asymptotic flatness
give a uniform bound on enclosed volume for $t\leq T$.
If a boundary point has large coordinate radius $R$, the obstacle-free
comparison in Lemma \ref{Lem: sphere lower area}, applied in a coordinate
ball of radius $R/2$, gives enclosed volume at least $cR^n$.
Thus all discrete boundaries lie in a fixed compact set.
Choose the approximating sets $V_i$ in a slightly larger fixed set,
and let $\varphi$ be the exterior harmonic potential of a sphere
enclosing it, equal to one on the sphere and zero at infinity.
The maximum principle gives $0\leq v_i+a\leq a\varphi$, where
$a=(1-\epsilon)^{\lfloor t/\epsilon\rfloor}$.
This estimate also holds for the discrete velocities and their limits.

Now we are ready to prove Proposition \ref{Prop: v_t^epsilon}.
\begin{proof}[Proof of Proposition \ref{Prop: v_t^epsilon}]
    Since \(v_i\) is harmonic in \(M\setminus V_i\) and satisfies \(|v_i|\leq 1\),
up to a subsequence we have \(v_i\to v_\infty\) in
\(C^\infty_{\mathrm{loc}}(E_t^\epsilon)\)  for some smooth harmonic function $v_\infty$ in $E_t^\epsilon$. It is clear that we have
    \begin{equation}\label{Eq: boundary holder}
        |v_\infty(x)|\leq \Lambda\dist_g(x,\Sigma_t^\epsilon)^\alpha
    \end{equation}
    whenever $\dist_g(x,\Sigma_t^\epsilon)<\min\{r_0,\iota\}$. In particular, $v_\infty$ is continuous up to $\Sigma_t^\epsilon$ with value zero. Extend $v_\infty$ on $M$ by taking $v_\infty\equiv 0$ on $K_t^\epsilon$. The preceding exterior barrier preserves the prescribed limit at
infinity, and uniqueness follows from the maximum principle. Combining the boundary H\"older estimate \eqref{Eq: boundary holder} and the $C^0$-estimate $|v_\infty|\leq 1$, we can find constants $\alpha=\alpha(n,T)\in(0,1)$ and $C=C(n,T,g,\beta,\mu,r_0)>0$ such that
    \[
    \|v_\infty\|_{C^\alpha(M,g)}\leq C.
    \]
    The proof is completed by taking $v_t^\epsilon=v_\infty$.
\end{proof}
    
As an immediate consequence of Proposition \ref{Prop: v_t^epsilon} and Lemma \ref{Lem: equivalent metric}, we have the following corollary. 
\begin{corollary}\label{Cor: holder}
    For each $T>0$ there exist constants $\alpha=\alpha(n,T)\in(0,1)$ and $C=C(n,T,g,\beta,\mu, r_0)>0$ such that
    \[
    \sup_{0\leq t\leq T}\|u^\epsilon_t\|_{C^\alpha(M,g)}\leq C.
    \]
    Moreover, we have $u^\epsilon_t\geq 4^{-T}$ and $\lambda g\leq g_t^\epsilon\leq g$ for some constant $\lambda=\lambda(T)>0$ and all $t\in [0,T]$.
\end{corollary}

\subsubsection{Construction of a conformal flow}
Note that the functions $u_t^\epsilon$ are uniformly H\"older in $x$ and uniformly Lipschitz in $t$. As a result, we can find a sequence $\epsilon_i\to 0$ as $i\to+\infty$ such that 
$$u_t^{\epsilon_i}\to u_t\mbox{ as }i\to\infty\mbox{ in }C^0\mbox{-sense on bounded time intervals}.$$
From this we can define
$$g_t=u_t^{\frac{4}{n-2}}g$$
and
$$\Sigma_t=\mbox{outermost } g_t\mbox{-minimizing enclosure of }\Sigma_0 \mbox{ in }M.$$
The common exterior barrier above makes this convergence uniform on
$M$. The same area, density, and isoperimetric argument gives uniform
compact support for $\Sigma_t$ on bounded time intervals.
Let $v_t$ denote the function defined in the item (ii) of Definition \ref{Defn: conformal flow} with the above choice of $\Sigma_t$. 

The goal of this subsection is to show the following proposition.
\begin{proposition}\label{Prop: conformal flow}
We have
$$u_t(x)=1+\int_0^tv_s(x)\,\mathrm ds.$$
 Therefore,   $\{(g_t,\Sigma_t)\}_{t\geq 0}$ is a conformal flow with initial data $(g,\Sigma_0)$.
\end{proposition}
Recall from the previous subsection that we have 
$$
u^\epsilon_t(x)=1+\int_0^tv_s^\epsilon(x)\,\mathrm ds.$$
So we just need to establish appropriate convergence from $v_t^{\epsilon_i}$ to $v_t$ as $i\to +\infty$, which is the main purpose of the following lemma.

\begin{lemma}\label{Lem: convergence to limit}
    Except for countably many $t$, the hypersurfaces $\Sigma_t^{\epsilon_i}$ converge to $\Sigma_t$ in the sense of Hausdorff distance. In particular, we have
    $$\lim_{i\to +\infty}v_t^{\epsilon_i}=v_t\mbox{ in }C^0\mbox{-sense}.$$
\end{lemma}

Denote $\Sigma_t=\partial E_t$.

\begin{lemma}\label{Lem: almost minimizing limit}
    Given any constant $T>0$ we can find constants $\beta_T>0$, $\mu_T>0$ and $r_T>0$ such that $\Sigma_t$ with $t\in(0,T]$ are all $(\beta_T,\mu_T,r_T,g)$-almost-minimizing.
\end{lemma}
\begin{proof}
    Take $F\in \mathcal E$ such that $E_t\Delta F\subset B^g_r(x)$ with $r\leq r_T$, where $0<r_T<r_0$ is a constant small enough independent of $t$ such that
    $$\mathcal H^{n-1}_{g_t}(S^g_r(x))\leq \mathcal H^{n-1}_{g}(S^g_r(x))\leq C(g,n)r^{n-1}\mbox{ for all }r\leq r_T.$$
    Write $\Sigma_0=\partial E_0$. Recall that $\Sigma_0$ is $(\beta,\mu,r_0,g)$-almost-minimizing. Since $F\setminus E_0\subset F\setminus E_t$, we know that $(E_0\cup F)\Delta E_0$ is contained in some $g$-geodesic ball with radius $r_0$. Therefore, we have
    \begin{equation}\label{Eq: cut-and-paste}
        \mathcal H^{n-1}_g(\partial E_0\cap B^g_r(x))\leq \mathcal H^{n-1}_g(\partial (E_0\cup F)\cap B^g_r(x))+\beta r^{n-1+\mu}.
    \end{equation}
    From the $g_t$-minimizing property of $\Sigma_t$ outside $\Sigma_0$ we know
    \begin{equation}\label{Eq: gt minimizing}
        \mathcal H^{n-1}_{g_t}(\partial E_t\cap B^g_r(x))\leq \mathcal H^{n-1}_{g_t}(\partial (E_0\cap F)\cap B^g_r(x))
    \end{equation}
    and
    $$\mathcal H^{n-1}_{g_t}(\partial E_t\cap B^g_r(x))\leq \mathcal H^{n-1}_{g_t}(S^g_r(x)).$$
     It follows from Corollary \ref{Cor: holder} that $g_t\geq \lambda g$ for some constant $\lambda=\lambda(T)$ and all $t\in [0,T]$. Therefore, we have
    $$\mathcal H^{n-1}_{g}(\partial E_t\cap B^g_r(x))\leq C(n,T)r^{n-1}.$$
    Since $g_t=u_t^{\frac{4}{n-2}}g$, we can write \eqref{Eq: gt minimizing} as
    $$\int_{\partial E_t\cap B^g_r(x)}u_t^{\frac{2(n-1)}{n-2}}\mathrm d\mathcal H^{n-1}_g\leq \int_{\partial (E_0\cap F)\cap B^g_r(x)}u_t^{\frac{2(n-1)}{n-2}}\mathrm d\mathcal H^{n-1}_g.$$
    Using the H\"older continuity of $u_t$ from Corollary \ref{Cor: holder} we conclude that
    $$\mathcal H^{n-1}_g(\partial E_t\cap B^g_r(x))\leq \mathcal H^{n-1}_g(\partial (E_0\cap F)\cap B_r^g(x))+Cr^{n-1+\alpha},$$
   for some constant $C$ independent of $t$.
    Combined with \eqref{Eq: cut-and-paste} and the strong subadditivity, we obtain
    $$\mathcal H^{n-1}_g(\partial E_t\cap B^g_r(x))\leq \mathcal H^{n-1}_g(\partial F\cap B^g_r(x))+\beta_Tr^{n-1+\mu_T}$$
    for an appropriate choice of constants $\beta_T$ and $\mu_T$ independent of $t$.
\end{proof}

\begin{lemma}\label{Lem: density lower bound}
    By taking $\iota$ small enough, if $\Sigma\in\mathcal C$ is $(\beta_T,\mu_T,r_T,g)$-almost-minimizing, then there is a constant $c=c(n,g,\beta_T,\mu_T,r_T)>0$ such that at any point $x\in \Sigma$ we have
    $$\mathcal H^{n}_g(B^g_\rho(x)\setminus E)\geq c\rho^{n}$$
    and
    $$\mathcal H^{n-1}_g(\Sigma\cap B^g_\rho(x))\geq c\rho^{n-1}$$
    for all $\rho \in (0,\min\{r_T,\iota\})$.
\end{lemma}
\begin{proof}
  Write $\Sigma=\partial E$.  It follows from the proof of Lemma \ref{Lem: sphere lower area} that
  $$\mathcal H^{n}_g(B^g_\rho(x)\setminus E)\geq c\rho^{n}.$$
  Similarly, we have
   $$\mathcal H^{n}_g(B^g_\rho(x)\cap E)\geq c\rho^{n}$$
   for some constant $c=c(n,g,\beta_T,\mu_T,r_T)>0$. By taking $\iota$ small enough we can guarantee $B^g_\rho(x)$ is uniformly equivalent to the Euclidean $\rho$-ball. Using the Sobolev inequality (see \cite[(1.18)]{Giusti}) we can obtain
   $$\mathcal H^{n-1}_g(\Sigma\cap B^g_\rho(x))\geq c\rho^{n-1},$$
   where we abuse the notation $c$ to denote a different constant depending only on $n,\beta_T,\mu_T,r_T$.
   This completes the proof.
\end{proof}

\begin{lemma}\label{Lem: almost minimizing approximation}
    Given any constant $T>0$, we can find constants $\beta_T>0$, $\mu_T>0$ and $r_T>0$ such that $\Sigma^\epsilon_t$ with $t\in(0,T]$ are all $(\beta_T,\mu_T,r_T,g)$-almost-minimizing.
\end{lemma}
\begin{proof}
   Apply the proof of Lemma \ref{Lem: almost minimizing limit} to
$g_{k\epsilon}^\epsilon$, where $k=\lfloor t/\epsilon\rfloor$,
using exact grid-time minimization and Corollary \ref{Cor: holder}.
\end{proof}

\begin{corollary}\label{Cor: discrete compactness}
   Given any sequence of hypersurfaces $\{\Sigma_{t_i}^{\epsilon_i}\}$ with $t_i\leq T$, $\Sigma_{t_i}^{\epsilon_i}$ converge to some $(\beta_T,\mu_T,r_T,g)$-almost-minimizing hypersurface $\hat\Sigma$ in the sense of currents and varifolds up to a subsequence.
\end{corollary}
\begin{proof}
   By Lemma \ref{Lem: almost minimizing approximation}, the uniform
area bound, and compact support, the compactness theorem for
almost-minimizing boundaries
gives a subsequence converging as currents and varifolds to an
almost-minimizing boundary. Lemma \ref{Lem: density lower bound}
also gives Hausdorff convergence.
\end{proof}

In the following, we use $\hat\Sigma_t$ to denote a limit of $\Sigma_t^{\epsilon_i}$ up to a subsequence. We first establish the following area estimate.

\begin{lemma}
   With $t$ fixed we have
    $$\lim_{\epsilon\to 0}\mathcal H^{n-1}_{g_t^\epsilon}(\Sigma_t^\epsilon)=\mathcal H^{n-1}_g(\Sigma_0).$$
    
\end{lemma}
\begin{proof}
    Let $t\leq T$. We are going to show that
    $$\mathcal H^{n-1}_{g^\epsilon_t}(\Sigma_t^\epsilon)=\mathcal H^{n-1}_g(\Sigma_0)+o(1)\mbox{ as }\epsilon\to 0.$$
    The argument follows from \cite[Section 5]{Bray} with slight modifications.  First we point out that $\Sigma_t^\epsilon$ coincides with $\Sigma_{k\epsilon}^\epsilon$ if $k=\lfloor\frac{t}{\epsilon}\rfloor$. So we just need to focus on those $\Sigma_{t}^\epsilon$ where $t$ is a multiple of $\epsilon$. The defining identity
$u_{k\epsilon}^\epsilon=u_{(k-1)\epsilon}^\epsilon
+\epsilon v_{(k-1)\epsilon}^\epsilon$
and minimization at the preceding grid time give
\[
\begin{split}
&\mathcal H^{n-1}_{g_{(k-1)\epsilon}^\epsilon}
  (\Sigma_{(k-1)\epsilon}^\epsilon)
-\mathcal H^{n-1}_{g_{k\epsilon}^\epsilon}
  (\Sigma_{k\epsilon}^\epsilon)\\
&\qquad\leq
\int_{\Sigma_{k\epsilon}^\epsilon}
\left[
1-\left(
1+\epsilon\frac{v_{(k-1)\epsilon}^\epsilon}
{u_{(k-1)\epsilon}^\epsilon}
\right)^{\frac{2(n-1)}{n-2}}
\right]
\,\mathrm d\mathcal H^{n-1}_{g_{(k-1)\epsilon}^\epsilon}\\
&\qquad\leq
-\frac{2(n-1)}{n-2}\epsilon
\int_{\Sigma_{k\epsilon}^\epsilon}
\frac{v_{(k-1)\epsilon}^\epsilon}
{u_{(k-1)\epsilon}^\epsilon}
\,\mathrm d\mathcal H^{n-1}_{g_{(k-1)\epsilon}^\epsilon}.
\end{split}
\]
    where we have used the Bernoulli inequality $$(1+x)^\alpha\geq 1+\alpha x$$ for $\alpha>1$ and $x>-1$ in the second line. Recall that we have $u^\epsilon_{(k-1)\epsilon}\geq 4^{-T}$. Grid-time minimization makes these areas non-increasing, and metric
comparison gives
\[
\mathcal H^{n-1}_{g_{(k-1)\epsilon}^\epsilon}
(\Sigma_{k\epsilon}^\epsilon)
\leq C(T)\mathcal H^{n-1}_{g_{k\epsilon}^\epsilon}
(\Sigma_{k\epsilon}^\epsilon)
\leq C(T)\mathcal H_g^{n-1}(\Sigma_0).
\]
Then we arrive at the estimate
\begin{equation}\label{Eq: step difference}
\begin{split}
&\mathcal H^{n-1}_{g_{(k-1)\epsilon}^\epsilon}
(\Sigma_{(k-1)\epsilon}^\epsilon)
-\mathcal H^{n-1}_{g_{k\epsilon}^\epsilon}
(\Sigma_{k\epsilon}^\epsilon)\\
&\qquad\leq
\Lambda(T)\mathcal H_g^{n-1}(\Sigma_0)\epsilon
\left\|v_{(k-1)\epsilon}^\epsilon
\right\|_{L^\infty(\Sigma_{k\epsilon}^\epsilon)}.
\end{split}
\end{equation}
   Recall from Corollary \ref{Cor: holder} that we have the comparison $\lambda g\leq g_t^\epsilon\leq g$. In particular, the $g$-areas of $\Sigma_T^\epsilon$ are uniformly bounded independent of $\epsilon$, and so the $g$-volume of the region enclosed by $\Sigma^\epsilon_T$ are uniformly bounded by some constant $V_0$ independent of $\epsilon$ as a consequence of the isoperimetric inequality (since $g$ is asymptotically flat). We remind that we are considering positive integers $k$ such that $k\epsilon\leq T$. Denote
   $$\mathcal K_1=\{k:\mbox{the } g\mbox{-volume between }\Sigma_{(k-1)\epsilon}^\epsilon\mbox{ and }\Sigma_{k\epsilon}^\epsilon\mbox{ is greater than }V_0\sqrt\epsilon\}.$$
   Clearly, we have the estimate
   $$\#\mathcal K_1\leq \left\lfloor\frac{1}{\sqrt\epsilon}\right\rfloor+1.$$
   It follows from Lemma \ref{Lem: density lower bound} that
   $$d_{\mathcal H}(\Sigma_{(k-1)\epsilon}^\epsilon,\Sigma_{k\epsilon}^\epsilon)\leq \omega(V_0\sqrt\epsilon)\mbox{ with }k\notin\mathcal K_1$$
   for some function $\omega$ satisfying $\omega(\tau)\to 0$ as $\tau\to 0$.

   Summing over \eqref{Eq: step difference}, using $|v_t^\epsilon|\leq 1$ and Proposition \ref{Prop: v_t^epsilon}, we have
   $$\mathcal H^{n-1}_g(\Sigma_0)-\mathcal H^{n-1}_{g_{k\epsilon}^\epsilon}(\Sigma_{k\epsilon}^\epsilon)\leq C\epsilon\cdot\#\mathcal K_1+C[\omega(V_0\sqrt\epsilon)]^\alpha$$
   for some universal constants $C$ and $\alpha$ independent of $k$. This implies that for $t\leq T$ we have
   $$\mathcal H^{n-1}_{g^\epsilon_t}(\Sigma_t^\epsilon)=\mathcal H^{n-1}_g(\Sigma_0)+o(1)\mbox{ as }\epsilon\to 0.$$
   This completes the proof.
\end{proof}

\begin{corollary}\label{Cor: constant area}
    We have $$\mathcal H^{n-1}_{g_t}(\hat\Sigma_t)=\mathcal H^{n-1}_{g_t}(\Sigma_t)=\mathcal H^{n-1}_g(\Sigma_0)$$ 
    The area constant will be denoted by $\mathscr A$.
\end{corollary}
\begin{proof}
Recall that $\Sigma_t^{\epsilon_i}$ converge to $\hat\Sigma_t$
in the sense of varifolds up to a subsequence. Combined with the
$C^0$-convergence of $u_t^{\epsilon_i}$ to $u_t$, we conclude
\[
\mathcal H^{n-1}_{g_t}(\hat\Sigma_t)
=\lim_{i\to\infty}
\mathcal H^{n-1}_{g_t^{\epsilon_i}}(\Sigma_t^{\epsilon_i})
=\mathcal H^{n-1}_g(\Sigma_0).
\]
Since $\Sigma_t$ is the outermost $g_t$-minimizing enclosure
of $\Sigma_0$, we have
\[
\mathcal H^{n-1}_{g_t}(\Sigma_t)
\leq\mathcal H^{n-1}_{g_t}(\Sigma_0)
=\mathcal H^{n-1}_g(\Sigma_0).
\]
On the other hand, put $k=\lfloor t/\epsilon_i\rfloor$.
Since $\|u_t^{\epsilon_i}-u_{k\epsilon_i}^{\epsilon_i}\|_\infty
\leq\epsilon_i$, grid-time minimization gives
\[
\mathcal H^{n-1}_{g_t^{\epsilon_i}}(\Sigma_t^{\epsilon_i})
\leq(1+C_T\epsilon_i)
\mathcal H^{n-1}_{g_t^{\epsilon_i}}(\Sigma_t).
\]
Taking the limit gives the reverse inequality. This yields
$\mathcal H^{n-1}_{g_t}(\Sigma_t)=\mathcal H^{n-1}_g(\Sigma_0)$.
This completes the proof.
\end{proof}

\begin{lemma}\label{Lem: monotonicity}
    $E_t$ is monotone non-increasing as $t$ increases.
\end{lemma}

\begin{proof}
Since $\Sigma_t$ is the outermost $g_t$-minimizing enclosure
of $\Sigma_0$ and $\hat\Sigma_t$ has the same $g_t$-area as
$\Sigma_t$, we conclude that $\hat\Sigma_t$ is enclosed by
$\Sigma_t$. We show that, for any $t<s$, the hypersurface
$\hat\Sigma_s$ encloses $\Sigma_t$.

Otherwise, we can find a compact portion $\Gamma$ of $\Sigma_t$
lying outside $\hat\Sigma_s$ with
$\mathcal H_g^{n-1}(\Gamma)>0$. We claim that there is a constant
$c_\Gamma>0$ such that
\[
u_s\leq u_t-c_\Gamma
\quad\text{on }\Gamma.
\]
Otherwise,
$\min_\Gamma(u_t^{\epsilon_i}-u_s^{\epsilon_i})\rightarrow0$.
Recall that
\[
u_t^{\epsilon_i}-u_s^{\epsilon_i}
=-\int_t^s v_\tau^{\epsilon_i}\,\mathrm d\tau,
\]
and that $v_\tau^{\epsilon_i}$ is monotone increasing in $\tau$,
by the maximum principle on the later exterior and discrete
nesting. Consequently,
\[
\min_\Gamma(-v_s^{\epsilon_i})\rightarrow0.
\]
Recall that $\Sigma_s^{\epsilon_i}$ converge to $\hat\Sigma_s$
as varifolds up to a subsequence. Lemma
\ref{Lem: density lower bound} also gives Hausdorff convergence.
In particular, $v_s^{\epsilon_i}$ converge to the function
$\hat v_s$ determined by
\[
\begin{cases}
\Delta_g\hat v_s=0
    &\text{outside }\hat\Sigma_s,\\
\hat v_s=0
    &\text{inside and along }\hat\Sigma_s,\\
\hat v_s\rightarrow-e^{-s}
    &\text{as }x\rightarrow\infty.
\end{cases}
\]
The common exterior barrier preserves the limit at infinity.
A minimizing exterior has no bounded component, since deleting
one decreases perimeter; hence it is connected. However, the
preceding discussion would give $\hat v_s=0$ at a point of
$\Gamma$, contradicting the strong maximum principle.

The claimed strict decrease on $\Gamma$, together with
$u_s\leq u_t$ elsewhere, gives
\[
\mathcal H_{g_s}^{n-1}(\Sigma_t)
<
\mathcal H_{g_t}^{n-1}(\Sigma_t)
=
\mathcal H_g^{n-1}(\Sigma_0),
\]
contradicting the minimum area already established at time $s$.
Thus $\hat\Sigma_s$ encloses $\Sigma_t$. Since $\Sigma_s$
encloses $\hat\Sigma_s$, it follows that $E_s\subset E_t$.
\end{proof}

\begin{corollary}\label{Cor: convergence}
Let $t\leq T$. As $t\uparrow s$, $\Sigma_t$ converge to some
$(\beta_T,\mu_T,r_T,g)$-almost-minimizing hypersurface
$\Sigma_s^-\in\mathcal C$ in the sense of both currents and
varifolds. Similarly, $\Sigma_t$ converge to some
$(\beta_T,\mu_T,r_T,g)$-almost-minimizing hypersurface
$\Sigma_s^+\in\mathcal C$ as $t\downarrow s$ in the sense of
both currents and varifolds. Moreover, $\Sigma_s^+=\Sigma_s$,
and
\[
\mathcal H_{g_s}^{n-1}(\Sigma_s^-)
=
\mathcal H_{g_s}^{n-1}(\Sigma_s^+)
=
\mathscr A.
\]
\end{corollary}

\begin{proof}
Since $\Sigma_t$ is the outermost $g_t$-minimizing enclosure
of $\Sigma_0$, we have
\[
\mathcal H_{g_t}^{n-1}(\Sigma_t)
\leq
\mathcal H_{g_t}^{n-1}(\Sigma_0)
=
\mathcal H_g^{n-1}(\Sigma_0).
\]
Recall that $g_t\geq\lambda g$ by Corollary \ref{Cor: holder}.
Thus the hypersurfaces $\Sigma_t$, with $t\leq T$, have uniformly
bounded $g$-area. Together with Lemma \ref{Lem: monotonicity},
this gives convergence as currents to $\Sigma_s^-$ or
$\Sigma_s^+$ as $t\uparrow s$ or $t\downarrow s$, respectively.
The compactness argument in Corollary
\ref{Cor: discrete compactness} preserves almost-minimality
and gives varifold convergence. The density estimates also
give Hausdorff convergence.

Their $g_s$-areas equal $\mathscr A$ by uniform convergence
of the metrics. The right limit is a minimizing enclosure,
so outermostness and nesting give $\Sigma_s^+=\Sigma_s$.
\end{proof}

\begin{lemma}\label{Lem: generic coincidence}
Except for countably many $s>0$, we have
$\Sigma_s^-=\Sigma_s^+$.
\end{lemma}

\begin{proof}
It suffices to prove the assertion on each bounded time interval.
If $\Sigma_s^-\neq\Sigma_s^+$, Lemma
\ref{Lem: density lower bound} implies that the region
$\Omega_s$ between $\Sigma_s^-$ and $\Sigma_s^+$ has positive
$g$-volume. By nesting, these regions are pairwise disjoint.
Since they all lie in a fixed bounded region on each bounded
time interval, there are at most countably many such $s$.
\end{proof}

\begin{lemma}\label{Lem: lie between}
$\hat\Sigma_s$ lies between $\Sigma_s^-$ and $\Sigma_s^+$
for all $s>0$.
\end{lemma}

\begin{proof}
The proof of Lemma \ref{Lem: monotonicity} shows that
$\hat\Sigma_s$ is enclosed by $\Sigma_s=\Sigma_s^+$ and
encloses every $\Sigma_t$ with $t<s$. Letting $t\uparrow s$
proves the assertion.
\end{proof}

\begin{proof}[Proof of Lemma \ref{Lem: convergence to limit}]
By Lemmas \ref{Lem: generic coincidence} and
\ref{Lem: lie between}, the hypersurfaces
$\Sigma_t^{\epsilon_i}$ converge to $\Sigma_t$ in the sense
of currents and varifolds for all $t$ outside a countable set.
Lemma \ref{Lem: density lower bound} also gives Hausdorff
convergence. The boundary estimates and the common exterior
barrier then give
\[
v_t^{\epsilon_i}\rightarrow v_t
\quad\text{in }C^0(M).
\]
\end{proof}

\begin{proof}[Proof of Proposition \ref{Prop: conformal flow}]
Recall that
\[
u_t^{\epsilon_i}
=
1+\int_0^t v_s^{\epsilon_i}\,\mathrm ds.
\]
Let $\epsilon_i\rightarrow0$. Since
$|v_s^{\epsilon_i}|\leq1$ and convergence fails only at
countably many times, dominated convergence gives
\[
u_t=1+\int_0^t v_s\,\mathrm ds.
\]
This completes the proof.
\end{proof}

\subsection{Regularity consequences of the almost-minimizing property}

In this subsection, we derive the regularity of  the conformal
flow constructed above. The main tool is the regularity theorem below for
almost-minimizing boundaries due to Almgren \cite{Almgren1976}, Bombieri \cite{Bombieri1982}, and Tamanini \cite{Tamanini1984}.

\begin{theorem}[Almgren--Bombieri--Tamanini]\label{Thm: ABT regularity}
Let \(\Sigma=\partial E\in\mathcal C\) be a
\((\beta,\mu,r_0,g)\)-almost-minimizing hypersurface. Then there is a regular-singular decomposition
\[
\Sigma=\mathcal R(\Sigma)\sqcup\mathcal S(\Sigma)
\]
such that
\begin{itemize}
    \item $\mathcal R(\Sigma)$ is a relatively open subset of $\Sigma$, and \(\mathcal R(\Sigma)\) is a \(C^{1,\theta}\)-embedded hypersurface for some \(\theta=\theta(n,\mu)\in(0,1)\);
    \item $\mathcal S(\Sigma)$ is a closed set satisfying the Hausdorff dimension estimate
    \[
    \dim_{\mathcal H}\mathcal S(\Sigma)\leq n-8 .
\]
\end{itemize}
Moreover, 
we have the following \(\varepsilon\)-regularity result: there exist
constants \(r_*>0\), \(\varepsilon_0>0\), and $\gamma\in (0,1)$, depending only on \(\beta,\mu,r_0,g\),
such that if we have 
\[
  \Theta^{n-1}(\Sigma,x,r):=  \frac{\mathcal H_g^{n-1}(\Sigma\cap B_r^g(x))}
         {\omega_{n-1}r^{n-1}}
    \leq 1+\varepsilon_0
\]
for \(x\in\Sigma\) and \(r\leq r_*\),
then \(x\in\mathcal R(\Sigma)\). Moreover, after choosing suitable
geodesic coordinates, \(\Sigma\) is a \(C^{1,\theta}\)-graph in $B^g_{\gamma r}(x)$
with uniform $C^{1,\theta}$-estimates depending only on \(\beta,\mu,r_0,g\).
\end{theorem}

Applying Theorem \ref{Thm: ABT regularity} to the conformal flow gives the following.

\begin{corollary}\label{Cor: regularity Sigma t}
Let \(\{(g_t,\Sigma_t)\}_{t\geq0}\) be the conformal flow constructed in Theorem~\ref{Prop: existence}. For every \(T>0\), there exists
\[
    \theta=\theta(n,T,g,\beta,\mu,r_0)\in(0,1)
\]
such that for every \(t\in(0,T]\), we have
\[
    \Sigma_t=\mathcal R_t\sqcup\mathcal S_t,
\]
where \(\mathcal R_t\) is a relatively open \(C^{1,\theta}\)-embedded hypersurface and
\[
    \dim_{\mathcal H}\mathcal S_t\leq n-8 .
\]
The same conclusion holds for the approximation hypersurfaces \(\Sigma_t^\epsilon\), and for the one-sided limits \(\Sigma_s^-\) and \(\Sigma_s^+\). 
\end{corollary}

\begin{proof}
By Lemma \ref{Lem: almost minimizing limit}, Corollary \ref{Cor: convergence} and Lemma \ref{Lem: almost minimizing approximation}, the hypersurfaces mentioned above are uniformly almost-minimizing with respect to the fixed background metric \(g\) on every compact time interval. The conclusion follows from Theorem \ref{Thm: ABT regularity}.
\end{proof}

\subsection{Local smoothness of the regular part}

We first record a one-sided regularity improvement near regular pieces of the
flow hypersurfaces. The key input is the half-space rigidity for
stationary cones.

Fix $T>0$, and we always restrict our attention to $\Sigma_t$ with $0\leq t\leq T$. Recall from Lemma \ref{Lem: almost minimizing limit} that $\Sigma_t$ are all $(\beta_T,\mu_T,r_T,g)$-almost-minimizing. Let $r_*$ and $\gamma$ denote the constants in the $\varepsilon$-regularity from Theorem \ref{Thm: ABT regularity} depending on $\beta_T,\mu_T,r_T,g$.

\begin{lemma}\label{Lem: one-sided density}
For every \(\varepsilon>0\), there exist 
\(\delta\in(0,1)\) and \(\varepsilon_1>0\), depending only on $\beta_T,\mu_T,r_T,g$
and \(\varepsilon\), such that the following holds: if we have
$$\Theta^{n-1}(\Sigma_t,p,r)\leq 1+\varepsilon_1$$
for some $t\in [0,T]$, $p\in \Sigma_t$ and $r\leq r_*$,
 then for every \(s\in[0,T]\) and every point
\[
    q\in \Sigma_s\cap B_{\gamma r/2}^g(p)
    \mbox{ with }
    \dist_g(q,\Sigma_{t})\leq \delta r,
\]
one has
$$\Theta^{n-1}(\Sigma_s,q,\delta r)\leq 1+\varepsilon.$$
In particular, if \(\varepsilon<\varepsilon_0\), where \(\varepsilon_0\) is
the threshold in Theorem \ref{Thm: ABT regularity}, then \(\Sigma_s\) is a
\(C^{1,\theta}\)-graph 
in $B^g_{\gamma\delta r}(q)$ with uniform estimates.
\end{lemma}

\begin{proof}
We will argue by contradiction. Suppose that the conclusion fails for some
\(\varepsilon>0\). Then there exist sequences
\[
    \varepsilon_{1,i}\downarrow0\mbox{ and } \delta_i\downarrow0,
\]
such that 
\begin{itemize}
    \item there are hypersurfaces $\Sigma_{t_i}$ with \(t_{i}\in[0,T]\), points $p_i\in\Sigma_{t_i}$,
 and $0<r_i<r_*$ satisfying 
 $$\Theta^{n-1}(\Sigma_{t_i},p_i,r_i)\leq 1+\varepsilon_{1,i};$$
 \item there are hypersurfaces $\Sigma_{s_i}$ with $s_i\in [0,T]$ and points $q_i\in \Sigma_{s_i}$ with $$\dist_g(q_i,p_i)\leq \gamma r_i/2\mbox{ and }\dist_g(q_i,\Sigma_{t_i})\leq \delta_i r_i$$ satisfying
 $$\Theta^{n-1}(\Sigma_{s_i},q_i,\delta_ir_i)\geq 1+\varepsilon.$$
\end{itemize}
By Theorem \ref{Thm: ABT regularity}, $\Sigma_{t_i}\cap B^g_{\gamma r_i}(p_i)$ is a $C^{1,\theta}$-graph with uniform estimates. Up to a subsequence, we may assume $r_i\to r_\infty\in [0,r_*]$ as $i\to \infty$. 

If $r_\infty>0$, we directly investigate the limit of $\Sigma_{t_i}$. Up to a subsequence, $\Sigma_{t_i}$ converge to a hypersurface $\Sigma_{t_\infty}$ and $p_i$ converge to a point $p_\infty$, where $\Sigma_{t_\infty}\cap B^g_{\gamma r_\infty}(p_\infty)$ is a $C^{1,\theta}$-graph with uniform estimates.
Up to a subsequence, $\Sigma_{s_i}$ converge to a $(\beta_T,\mu_T,r_T,g)$-almost-minimizing hypersurface $S_\infty$ and $q_i$ converge to a point $q_\infty\in \Sigma_{t_\infty}\cap B^g_{\gamma r_\infty/2}(p_\infty)$. In particular, we have 
$$\Theta^{n-1}(\Sigma_{t_\infty},q_\infty)=1.$$
On the other hand, by the
almost-minimizing monotonicity formula, see for instance
\cite[Proposition~2.1]{DeLellisSpadaroSpolaorTangentCones}, applied at the
scales \(\delta_i r_i\), we obtain
\[
    \Theta^{n-1}(S_\infty,q_\infty)\geq 1+\varepsilon .
\]
Blowing up $\Sigma_{t_\infty}$ and $S_\infty$ at $q_\infty$, we obtain a minimizing cone $C$ lying on one side of a hyperplane with $\Theta^{n-1}(C,0)\geq 1+\varepsilon$. However, by the half-space rigidity for
stationary cones, \(C\) must be the hyperplane, and this leads to a contradiction.
If $r_\infty=0$, then we consider $\Sigma_{t_i}$ and $\Sigma_{s_i}$ as hypersurfaces in $(M,r_i^{-2}g)$ and investigate their limits. The argument works as before.

This proves the density estimate. The final graphical conclusion follows from
the \(\varepsilon\)-regularity part of Theorem \ref{Thm: ABT regularity}.
\end{proof}

\begin{proposition}\label{Prop: local regularity near regular points}
Let \(t\in[0,T]\) and 
\[
    p\in\mathcal R(\Sigma_{t}) .
\]
Then there exists \(\rho>0\) such that, for every \(s\in[0,T]\),
\[
    \Sigma_s\cap B_\rho^g(p)\subset \mathcal R(\Sigma_s).
\]
Moreover, hypersurfaces \(\Sigma_s\cap B_\rho^g(p)\) are
\(C^{1,\theta}\)-graphs with uniform estimates.
\end{proposition}

\begin{proof}
Fix \(\varepsilon<\varepsilon_0\), where \(\varepsilon_0\) is the
\(\varepsilon\)-regularity threshold in Theorem \ref{Thm: ABT regularity}.
Let \(\delta\) and \(\varepsilon_1\) be given by Lemma
\ref{Lem: one-sided density}. Since \(p\in\mathcal R(\Sigma_{t})\), we can choose
\(r>0\) sufficiently small so that
\[
 \Theta^{n-1}(\Sigma_{t},p,r)=   \frac{\mathcal H_g^{n-1}(\Sigma_{t}\cap B_r^g(p))}
         {\omega_{n-1}r^{n-1}}
    \leq 1+\varepsilon_1 .
\]
Set
\[
    \rho=\frac{\gamma\delta r}{4}.
\]
If \(q\in\Sigma_s\cap B_\rho^g(p)\), then we have
\[
    q\in B_{\gamma r/2}^g(p)\mbox{ and }
\dist_g(q,\Sigma_{t})\leq d_g(q,p)<\delta r .
\]
Hence Lemma \ref{Lem: one-sided density}, applied with reference hypersurface
\(\Sigma_{t}\), gives
\[
  \Theta^{n-1}(\Sigma_s,q,\delta r)=  \frac{\mathcal H_g^{n-1}(\Sigma_s\cap B_{\delta r}^g(q))}
         {\omega_{n-1}(\delta r)^{n-1}}
    \leq 1+\varepsilon .
\]
The \(\varepsilon\)-regularity theorem from Theorem \ref{Thm: ABT regularity} therefore implies
\(q\in\mathcal R(\Sigma_s)\). Since \(q\) was arbitrary, we obtain
\[
    \Sigma_s\cap B_\rho^g(p)\subset \mathcal R(\Sigma_s)
\]
for every \(s\in[0,T]\). Note that $\Sigma_s\cap B^g_\rho(p)\subset \Sigma_s\cap B_{\gamma\delta r}^g(q)$ for $q\in \Sigma_s\cap B^g_\rho(p)$,  so $\Sigma_s\cap B^g_\rho(p)$ is $C^{1,\theta}$-graph with uniform 
estimates.
\end{proof} 

To clarify, we use the following convention for mean curvature throughout the paper.

\vspace{2mm}\noindent {\bf Convention III.} For a two-sided $C^2$-hypersurface $\Gamma$ with unit normal $\nu$, the mean curvature of $\Gamma$ with respect to $\nu$ is given by
$$H_{\Gamma}^g=\ddiv_\Gamma^g\nu.$$

In the following, we assume $\Sigma_0=\partial M$, $\mathcal R(\partial M)$ is minimal.

\begin{proposition}\label{Prop: smoothness regular part}

For every \(t>0\), the regular part \(\mathcal R(\Sigma_t)\) is a smooth
\(g_t\)-minimal hypersurface, and \(u_t\) is smooth up to
\(\mathcal R(\Sigma_t)\) from the exterior side. 
\end{proposition}

\begin{proof}
Fix \(p\in\mathcal R(\Sigma_t)\). By Proposition
\ref{Prop: local regularity near regular points}, 
we can find a neighborhood $V_p$ of $p$ such that every flow hypersurface restricted to $V_p$ is a \(C^{1,\theta}\)-graph with uniform estimates.

We first construct a one-sided \(C^{1,\theta}\) extension of the conformal
factor near \(p\). For those \(s\in[0,t]\) for which \(\Sigma_s\) meets the
chosen neighbourhood, the graphical regularity and the boundary estimates for
the harmonic functions \(v_s\) imply uniform one-sided \(C^{1,\theta}\) estimates up to
\(\Sigma_s\) from the exterior side. Extending \(v_s\) slightly across
\(\Sigma_s\), and increasing it on the other side if necessary, we obtain
local functions \(\tilde v_s\) which agree with \(v_s\) on the exterior side
and satisfy
\[
    \tilde v_s\geq v_s
\]
in $V_p$. For the remaining values of $s$, interior estimates apply on a smaller
neighborhood, or $v_s$ vanishes there. All bounds are uniform for
$s\in[0,t]$. The required background coefficient bounds follow from
smoothness away from $\Sigma_0$, or smoothness up to its regular part
when $p\in\Sigma_0$, by Proposition
\ref{Prop: local regularity near regular points}. Define
\[
    \tilde u_t=1+\int_0^t \tilde v_s\,ds .
\]
Then, after possibly shrinking the neighborhood $V_p$, we can guarantee
\[
    \tilde u_t\in C^{1,\theta},\,
    \tilde u_t=u_t \mbox{ on the exterior side of  }\Sigma_t,\mbox{ and }
    \tilde u_t\geq u_t
\]
locally. Set
\[
    \tilde g_t=\tilde u_t^{\frac4{n-2}}g .
\]

We first rule out contact of $\Sigma_t$ with \(\Sigma_0\) at $p$. Otherwise, we choose local coordinates near
\(p\) in which the common tangent plane is \(\{x_n=0\}\), the exterior side is
\(\{x_n>0\}\), and both \(\mathcal R(\Sigma_t)\) and
\(\mathcal R(\Sigma_0)\) are written as graphs. On \(\Sigma_0\), we have \(u_t=1\), and the Hopf lemma gives
\[
    \partial_\nu u_t(p)
    =
    \int_0^t \partial_\nu v_s(p)\,ds
    <0,
\]
where \(\nu\) points towards the exterior region. From the assumption, $\mathcal R(\Sigma_0)$ is smooth and we have
\(H_{\Sigma_0}^g(p)=0\), the conformal transformation formula gives
\[
    H_{\Sigma_0}^{g_t}(p)
    =
    u_t^{-\frac{2}{n-2}}
    \left(
        H_{\Sigma_0}^g(p)
        +\frac{2(n-1)}{n-2}u_t^{-1}\partial_\nu u_t
    \right)
    <0 .
\]
In particular, we can construct a $g_t$-mean-concave foliation around $p$ and a replacement argument yields a contradiction to the fact that $\Sigma_t$ is $g_t$-outer-minimizing.

After shrinking the neighborhood away from $\Sigma_0$,
the background metric is smooth there. We now show that \(\mathcal R(\Sigma_t)\) is stationary with respect to
\(\tilde g_t\) near \(p\). Variations to the exterior side cannot decrease
area by the outer-minimizing property of \(\Sigma_t\). On the other hand, if
an interior variation decreased the \(\tilde g_t\)-area, then, since
\(\tilde u_t=u_t\) along \(\Sigma_t\) and \(\tilde u_t\geq  u_t\) on the
interior side, the same localized variation would give a strict decrease of
the \(g_t\)-area. This contradicts the local minimizing property of
\(\Sigma_t\). Hence \(\mathcal R(\Sigma_t)\) is stationary with
respect to \(\tilde g_t\) near \(p\), and therefore satisfies the
\(\tilde g_t\)-minimal hypersurface equation in weak form.

Since \(\tilde u_t\in C^{1,\theta}\), the coefficients of the minimal graph
equation are \(C^{0,\theta}\). The \(C^{1,\theta}\)-graph
\(\mathcal R(\Sigma_t)\) is therefore \(C^{2,\theta}\) by elliptic regularity.
Once 
$\mathcal R(\Sigma_t)$ is \(C^{k,\theta}\), the harmonic functions \(v_t\) enjoy $C^{k,\theta}$
Schauder estimates up to $\mathcal R(\Sigma_t)$ from the exterior side. After repeating the above construction, the
regularity of the conformal factor $\tilde u_t$ improves. 
The bootstrap process finally gives the smoothness of
\(\mathcal R(\Sigma_t)\) and of \(u_t\) up to \(\mathcal R(\Sigma_t)\) from
the exterior side.
\end{proof}

\begin{corollary}\label{Cor: smooth regular decomposition}
For every \(t>0\), \(\Sigma_t\) has a decomposition
\[
    \Sigma_t=\mathcal R_t\cup\mathcal S_t,
\]
where \(\mathcal R_t\) is a smooth \(g_t\)-minimal hypersurface and
\[
    \dim_{\mathcal H}\mathcal S_t\leq n-8 .
\]
\end{corollary}

\begin{proof}
The \(C^{1,\theta}\) regular-singular decomposition follows from Corollary
\ref{Cor: regularity Sigma t}.  The smoothness and \(g_t\)-minimality of the
regular part follow from Proposition \ref{Prop: smoothness regular part}.
\end{proof}
\begin{corollary}\label{Cor: E_t AF}
    Let $\Sigma_t=\partial E_t$. Then $(\overline E_t,g_t)$ is a smooth complete asymptotically flat manifold with compact almost-minimizing Caccioppoli boundary, and $\mathcal R(\Sigma_t)$ is minimal. If $R_g\geq 0$, then we have $R_{g_t}\geq 0$ in $E_t$.
\end{corollary}

\subsection{Excluding singular contact}
\label{subsec: excluding singular contact}

In this subsection, we prove a separation result, which will be used later in two
places. The first is the separation of different time slices in the conformal
flow, and the second is the perturbed minimizing problem used in the smoothing
argument.

We establish the following general local separation criterion.

\begin{proposition}
\label{prop: singular contact exclusion}
Fix a H\"older metric $g_0$. Let $\Sigma=\partial E$ and
$\Sigma'=\partial E'$ be compact $g_0$-almost-minimizing
boundaries with $E'\subset E$. Suppose the following:
\begin{itemize}
\item[(i)] There are continuous metrics $g$ and $g'$, conformal to
$g_0$ and $C^{2,\alpha}$ up to $\mathcal R(\Sigma)$ and
$\mathcal R(\Sigma')$ from $E$ and $E'$, respectively, such that
\[
g'=u^{\frac4{n-2}}g
\]
for some positive function $u$.
\item[(ii)] We have $H_\Sigma^g=0$ on $\mathcal R(\Sigma)$ and
$H_{\Sigma'}^{g'}=h$ on $\mathcal R(\Sigma')$, with respect to
the unit normals pointing into $E$ and $E'$, respectively.
Here $h$ is smooth on an ambient neighborhood of $\Sigma$ and
satisfies $|h(x)|\leq\Lambda\dist_{g_0}(x,\Sigma)$ for some
constant $\Lambda>0$ near $\Sigma$.
\item[(iii)] At any contact point $x_0\in\Sigma\cap\Sigma'$, for
every sequence $r_i\downarrow0$, after passing to a subsequence,
the scaled boundaries in local coordinates
\[
\Sigma_i=r_i^{-1}(\Sigma-x_0),
\qquad
\Sigma_i'=r_i^{-1}(\Sigma'-x_0)
\]
converge to the same $g(x_0)$-minimizing cone $\mathfrak C$
locally as multiplicity-one varifolds and in the Hausdorff sense.
The convergence is locally $C^{2,\alpha}$-graphical on
$\mathcal R(\mathfrak C)$. Moreover, near
$\mathcal R(\mathfrak C)$, the rescaled metrics converge locally
in $C^{2,\alpha}$, in fixed ambient coordinates and from their
respective exterior sides, to $g(x_0)$ and $g'(x_0)$.
\item[(iv)] For every compact $K\Subset\mathcal R(\mathfrak C)$,
there are constants $c_K,C_K>0$ such that, for every $p\in K$,
we can choose $r_p>0$ with
$B_{\mathfrak C}(p,2r_p)\Subset\mathcal R(\mathfrak C)$ and,
for sufficiently large $i$, numbers $\lambda_{i,p,r_p}\geq0$
such that
\[
F_i=\frac{2(n-1)}{n-2}u_i^{-1}\partial_{\nu_i'}^{g_i}u_i
\]
satisfies
\[
-C_K\lambda_{i,p,r_p}\leq F_i\leq0
\quad\text{on }B_i'(p,2r_p),
\qquad
F_i\leq-c_K\lambda_{i,p,r_p}
\quad\text{on }B_i'(p,r_p).
\]
Here
\[
u_i=u\circ\eta_i,\qquad
g_i=r_i^{-2}\eta_i^*g,\qquad
\eta_i(y)=x_0+r_i y,
\]
$\nu_i'$ is the $g_i$-unit normal of $\Sigma_i'$ pointing into
the rescaled exterior region $E_i'=r_i^{-1}(E'-x_0)$, and
$B_i'(p,r)$ denotes the graph over $B_{\mathfrak C}(p,r)$ in
$\Sigma_i'$.
\item[(v)] We have
$\Sigma\cap\Sigma'\subset\mathcal S(\Sigma)\cap\mathcal S(\Sigma')$.
\end{itemize}
Then we have $\Sigma\cap\Sigma'=\emptyset$.
\end{proposition}

We first prove a local comparison estimate.
\begin{lemma}
\label{lem: compact lower bound fixed quantile}
Fix a contact point $x_0$ and a blow-up sequence as in Item (iii),
and let $\mathfrak C$ be the resulting common area-minimizing cone.
Let
\[
K,L\Subset\mathcal R(\mathfrak C)\cap B_1,
\qquad L\neq\emptyset.
\]
In the graphical coordinates furnished by Item (iii), let
$v_i\geq0$ denote the difference of the graph functions of
$\Sigma_i'$ and $\Sigma_i$ over $\mathfrak C$.
Assume that Item (iv) holds and
\[
\hat J_i v_i\leq F_i,
\qquad
\hat J_i\rightarrow J_{\mathfrak C}
:=\Delta_{\mathfrak C}+|A_{\mathfrak C}|^2
\]
with coefficients converging in
$C^\alpha_{\mathrm{loc}}(\mathcal R(\mathfrak C))$.
Then there exists $C_{K,L}>0$, independent of $i$, such that
\[
0\leq-F_i\leq C_{K,L}\inf_L v_i
\quad\text{on }K
\]
for sufficiently large $i$.
\end{lemma}

\begin{proof}
By \cite[Lemma 5]{IlmanenSMP}, $\mathcal R(\mathfrak C)$ is
connected. Choose a connected open set $U$ such that
\[
K\cup L\subset U\Subset\mathcal R(\mathfrak C)\cap B_1.
\]
By Item (iii), for sufficiently large $i$, the relevant portions
of $\Sigma_i$ and $\Sigma_i'$ are graphical over $U$. We use
these identifications to regard $v_i,F_i$, and $\hat J_i$ as
defined on $U$.

For each $p\in K$, first decrease the radius supplied by Item
(iv), retaining the associated numbers $\lambda_{i,p,r_p}$,
so that $B_{\mathfrak C}(p,2r_p)\Subset U$ and the radius is
sufficiently small for the Dirichlet comparison below.
The estimates in Item (iv) remain valid on the smaller balls.
By compactness, choose finitely many such points and radii with
\[
K\subset\bigcup_{\ell=1}^N B_{\mathfrak C}(p_\ell,r_\ell/2).
\]
Set
\[
\lambda_i^K:=\max_{1\leq\ell\leq N}\lambda_{i,p_\ell,r_\ell}.
\]
Choose smooth domains $D_\ell$ satisfying
\[
B_{\mathfrak C}(p_\ell,r_\ell/2)\Subset D_\ell
\Subset B_{\mathfrak C}(p_\ell,r_\ell),
\qquad
\lambda_1(-J_{\mathfrak C},D_\ell)>0.
\]
The local coefficient convergence implies that $-\hat J_i$ has
positive first Dirichlet eigenvalue on every $D_\ell$ for all
sufficiently large $i$. Hence there is a unique solution
$\phi_{i,\ell}$ of
\[
\hat J_i\phi_{i,\ell}=-1\quad\text{in }D_\ell,
\qquad
\phi_{i,\ell}=0\quad\text{on }\partial D_\ell.
\]
The maximum principle and local elliptic estimates give
$\theta_\ell>0$, independent of $i$, such that
\[
\phi_{i,\ell}\geq\theta_\ell
\quad\text{on }B_{\mathfrak C}(p_\ell,r_\ell/2).
\]
Set $\theta_K:=\min_{1\leq\ell\leq N}\theta_\ell>0$.
Choose $\ell_i$ such that
$\lambda_{i,p_{\ell_i},r_{\ell_i}}=\lambda_i^K$.
By Item (iv), $\hat J_i v_i\leq-c_K\lambda_i^K$ on
$D_{\ell_i}$. Comparing $v_i$ with
$c_K\lambda_i^K\phi_{i,\ell_i}$ gives
\begin{equation}\label{Eq: vi lower bound}
v_i\geq c_K\theta_K\lambda_i^K
\quad\text{on }B_{\mathfrak C}(p_{\ell_i},r_{\ell_i}/2).
\end{equation}
Moreover, Item (iv) gives $F_i\leq0$ on $U$, and hence
$\hat J_i v_i\leq0$ there. The weak Harnack inequality and a
Harnack-chain argument, together with \eqref{Eq: vi lower bound},
give a constant $b_{K,L}>0$, independent of $i$, such that
\[
\inf_L v_i\geq b_{K,L}\lambda_i^K.
\]
Finally, Item (iv) and the finite covering of $K$ give
$0\leq-F_i\leq C_K\lambda_i^K$ on $K$. Therefore,
\[
0\leq-F_i\leq\frac{C_K}{b_{K,L}}\inf_L v_i
\quad\text{on }K.
\]
This proves the lemma.
\end{proof}

We now prove the local separation criterion.
\begin{proof}[Proof of Proposition \ref{prop: singular contact exclusion}]
Assume for contradiction that $x_0\in\Sigma\cap\Sigma'$.
By Item (v),
$x_0\in\mathcal S(\Sigma)\cap\mathcal S(\Sigma')$.
Set $\bar g=g(x_0)$. After a linear change of coordinates, we
may regard $\bar g$ as the Euclidean metric. All balls $B_R$
below are taken with respect to $\bar g$.

Fix a small constant $\rho_0$. For $0<\rho\leq\rho_0$, write
\[
\Sigma_\rho=\rho^{-1}(\Sigma-x_0),\qquad
\Sigma_\rho'=\rho^{-1}(\Sigma'-x_0),\qquad
g_\rho=\rho^{-2}\eta_\rho^*g,
\]
where $\eta_\rho(y)=x_0+\rho y$, and define
\[
\delta_\rho(x):=\dist_{g_\rho}(x,\Sigma_\rho'),
\qquad x\in\Sigma_\rho.
\]
For a measurable function $f$ on an $(n-1)$-rectifiable set $A$
and a metric $\gamma$, denote
\[
m_{A,\gamma}^a(f)
:=\sup\{t\geq0:
\mathcal H_\gamma^{n-1}(A\cap\{f\leq t\})\leq a\}.
\]
Set
\[
\beta_\rho:=\mathcal H_{g_\rho}^{n-1}(\Sigma_\rho\cap B_1),
\qquad
\mathscr S(\rho)
:=m_{\Sigma_\rho\cap B_1,g_\rho}^{\beta_\rho/2}(\delta_\rho).
\]
The almost-minimizing density bounds give $0<\beta_\rho<\infty$.
By Item (v) and Theorem \ref{Thm: ABT regularity},
\[
\mathcal H_{g_\rho}^{n-1}(\{\delta_\rho=0\})=0.
\]
Consequently, $\mathscr S(\rho)$ is finite and positive.
Item (iii) implies $\mathscr S(\rho)\rightarrow0$ as
$\rho\downarrow0$. Moreover, for every $0<a<\rho_0$, $\mathscr S(\rho)$ is
uniformly bounded below by a positive constant for
$\rho\in[a,\rho_0]$. Indeed, since the contact set has zero
area, choose $\epsilon>0$ so that the portion of
$\Sigma\cap B_{\rho_0}^{\bar g}(x_0)$ within $g$-distance
$\epsilon$ of $\Sigma'$ has area less than half that of
$\Sigma\cap B_a^{\bar g}(x_0)$. By inclusion and rescaling,
$\mathscr S(\rho)\geq\epsilon/\rho\geq\epsilon/\rho_0>0$
for every $a\leq\rho\leq\rho_0$.

Choose scales $r_i\downarrow0$ by the scale-selection procedure
of \cite[Lemma 3]{IlmanenSMP}, with $\ve_i\downarrow0$, such that
\[
(1-\ve_i)\mathscr S(r_i)\leq\mathscr S(\rho)
\quad\text{for all }\rho\in[r_i,\rho_0].
\]
For every fixed $R\geq1$ and all sufficiently large $i$, rescaling
gives
\begin{equation}\label{Eq: half separation R}
m_{\Sigma_i\cap B_R,g_i}^{(\beta_{Rr_i}/2)R^{n-1}}
\left(\frac{\delta_{r_i}}{\mathscr S(r_i)}\right)
=\frac{R\mathscr S(Rr_i)}{\mathscr S(r_i)}
\geq(1-\ve_i)R.
\end{equation}

By Item (iii), after passing to a subsequence, both boundaries
converge to the same minimizing cone $\mathfrak C$. Put
\[
\beta:=\mathcal H_{\bar g}^{n-1}(\mathfrak C\cap B_1).
\]
Varifold convergence and continuity of $g$ give, for every fixed
$R>0$,
\[
\mathcal H_{g_i}^{n-1}(\Sigma_i\cap B_R)
\rightarrow\beta R^{n-1},
\qquad
\beta_{Rr_i}\rightarrow\beta.
\]

Fix $K\Subset\mathcal R(\mathfrak C)$. By Item (iii), both
boundaries are $C^{2,\alpha}$ graphs over $\mathfrak C$ near
$K$. Using the fixed $\bar g$-unit normal of $\mathfrak C$,
let $v_i\geq0$ be the difference of their graph functions.
We use these graphical identifications to pull back quantities
from the two boundaries without further notation. Then
\begin{equation}\label{Eq: vi}
\|v_i\|_{C^{2,\alpha}(K)}\rightarrow0.
\end{equation}
The graphical and metric convergence also give
\[
\delta_{r_i}=(1+o(1))v_i
\]
locally uniformly on $\mathcal R(\mathfrak C)$.

Put $h_i=r_i h\circ\eta_i$, evaluated on $\Sigma_i'$.
By Item (ii), $h$ vanishes on $\Sigma$. The mean-value formula
along the segments between the two graphs therefore gives
\begin{equation}\label{Eq:h_i}
h_i=b_i v_i,
\qquad
\|b_i\|_{C^\alpha(K)}=O(r_i^2).
\end{equation}
Subtracting the $g_i$-mean-curvature graph expressions gives
\[
H_{\Sigma_i'}^{g_i}-H_{\Sigma_i}^{g_i}=-J_i v_i,
\]
where $J_i$ is a second-order linear operator. By the
$C^{2,\alpha}$ graphical and metric convergence in Item (iii),
its coefficients converge in $C^\alpha$ to those of
$J_{\mathfrak C}=\Delta_{\mathfrak C}+|A_{\mathfrak C}|^2$.
Since $H_{\Sigma_i}^{g_i}=0$, the conformal mean-curvature
formula gives
\[
-J_i v_i+F_i=u_i^{\frac2{n-2}}h_i.
\]
Thus, setting
\[
\hat J_i:=J_i+u_i^{\frac2{n-2}}b_i,
\]
we obtain
\[
\hat J_i v_i=F_i,
\qquad
\hat J_i\rightarrow J_{\mathfrak C}
\]
with coefficients converging in
$C^\alpha_{\mathrm{loc}}(\mathcal R(\mathfrak C))$.

Since $\mathcal S(\mathfrak C)$ has zero
$(n-1)$-dimensional measure, choose
$L\Subset\mathcal R(\mathfrak C)\cap B_1$ such that
\[
\mathcal H_{\bar g}^{n-1}(L)>\frac\beta2.
\]
Let $K\Subset\mathcal R(\mathfrak C)$, and choose $R>1$
with $K\cup L\subset B_R$. Applying Lemma
\ref{lem: compact lower bound fixed quantile}, after a fixed
dilation, gives
\[
-C_{K,L,R}\inf_{L_i}v_i\leq F_i\leq0
\quad\text{on }K,
\]
where $L_i\subset\Sigma_i$ is the graphical copy of $L$.
For large $i$,
\[
\mathcal H_{g_i}^{n-1}(L_i)>\frac{\beta_{r_i}}2,
\qquad
\delta_{r_i}\geq\frac12\inf_{L_i}v_i
\quad\text{on }L_i.
\]
Hence
\[
\mathcal H_{g_i}^{n-1}
\left(\Sigma_i\cap B_1\cap
\left\{\delta_{r_i}<\frac12\inf_{L_i}v_i\right\}\right)
<\frac{\beta_{r_i}}2,
\]
and therefore $\mathscr S(r_i)\geq\frac12\inf_{L_i}v_i$.
In particular,
\begin{equation}\label{Eq: Fi normalization}
-A_{K,L,R}\leq\frac{F_i}{\mathscr S(r_i)}\leq0
\quad\text{on }K,
\qquad A_{K,L,R}:=2C_{K,L,R}.
\end{equation}

Set $w_i:=v_i/\mathscr S(r_i)$. Then
$\hat J_i w_i=F_i/\mathscr S(r_i)$.
The set $\{\delta_{r_i}\leq\mathscr S(r_i)\}$ in
$\Sigma_i\cap B_1$ has measure at least $\beta_{r_i}/2$,
so it meets $L_i$ in a set of positive measure. Consequently,
$w_i\leq2$ somewhere on $L_i$ for large $i$.
The inhomogeneous Harnack inequality, a Harnack-chain argument,
and \eqref{Eq: Fi normalization} give locally uniform bounds
for $w_i$. Local $W^{2,p}$ estimates, with $p$ sufficiently
large, then give, after passing to a subsequence,
\[
w_i\rightarrow w
\quad\text{in }C^{1,\alpha}_{\mathrm{loc}}
(\mathcal R(\mathfrak C))
\quad\text{and weakly in }W^{2,p}_{\mathrm{loc}}.
\]
Here $w\geq0$ and $J_{\mathfrak C}w\leq0$.
If $w\equiv0$, then $\delta_{r_i}/\mathscr S(r_i)<1/2$
on $L_i$ for large $i$, contradicting the half-volume
normalization. Thus $w$ is not identically zero. Since
$\Delta_{\mathfrak C}w\leq-|A_{\mathfrak C}|^2w\leq0$,
the strong maximum principle gives $w>0$ on
$\mathcal R(\mathfrak C)$.

Passing to sublevel measures and exhausting the regular part,
the normalization and \eqref{Eq: half separation R} imply
\[
\mathcal H_{\bar g}^{n-1}
(\mathfrak C\cap B_1\cap\{w\leq1\})\geq\frac\beta2,
\qquad
\mathcal H_{\bar g}^{n-1}
(\mathfrak C\cap B_R\cap\{w<R\})
\leq\frac\beta2 R^{n-1}.
\]
In particular,
\[
m_{\mathfrak C\cap B_R,\bar g}^{(\beta/2)R^{n-1}}(w)
\geq R.
\]
The Harnack estimate from \cite[equation (20)]{IlmanenSMP}
gives $\eta_0>0$, independent of $R$, such that
\begin{equation}\label{Eq: m lower bound}
m_{\mathfrak C\cap B_R,\bar g}^{(\beta/4)R^{n-1}}(w)
\geq\eta_0R.
\end{equation}
For $R$ sufficiently large that $\eta_0R\geq1$, set
$w_T:=\min\{w,\eta_0R\}$. By \eqref{Eq: m lower bound},
the cutoff argument and mean-value inequality of
\cite[Lemma 4]{IlmanenSMP} give
\[
\fint_{\mathfrak C\cap B_1}w_T\,\mathrm d\mathcal H_{\bar g}^{n-1}
\geq
\fint_{\mathfrak C\cap B_R}w_T\,\mathrm d\mathcal H_{\bar g}^{n-1}
\geq\frac{3\eta_0R}{4}.
\]
On the other hand, at least half the measure of
$\mathfrak C\cap B_1$ lies in $\{w\leq1\}$, so
\[
\fint_{\mathfrak C\cap B_1}w_T\,\mathrm d\mathcal H_{\bar g}^{n-1}
\leq\frac12+\frac{\eta_0R}{2}.
\]
This is a contradiction for large $R$. Therefore,
$\Sigma\cap\Sigma'=\emptyset$.
\end{proof}

\begin{corollary}
\label{cor: separation time slices}
 For every $0\leq s<t$,
the conformal-flow slices are disjoint:
\[
\Sigma_t\cap\Sigma_s=\emptyset.
\]
\end{corollary}

\begin{proof}
Fix $0\leq s<t$. We apply Proposition
\ref{prop: singular contact exclusion} with
\[
\Sigma=\Sigma_s,\qquad\Sigma'=\Sigma_t,
\qquad g=g_s,\qquad g'=g_t,
\]
taking the initial metric as the reference metric $g_0$.
By construction, $E_t\subset E_s$, and both boundaries are
almost-minimizing with respect to $g_0$.
By Proposition \ref{Prop: smoothness regular part}, $g_s$ and
$g_t$ are smooth up to the respective regular parts from their
exterior sides. Moreover,
\[
g_t=u^{\frac4{n-2}}g_s,
\qquad u=\frac{u_t}{u_s}.
\]
Thus Item (i) holds. Since
\[
H_{\Sigma_s}^{g_s}=0,
\qquad H_{\Sigma_t}^{g_t}=0,
\]
Item (ii) holds with $h=0$.

Let $x_0\in\Sigma_s\cap\Sigma_t$, and let $r_i\downarrow0$.
Almost-minimizing compactness, nesting, and the one-sided
tangent-cone argument of \cite[Lemma 5]{IlmanenSMP} give,
after passing to a subsequence, a common minimizing cone
$\mathfrak C$ for the rescaled boundaries.

Fix $K\Subset\mathcal R(\mathfrak C)$. Theorem
\ref{Thm: ABT regularity} and the argument of Proposition
\ref{Prop: local regularity near regular points} give uniform
$C^{1,\theta}$ graphical estimates near $K$ for all earlier
flow boundaries meeting the chosen rescaled neighborhood.
The $C^{2,\alpha}$ extension of the initial metric gives
uniform $C^{2,\alpha}$ bounds for its rescalings near $K$.

Apply the bootstrap of Proposition
\ref{Prop: smoothness regular part} on these neighborhoods, uniformly
for times in $[0,t]$. Boundary estimates for the harmonic
functions $v_\tau$, together with interior estimates when the
corresponding boundary misses the larger neighborhood, give uniform
one-sided $C^{1,\theta}$ estimates. Integrating in time gives
the same estimates for $u_\tau$. The minimal graph equations
then give uniform $C^{2,\theta}$ graphical estimates for the
boundaries. Repeating the harmonic boundary estimates gives
uniform one-sided $C^{2,\theta}$ estimates for $u_\tau$ and
$g_\tau$ on smaller neighborhoods. Compactness and interpolation,
after decreasing the exponent, yield
\[
\Sigma_{s,i},\Sigma_{t,i}\rightarrow\mathfrak C
\quad\text{locally }C^{2,\alpha}\text{-graphically},
\]
and
\[
g_{s,i}\rightarrow g_s(x_0),
\qquad g_{t,i}\rightarrow g_t(x_0)
\]
locally in $C^{2,\alpha}$ from the respective exterior sides.
This verifies Item (iii). The same estimates and nesting give
uniform graphical convergence for $\tau\in[s,t]$.

To verify Item (iv), set
\[
\psi_\tau:=-\frac{v_\tau}{u_s},\qquad\tau\in[s,t].
\]
Then $\psi_\tau$ is positive and $g_s$-harmonic in $E_\tau$,
vanishes on $\Sigma_\tau$, and
\[
u=1-\int_s^t\psi_\tau\,\mathrm d\tau.
\]
For $p\in K$, choose $r_p>0$ such that
$B_{\mathfrak C}(p,4r_p)\Subset\mathcal R(\mathfrak C)$.
Apply boundary Harnack and boundary Schauder estimates to
$\psi_{\tau,i}:=\psi_\tau\circ\eta_i$, normalized by its
value at a fixed interior reference point. We obtain positive
numbers $a_{i,\tau,p}$ such that
\begin{equation}\label{Eq: psi upper bound}
0<\partial_{\nu_{t,i}}^{g_{s,i}}\psi_{\tau,i}
\leq C_Ka_{i,\tau,p}
\quad\text{on }B_i'(p,2r_p),
\end{equation}
and
\begin{equation}\label{Eq: psi lower bound}
\partial_{\nu_{t,i}}^{g_{s,i}}\psi_{\tau,i}
\geq c_Ka_{i,\tau,p}
\quad\text{on }B_i'(p,r_p).
\end{equation}
Here the estimates are first obtained on $\Sigma_{\tau,i}$
and then transferred to $\Sigma_{t,i}$, using the uniform
graphical convergence and the estimates for the normalized
harmonic functions. Set
\[
\lambda_{i,p,r_p}:=\int_s^t a_{i,\tau,p}\,\mathrm d\tau>0.
\]
Integrating \eqref{Eq: psi upper bound} and
\eqref{Eq: psi lower bound}, and using the local positive
upper and lower bounds for $u$, verifies Item (iv).

Finally, suppose that a contact point is regular for either
boundary. Proposition \ref{Prop: local regularity near regular points}
implies that it is regular for both. The function $u$ is
$g_s$-harmonic in $E_t$, with $u<1$ there and $u(x_0)=1$.
The Hopf lemma gives $(\partial_\nu^{g_s}u)(x_0)<0$.
The conformal mean-curvature formula and the strict barrier
argument of Proposition \ref{Prop: smoothness regular part}
exclude regular contact. Thus Item (v) holds.
Proposition \ref{prop: singular contact exclusion} now gives
$\Sigma_s\cap\Sigma_t=\emptyset$.
\end{proof}

For later use, fix $t>0$ and consider a tangent-cone blow-up
of $\Sigma_t$ at $x_0$, with limiting cone $\mathfrak C$.
By Corollary \ref{cor: separation time slices},
$\Sigma_t\cap\Sigma_0=\emptyset$. Hence the initial metric
is smooth near $x_0$, and its rescalings have uniformly
bounded derivatives.

Fix $K\Subset\mathcal R(\mathfrak C)$. Theorem
\ref{Thm: ABT regularity} gives uniform $C^{1,\theta}$
graphical estimates for the rescaled $\Sigma_t$ near $K$.
By Lemma \ref{Lem: one-sided density} and the argument of
Proposition \ref{Prop: local regularity near regular points},
the same estimates hold for all earlier flow boundaries
that intersect a sufficiently small neighborhood of $K$.
In what follows, all estimates are taken in rescaled
coordinates, on smaller neighborhoods when necessary,
with constants independent of the rescaling index and
of $s\in[0,t]$.

We proceed as in Proposition \ref{Prop: smoothness regular part}.
Since $-1\leq v_s\leq0$, the boundary estimates for the harmonic
equation give uniform $C^{1,\theta}$ bounds for $v_s$ up to
$\Sigma_s$ from the exterior. When $\Sigma_s$ does not intersect
the larger neighborhood, $v_s$ is harmonic or identically zero
there, so interior estimates apply on the smaller neighborhood.
Since $E_s\subset E_\tau$ for $\tau\leq s$, the identity
\[
u_s=1+\int_0^s v_\tau\,\mathrm d\tau
\]
gives uniform $C^{1,\theta}$ bounds for $u_s$ on the exterior
side of $\Sigma_s$. Together with $e^{-t}\leq u_s\leq1$,
the minimal graph equation then gives uniform $C^{2,\theta}$
bounds for the boundary graphs. Applying the harmonic boundary
estimates again and integrating in time gives uniform
$C^{2,\theta}$ bounds for $u_t$ from the exterior.

Thus the rescaled metrics $g_t=u_t^{4/(n-2)}g$ have locally
uniform one-sided $C^{2,\theta}$ bounds near $K$. By continuity
and interpolation, they converge locally in $C^{2,\alpha}$
to the constant metric $g_t(x_0)$ from the exterior side,
for every $0<\alpha<\theta$.

\begin{corollary}
\label{cor: separation perturbed minimizers}
Let $\Sigma=\partial E$ be a compact $g$-almost-minimizing
boundary in $(M,g)$. Assume that $g$ has a H\"older extension
across $\Sigma$, is smooth in $E$ and up to $\mathcal R(\Sigma)$
from $E$, and satisfies
\[
H_\Sigma^g=0\quad\text{on }\mathcal R(\Sigma),
\]
with respect to the unit normal pointing into $E$.
Assume also that, along every tangent-cone blow-up of $\Sigma$,
the rescaled metric $g$ has locally uniform one-sided
$C^{2,\alpha}$ bounds near the regular part of the limiting
cone, in fixed ambient coordinates.
Let $u>0$ be harmonic in $E$, continuous up to $\Sigma$, with
\[
u=1\quad\text{on }\Sigma,
\qquad u<1\quad\text{in }E,
\]
and set $g'=u^{\frac4{n-2}}g$.
Let $h$ be a smooth function on $M$, with
$h=0$ on $M\setminus E$.
Let $E'\subset E$ be a Caccioppoli set with compact boundary
$\Sigma'=\partial E'$, and suppose that $E'$ minimizes
\[
\mathcal A_h(F)
=\mathcal H_{g'}^{n-1}(\partial F)
+\int_{E\setminus F}h\,\mathrm d\mu_{g'}
\]
among Caccioppoli sets $F\subset E$ with
$E\setminus F\Subset M$. Then
\[
\Sigma'\cap\Sigma=\emptyset.
\]
\end{corollary}

\begin{proof}
We take $g$ as the reference metric in Proposition
\ref{prop: singular contact exclusion}. The exterior-density
comparison in Lemma \ref{Lem: sphere lower area}, applied to
the almost-minimizing boundary $\Sigma$, and the boundary
oscillation estimate for uniformly elliptic equations in
divergence form give a H\"older extension of $1-u$ by zero across
$\Sigma$. Thus, extending $u=1$ outside $E$, the metric $g'$
is a positive H\"older conformal change of $g$.
Comparison with $F\cap E$, perimeter submodularity, and the
almost-minimizing property of $\Sigma$ show that $\Sigma'$
is almost-minimizing; the bounded forcing contributes only
$O(r^n)$ on a ball of radius $r$. As in Lemma
\ref{Lem: almost minimizing limit}, the H\"older conformal
factor allows us to take $g$ as the reference metric.
Away from the obstacle, first variation gives
$H_{\Sigma'}^{g'}=-h$.

We first exclude regular contact. If a contact point is regular
for either boundary, the one-sided tangent-cone argument and
Theorem \ref{Thm: ABT regularity}, as in Lemma
\ref{Lem: one-sided density}, imply that it is regular for both.
At such a point, the Hopf lemma gives
$\partial_\nu^g u<0$, and hence
\[
H_\Sigma^{g'}
=u^{-\frac2{n-2}}
\left(H_\Sigma^g+
\frac{2(n-1)}{n-2}u^{-1}\partial_\nu^g u\right)<0.
\]
Since $h=0$ on $\Sigma$, this is a strict barrier for the
variational inequality $H_{\Sigma'}^{g'}+h\geq0$ satisfied
by the obstacle minimizer. The local replacement argument of
Proposition \ref{Prop: smoothness regular part} therefore
excludes regular contact. Thus Item (v) holds, and first
variation now gives
\[
H_{\Sigma'}^{g'}=-h
\quad\text{on all of }\mathcal R(\Sigma').
\]
Since $h$ is smooth and vanishes on $\Sigma$,
\begin{equation}\label{Eq: h}
|(-h)(x)|\leq\Lambda\dist_g(x,\Sigma)
\end{equation}
near $\Sigma$, for some $\Lambda>0$.
This verifies Item (ii), with forcing function $-h$.

At any remaining contact point, almost-minimizing compactness,
nesting, and \cite[Lemma 5]{IlmanenSMP} give a common minimizing
tangent cone $\mathfrak C$. Fix
$K\Subset\mathcal R(\mathfrak C)$. Theorem
\ref{Thm: ABT regularity} gives uniform $C^{1,\theta}$
graphical estimates for both rescaled boundaries near $K$.
The assumed bounds for the rescaled background metric and
$H_\Sigma^g=0$ give uniform higher graphical estimates for
$\Sigma_i$. Since $u$ is harmonic and equals one on $\Sigma$,
boundary Schauder estimates give uniform one-sided
$C^{2,\theta}$ bounds for $u\circ\eta_i$. Consequently,
\[
g_i'=(u\circ\eta_i)^{\frac4{n-2}}g_i
\]
has the same local bounds on smaller patches. The equation
$H_{\Sigma_i'}^{g_i'}=-r_i h\circ\eta_i$ then gives uniform
higher graphical estimates for $\Sigma_i'$.
After decreasing the exponent, compactness and interpolation
yield locally $C^{2,\alpha}$ graphical convergence of both
boundaries to $\mathfrak C$, and
\[
g_i,g_i'\rightarrow g(x_0)
\]
locally in $C^{2,\alpha}$ from their respective exterior sides,
where we used $u(x_0)=1$. This verifies Items (i) and (iii).

Finally, apply the boundary Harnack and boundary Schauder
argument from Corollary \ref{cor: separation time slices}
to the positive harmonic function $1-u$. The normalized
normal-derivative estimates along $\Sigma_i$ transfer to
$\Sigma_i'$ and give Item (iv), since
\[
F_i=-\frac{2(n-1)}{n-2}u_i^{-1}
\partial_{\nu_i'}^{g_i}(1-u_i).
\]
All assumptions of Proposition \ref{prop: singular contact exclusion}
are satisfied. Therefore, $\Sigma\cap\Sigma'=\emptyset$.
\end{proof}

\section{Monotonicity of ADM mass along a conformal flow}
In this section, we assume that $(M,g)$ is a smooth complete asymptotically flat manifold with {\it nonnegative scalar curvature} and compact Caccioppoli boundary $\partial M$. The original metric admits a $C^{2,\alpha}$ Riemannian extension
across $\partial M$, as in Section \ref{Sec: 2}. We consider the conformal flow
$$\{(g_t,\Sigma_t)\}_{t\geq 0}$$
constructed in the previous section with initial hypersurface $\Sigma_0=\partial M$

We are going to show 

\begin{proposition}\label{Prop: mass monotonicity}
    Along this conformal flow $\{(g_t,\Sigma_t)\}_{t\geq 0}$, the ADM mass function 
    $$m(t)=m_{ADM}(M,g_t)$$ is  non-increasing as $t$ increases.
\end{proposition}
\subsection{The derivative of the mass function}
As a preparation, let us first compute the derivative of the mass function. 

Let $v_t$ be the function from Definition \ref{Defn: conformal flow}. Given any boundary $\Sigma=\partial E$ in $\mathcal C$ let us define the $h$-capacity of $\Sigma$ by
$$\mathfrak c(\Sigma,h)=\inf\left\{\frac{1}{n(n-2)\omega_n}\int_E|\nabla_h\phi|^2\,\mathrm d\mathcal H^{n}_h:\phi\in C_c^\infty(M),\,\phi\equiv 1\mbox{ around } \partial E \right\}.$$
As in the previous section, we can construct a function 
$$w_t\in C^\alpha(M,g)$$ satisfying
$$\Delta_g w_t=0\mbox{ outside }\Sigma_t,\,w_t\equiv 1\mbox{ inside and along }\Sigma_t,\,w_t\to 0\mbox{ as }x\to\infty.$$

\begin{lemma}
    The function $w_t$ is in $W^{1,2}_{loc}(M,g)$, has finite Dirichlet energy, and we have
    $$\frac{1}{n(n-2)\omega_n}\int_{E_t}|\nabla_g w_t|^2\,\mathrm d\mathcal H^n_g=\mathfrak c(\Sigma_t,g).$$
\end{lemma}
\begin{proof}
    Denote $\Sigma_t=\partial E_t$ with $E_t\in \mathcal E$ and let $K_t=M\setminus E_t$. Recall from the previous section that we can take smooth regions $V_i$ enclosing $K_t$ such that $d_{\mathcal H}(K_t,\bar V_i)\to 0$ as $i\to+\infty$. Let $w_{t,i}$ be the smooth solution of the following equation
    $$\Delta_g w_{t,i}=0\mbox{ in }M\setminus\bar V_i,\,w_{t,i}=1\mbox{ on }\partial V_i\mbox{ and }w_{t,i}\to 0\mbox{ as }x\to\infty,$$
    and let $\hat w_{t,i}$ be the extension of $w_{t,i}$ satisfying $\hat w_{t,i}\equiv 1$ in $\bar V_i$. The function $w_t$ is then constructed by taking limit of $\hat w_{t,i}$ in $C^{\infty}_{loc}$-sense away from $\Sigma_t$.

 On the other hand, all functions $\hat w_{t,i}$ have uniformly bounded Dirichlet energy from the monotonicity of the capacity. Since $0\leq\hat w_{t,i}\leq1$, they are uniformly bounded in
$W^{1,2}_{\mathrm{loc}}(M,g)$. A subsequence converges weakly
in this space, strongly in $L^2_{\mathrm{loc}}$, and almost
everywhere to $w_t$. Weak lower semicontinuity gives finite
Dirichlet energy for $w_t$.

    To relate the function $w_t$ to the capacity $\mathfrak c(\Sigma_t,g)$, we still use a coincidence argument. Through a standard minimizing procedure, we can construct a function $\tilde w_t\in W^{1,2}_{loc}(M,g)$ with finite Dirichlet energy such that $\tilde w_t$ is harmonic in $E_t$ with
$$\frac{1}{n(n-2)\omega_n}\int_{E_t}|\nabla_g \tilde w_t|^2\,\mathrm d\mathcal H^n_g=\mathfrak c(\Sigma_t,g),$$
   and that $\tilde w_t-\eta$ lies in the completion of
$C_c^\infty(E_t)$ in the Dirichlet norm, where $\eta$ is smooth,
compactly supported, and equal to one near $K_t$. Zero extension
and the Sobolev inequality give coercivity on this completion.
Truncation gives $0\leq\tilde w_t\leq1$; the boundary estimates
of Section 2 give boundary value one, and interior estimates
together with the Sobolev bound give decay to zero at infinity. It follows from the maximum principle that $$w_t=\tilde w_t$$ and so we have
    $$\frac{1}{n(n-2)\omega_n}\int_{E_t}|\nabla_g w_t|^2\,\mathrm d\mathcal H^n_g=\mathfrak c(\Sigma_t,g).$$
    This completes the proof.
\end{proof}

\begin{lemma}
   The function $w_t$ has the expansion
   $$w_t=\mathfrak c(\Sigma_t,g)|x|^{2-n}+\tau_t\mbox{ with }\tau_t=O(|x|^{2-n-\tau})\mbox{ as }x\to\infty.$$
   Moreover, for any constant $T>0$ we can find constants $r$, $C$ and $\alpha$ such that
   \begin{equation}\label{Eq: uniform estimate}
       \left(|\tau_t|+|x||\partial\tau_t|\right)(x)\leq C|x|^{2-n-\alpha}\mbox{ for all }t\in[0,T]\mbox{ and }|x|\geq r.
   \end{equation}
\end{lemma}
\begin{proof}
Take $f$ to be a smooth function on $M$, which is different from $w_t$ only in a compact subset. Since $w_t$ is harmonic outside a compact subset, we know that $\Delta_g f$ is a function in the weighted Sobolev space $W^{0,p}_\delta$ (see \cite{Bartnik}) with some $\delta\in(2-n,0)$. It follows from \cite[Proposition 2.2]{Bartnik} that we can find a function $\tilde f\in W^{2,p}_\delta$ such that $\Delta_g \tilde f=\Delta_gf$. In particular, $\tilde f-f$ is a harmonic function on $(M,g)$. Note that both $f$ and $\tilde f$ decay to zero at infinity. It follows from the maximum principle that we have $f=\tilde f$, which implies $f\in W^{2,p}_\delta$ and so $w_t\in W^{2,p}_\delta$. It follows from \cite[Theorem 1.17]{Bartnik} that $w_t$ has the expansion
$$w_t=c_t|x|^{2-n}+\tau_t\mbox{ with }\tau_t=O(|x|^{2-n-\tau})\mbox{ as }x\to\infty,$$
for some constant $c_t$. 

Next we derive uniform estimates for $\tau_t$. The above expansion for $w_T$ yields that there exist constants $r$ and $C$ (here and in the sequel $C$ denotes a constant independent of $t$) such that
$$|w_t|\leq |w_T|\leq C|x|^{2-n}\mbox{ when }|x|\geq r.$$
Consider the function $w_t^*=w_t|x|^{n-2}$. It follows from a direct computation that
$$\Delta_{|x|^{-4}g}w_t^*=-|x|^{n+2}w_t^*\cdot \Delta_g|x|^{2-n}.$$
Since $g$ is asymptotically flat, we have $\Delta_g|x|^{2-n}=O(|x|^{-n-\tau})$. By possibly increasing the value of $r$ we obtain
$$|\Delta_{|x|^{-4}g}w_t^*|\leq C|x|^{2-\tau}\mbox{ when }|x|\geq r.$$
      Let
    $$\Phi:B_1\setminus\{0\}\to \mathbb R^n\setminus \bar B_1,\,y\mapsto x=\frac{y}{|y|^2},$$
    denote the Kelvin transform. Define
    $$\check g= \Phi^*(|x|^{-4}g)\mbox{ and }\check w_t^*=w_t^*\circ\Phi.$$
It is easy to verify
$$|\check g_{ij}-\delta_{ij}|=O(|y|^\tau).$$
By further increasing the value of $r$ we have
\begin{equation}\label{Eq: harmonic kelvin}
|\partial_i(a_{ij}\partial_j\check w_t^*)|\leq C|y|^{\tau-2}\mbox{ when }|y|\leq r^{-1},
\end{equation}
where $a_{ij}$ satisfies the uniformly elliptic condition
$$\frac{I_n}{2}\leq(a_{ij})\leq 2I_n.$$
The functions $\check w_t^*$ are uniformly bounded. The Caccioppoli
estimate with cutoffs vanishing near the origin gives a uniform local
Dirichlet-energy bound; the inner cutoff error is
$O(\varepsilon^{n-2})$. Letting $\varepsilon\downarrow0$ shows that
\eqref{Eq: harmonic kelvin} holds weakly across the origin.
Using the fact $\tau >\frac{n-2}{2}$ we see $|y|^{\tau-2}\in L^p$ with some $p> n/2$. It follows from \cite[Theorem 8.24]{Gilbarg} that
$$|\check w^*_t(y)-c_t|\leq C|y|^{\alpha}\mbox{ when }|y|\leq (2r)^{-1}$$
for some constants $C$ and $\alpha$ independent of $t$. This yields
$$|\tau_t|\leq C|x|^{2-n-\alpha}.$$
Applying interior $W^{2,p}$ estimates, with $p>n$, on rescaled
annuli for the equation
$\Delta_g\tau_t=-c_t\Delta_g|x|^{2-n}$,
after taking $r$ large enough we can obtain
$$|\partial \tau_t|\leq C\left(|x|^{1-n-\alpha}+|x|^{1-n-\tau}\right)\mbox{ when }|x|\geq r.$$
To simplify we can take $\alpha$ smaller than $\tau$ and this gives \eqref{Eq: uniform estimate}.

Now we determine the constant $c_t$. Let $S_\rho$ denote the coordinate sphere $\{|x|=\rho\}$ and $\nu_\rho$ denote the outward unit $g$-normal of $S_\rho$ pointing to the infinity. Then we can compute
$$\int_{S_\rho}\frac{\partial w_t}{\partial \nu_\rho}\,\mathrm d\mathcal H^{n-1}_g\to-n(n-2)\omega_nc_t\mbox{ as }\rho\to+\infty.$$
On the other hand, for regular values $1-\epsilon$ and with
$\nu$ pointing toward infinity, we have
$$-\int_{\{w_t\leq 1-\epsilon\}}|\nabla_g w_t|^2\,\mathrm d\mathcal H^{n}_g=(1-\epsilon)\int_{\{w_t=1-\epsilon\}}\frac{\partial w_t}{\partial \nu}\,\mathrm d\mathcal H^{n-1}_g.$$
After integrating $\Delta_gw_t=0$ outside the hypersurface $\{w_t=1-\epsilon\}$ we obtain
$$n(n-2)\omega_nc_t=\frac{1}{1-\epsilon}\int_{\{w_t\leq 1-\epsilon\}}|\nabla_g w_t|^2\,\mathrm d\mathcal H^{n}_g\to n(n-2)\omega_n \mathfrak c(\Sigma_t,g)\mbox{ as }\epsilon\to0.$$
This yields
$c_t=\mathfrak c(\Sigma_t,g)$ and 
we complete the proof.
\end{proof}

\begin{corollary}\label{Cor: expansion ut}
    The function $u_t$ has the expansion
    $$u_t=e^{-t}+\int_0^te^{-s}\mathfrak c(\Sigma_s,g)\,\mathrm ds\cdot |x|^{2-n}+\epsilon_t,$$
    where $\epsilon_t$ satisfies \eqref{Eq: uniform estimate}.
\end{corollary}
\begin{proof}
    This follows from integrating the functions $v_t=e^{-t}(w_t-1)$.
\end{proof}

Now we have the following lemma.
\begin{lemma}\label{Lem: derivative mass function}
  The mass function $m(t)$ is locally Lipschitz with respect to $t$. Moreover, for almost every $t$ we have
    $$m'(t)=-2(m(t)-\mathfrak c(\Sigma_t,g_t)).$$
\end{lemma}
\begin{proof}
By Corollary \ref{Cor: expansion ut} and the conformal mass formula,
\[
m(t)=m(0)e^{-2t}
+2e^{-t}\int_0^t e^{-s}\mathfrak c(\Sigma_s,g)\,\mathrm ds.
\]
The capacities are locally bounded, so $m$ is locally Lipschitz.
By nesting, $u_t$ is $g$-harmonic on $E_t$. Hence $-v_t/u_t$
is $g_t$-harmonic there, vanishes on $\Sigma_t$, and tends to
one at infinity. Its expansion, expressed in the asymptotically
normalized coordinates for $g_t$, gives
\[
\mathfrak c(\Sigma_t,g_t)
=e^{-2t}\mathfrak c(\Sigma_t,g)
+e^{-t}\int_0^t e^{-s}\mathfrak c(\Sigma_s,g)\,\mathrm ds.
\]
Differentiating the mass identity at Lebesgue points of its
integrand therefore yields
$m'(t)=-2\bigl(m(t)-\mathfrak c(\Sigma_t,g_t)\bigr)$
for almost every $t$.
\end{proof}

\subsection{The mass-capacity inequality}\label{Sec: 3.2}

In this subsection, we prove the mass-capacity inequality needed for the
monotonicity of the ADM mass.

Let \((M^n,g)\) be a smooth complete asymptotically flat manifold with
nonnegative scalar curvature and compact Caccioppoli boundary $\Sigma$. Assume that $\Sigma$ is almost-minimizing. In particular, we have 
\[
    \Sigma=\mathcal R(\Sigma)\cup\mathcal S(\Sigma),
    \qquad
    \dim_{\mathcal H}\mathcal S(\Sigma)\leq n-8,
\]
Assume further that \(\mathcal R(\Sigma)\) is smooth, and that
\(
    H_\Sigma^g=0
\)
on \(\mathcal R(\Sigma)\), with respect to the unit normal pointing towards
the asymptotically flat end.

We also assume that $g$ satisfies the local bounds near the regular
parts of tangent cones stated in Corollary
\ref{cor: separation perturbed minimizers}. For the original metric,
these bounds follow from the $C^{2,\alpha}$ extension assumed in
Theorem \ref{Thm: intro main}. For a flow exterior
$(\overline E_t,g_t)$, they follow from the regularity argument
following Corollary \ref{cor: separation time slices}.

Let \(w\) be the capacitary potential of \(\Sigma\), namely
\[
    \Delta_g w=0\quad\text{in }\mathring M,\qquad
    w=1\quad\text{along }\Sigma,\qquad
    w\to0\quad\text{as }x\to\infty .
\]
Then
\[
    \mathfrak c(\Sigma,g)
    =
    \frac{1}{n(n-2)\omega_n}
    \int_M |\nabla_g w|^2\,d\mu_g .
\]

\begin{proposition}\label{Prop: mass capacity}
Under the assumptions above, we have
\[
    m_{\mathrm{ADM}}(M,g)\geq \mathfrak c(\Sigma,g).
\]
\end{proposition}

The proof is divided into three steps. First we perturb the metric conformally
and use a barrier argument to construct a hypersurface with strictly positive
mean curvature. Then we approximate this hypersurface from one side by smooth
mean-convex hypersurfaces, using mean-convex smoothing. Finally we
apply the smooth mean-convex mass-capacity inequality and pass to the limit.

Set
\[
    v=w-1.
\]
Then
\[
    \Delta_g v=0\quad\text{in }\mathring M,\qquad
    v=0\quad\text{along }\Sigma,\qquad
    v\to -1\quad\text{as }x\to\infty .
\]
For \(0<\varepsilon<1\), define
\[
    u_\varepsilon
    =
    \frac{1+\varepsilon v}{1-\varepsilon},\qquad
    \tilde g_\varepsilon
    =
    u_\varepsilon^{\frac{4}{n-2}}g .
\]
Then \(u_\varepsilon\to1\) at infinity. Since \(\Delta_g u_\varepsilon=0\) in
\(\mathring M\), the conformal scalar curvature formula gives
\[
    R_{\tilde g_\varepsilon}
    =
    u_\varepsilon^{-\frac{n+2}{n-2}}R_g u_\varepsilon\geq 0 .
\]

We now produce a hypersurface with strictly positive
mean curvature. 
Choose a smooth function $h\geq 0$ which is strictly positive in \(\mathring M\), vanishes outside $M$ in the fixed extension, and decays
sufficiently fast at infinity. For \(\delta>0\), consider the functional
\[
    \mathcal A_{\varepsilon,\delta}(F)
    =
    \mathcal H_{\tilde g_\varepsilon}^{n-1}(\partial F)
    +
    \int_{M\setminus F}\delta h\,d\mu_{\tilde g_\varepsilon}
\]
among Caccioppoli sets \(F\subset M\) such that $M\setminus F$ is compact.

\begin{lemma}\label{Lem: positive mean curvature replacement}
For every sufficiently small \(\varepsilon>0\), there exist
\(\delta>0\) and a minimizer \(E_{\varepsilon,\delta}\) of
\(\mathcal A_{\varepsilon,\delta}\) such that
\(
    \Sigma_{\varepsilon,\delta}:=\partial E_{\varepsilon,\delta}
\)
satisfies
\[
    \Sigma_{\varepsilon,\delta}\cap\Sigma=\emptyset .
\]
Moreover, on the regular part of \(\Sigma_{\varepsilon,\delta}\), we have
\[
    H_{\Sigma_{\varepsilon,\delta}}^{\tilde g_\varepsilon}
    =
    \delta h
\]
with respect to the outward unit normal of $E_{\ve,\delta}$. 
In particular, there exists a constant
\(
    \gamma_{\varepsilon,\delta}>0
\)
such that
\[
    H_{\Sigma_{\varepsilon,\delta}}^{\tilde g_\varepsilon}
    \geq \gamma_{\varepsilon,\delta}\quad \text{on}\quad \mathcal R(\Sigma_{\varepsilon,\delta}). 
\]
\end{lemma}

\begin{proof}
The existence of a minimizer $E_{\ve,\delta}$ follows from the direct method. It remains to explain why one
can choose a minimizer which is separated from \(\Sigma\).

Set
$$\bar g_\ve=(1-\ve)^{-\frac{4}{n-2}}g.$$
Since $\bar g_\ve$ is a constant multiple of $g$, the hypersurface $\Sigma$ is $\bar g_\ve$-minimal, and $1+\ve v$ is $\bar g_\ve$-harmonic, identical to one on $\Sigma$ and less than one in $\mathring M$. Applying Corollary
\ref{cor: separation perturbed minimizers} to $(M,\bar g_\ve)$, with conformal factor $1+\ve v$, conformal metric $\tilde g_\ve$, and forcing term $\delta h$, gives
$$\Sigma_{\ve,\delta}\cap \Sigma=\emptyset.$$ 
On $\mathcal R(\Sigma_{\varepsilon,\delta})$, the first variation of $\mathcal A_{\ve,\delta}$ gives
$$H^{\tilde g_\ve}_{\Sigma_{\ve,\delta}}=\delta h$$
with respect to the outward unit normal of the exterior region. 
Since $\Sigma_{\ve,\delta}$ is compact and disjoint from $\Sigma$, it is contained in a compact subset of $\mathring M$. Because $h>0$ in $\mathring M$, we have
$$H^{\tilde g_\ve}_{\Sigma_{\ve,\delta}}\geq \delta \min_{\Sigma_{\ve,\delta}}h=:\gamma_{\ve,\delta}>0\quad \text{on}\quad \mathcal R(\Sigma_{\ve,\delta}).$$
This completes the proof.
\end{proof}

We shall use the following mean-convex smoothing lemma:
\begin{lemma}\label{Lem: mean convex smoothing}
Let \((\overline E,\bar g)\) be a smooth asymptotically flat Riemannian manifold with compact almost-minimizing Caccioppoli boundary
\(\Gamma\). Assume that $\mathcal R(\Gamma)$ is smooth and satisfies
\[
    H_\Gamma^{\bar g}\geq\gamma>0
\]
with respect to the outward unit normal of $E$. Then \(\Gamma\) can
be approximated from the $E$-side in the Hausdorff sense by smooth compact
hypersurfaces \(\Gamma_j\) such that
\[
    H_{\Gamma_j}^{\bar g}>0,
\]
and
\[
    \mathcal H_{\bar g}^{n-1}(\Gamma_j)
    \to
    \mathcal H_{\bar g}^{n-1}(\Gamma).
\]
Moreover, the exterior regions $E_j$ bounded by \(\Gamma_j\) converge to $E$.
\end{lemma}
\begin{remark}
Gromov formulated a mean-convex smoothing result for \(C^2\)-quasi-regular hypersurfaces and outlined two proofs in \cite[Sections~3.3--3.4 and~5.6--5.7]{GromovPlateauStein}. We independently obtained a proof of the \(C^\infty\)-quasi-regular version, which is presented in \cite{BiZhuVolumeGrowth}. After the first version of this paper, we came to understand how to carry out Gromov's first argument, based on bending and smoothing, in full detail. We include this proof of \(C^2\)-version in Appendix~A. The \(C^\infty\)-version is enough for Lemma~\ref{Lem: mean convex smoothing}, since an almost-minimizing boundary with smooth regular part is \(C^\infty\)-quasi-regular.
\end{remark}

We also use the following smooth mass-capacity inequality.
\begin{lemma}
\label{Lem: smooth mass capacity}
Let \((N^n,h)\) be a smooth asymptotically flat manifold with nonnegative scalar
curvature and compact smooth boundary \(\Gamma=\partial N\). Assume that
\(\Gamma\) is mean-convex with respect to the outward unit normal of the exterior region. 
Then
\[
    m_{\mathrm{ADM}}(N,h)\geq \mathfrak c(\Gamma,h).
\]
\end{lemma}

\begin{proof}
The following is Bray's doubling and conformal compactification argument. Let \(\phi\)
be the capacitary potential of \(\Gamma\):
\[
    \Delta_h\phi=0,\qquad
    \phi=0\quad\text{on }\Gamma,\qquad
    \phi\to1\quad\text{at infinity}.
\]
With our normalization, the expansion of \(\phi\) is
\[
    \phi=1-\mathfrak c(\Gamma,h)|x|^{2-n}+O(|x|^{1-n}).
\]

Let \((\bar N,\bar h)\) be the double of \(N\) across \(\Gamma\). On the two
copies of \(N\), define
\[
    \psi =
    \begin{cases}
    \phi, & \text{on the first copy},\\
    -\phi, & \text{on the second copy}.
    \end{cases}
\]
 The doubled metric has a
corner along \(\Gamma\). The mean-convexity of \(\Gamma\) is precisely the
mean-curvature jump condition in Miao's positive mass theorem for manifolds
with corners \cite{Miao}.

Now compactify one of the two ends by the conformal factor
\[
    \left(\frac{1+\psi}{2}\right)^{\frac{4}{n-2}}
\]
and keep the other end asymptotically flat. Denote the resulting corner metric
by \(\widetilde h\), and let
\(
    \widetilde m
\)
be the ADM mass of its remaining asymptotically flat end. A direct computation
from the asymptotic expansions gives
\[
    \widetilde m
    =
    m_{\mathrm{ADM}}(N,h)-\mathfrak c(\Gamma,h).
\]
Moreover, the scalar curvature of \(\widetilde h\) is nonnegative away from the
corner, and the corner condition is satisfied. Therefore, Miao's positive mass theorem
with corners \cite{Miao}, combined with recent resolution of all dimensional positive mass theorem \cite{BrendleWangPMT}, gives
\(
    \widetilde m\geq0,
\)
and equivalently
\[
    m_{\mathrm{ADM}}(N,h)\geq \mathfrak c(\Gamma,h).
\]
This completes the proof.
\end{proof}

\begin{proof}[Proof of Proposition \ref{Prop: mass capacity}]
Fix \(\varepsilon>0\) sufficiently small. By
Lemma \ref{Lem: positive mean curvature replacement}, we can choose
\(\delta>0\) and a minimizer $E_{\ve,\delta}$ such that $\Sigma_{\ve,\delta}:=\partial E_{\ve,\delta}$ lies strictly outside \(\Sigma\), and its regular part has strictly positive
mean curvature in the metric \(\tilde g_\varepsilon\), with respect to the outward unit normal of $E_{\ve,\delta}$. Moreover, $\Sigma_{\ve,\delta}$ is almost-minimizing. Indeed, if $F$ agrees with $E_{\ve,\delta}$ outside a sufficiently small ball $B_r$, the minimizing property gives
$$\mathcal H^{n-1}_{\tilde g_\ve}(\Sigma_{\ve,\delta}\cap B_r)\leq \mathcal H^{n-1}_{\tilde g_\ve}(\partial F\cap B_r)+\delta\|h\|_{L^\infty}\mathcal H^n_{\tilde g_\ve}(B_r)\leq \mathcal H^{n-1}_{\tilde g_\ve}(\partial F\cap B_r)+Cr^n.$$

By Lemma \ref{Lem: mean convex smoothing} applied to $(\overline E,\bar g)=(\overline{E_{\ve,\delta}},\tilde g_\ve)$, we obtain exterior regions $E_{\ve,\delta,j}\subset E_{\ve,\delta}$ with smooth boundaries
\(
    \Gamma_{\varepsilon,\delta,j}:=\partial E_{\ve,\delta,j},
\)
lying strictly outside $\Sigma$,
approximating \(\Sigma_{\varepsilon,\delta}\) from the $E_{\ve,\delta}$-side and satisfying
$$H^{\tilde g_\ve}_{\Gamma_{\ve,\delta,j}}>0$$
with respect to the outward unit normal of $E_{\ve,\delta,j}$.

The region $(\overline {E_{\ve,\delta,j}},\tilde g_\ve)$ is smooth and asymptotically flat, has nonnegative scalar curvature, and has smooth mean-convex boundary.
Lemma \ref{Lem: smooth mass capacity} yields
\[
   m_{\mathrm{ADM}}(M,\tilde g_\varepsilon)= m_{\mathrm{ADM}}(\overline {E_{\ve,\delta,j}},\tilde g_\varepsilon)
    \geq
    \mathfrak c(\Gamma_{\varepsilon,\delta,j},\tilde g_\varepsilon).
\]
Since \(\Gamma_{\varepsilon,\delta,j}\) encloses \(\Sigma\), the monotonicity
of capacity gives
\[
    \mathfrak c(\Gamma_{\varepsilon,\delta,j},\tilde g_\varepsilon)
    \geq
    \mathfrak c(\Sigma,\tilde g_\varepsilon).
\]
Consequently, we arrive at
\[
    m_{\mathrm{ADM}}(M,\tilde g_\varepsilon)
    \geq
    \mathfrak c(\Sigma,\tilde g_\varepsilon).
\]

The conformal mass formula and the fact $1\leq u_\ve\leq (1-\ve)^{-1}$ imply
$$m_{\mathrm{ADM}}(M,\tilde g_\ve)=m_{\mathrm{ADM}}(M,g)+\frac{2\ve}{1-\ve}\mathfrak c(\Sigma,g)$$
and
$$\mathfrak c(\Sigma,g)\leq \mathfrak c(\Sigma,\tilde g_\ve)\leq (1-\ve)^{-2}\mathfrak c(\Sigma,g).$$
Letting \(\varepsilon\downarrow0\), we obtain
\[
    m_{\mathrm{ADM}}(M,g)\geq \mathfrak c(\Sigma,g).
\]
This completes the proof.
\end{proof}

\subsection{Rigidity in the mass-capacity inequality}\label{sec:rigidity-mass-capacity}

In this subsection, we prove the rigidity statement for the mass-capacity inequality. Throughout
this subsection, let \((M^n,g)\) be a smooth complete asymptotically flat manifold with compact
almost-minimizing Caccioppoli boundary \(\Sigma\).

When $\mathcal S(\Sigma)\ne\emptyset$, assume that $g$ agrees near
$\mathcal S(\Sigma)$ with a fixed asymptotically flat metric $g_0$
for which $\Sigma$ is outer-minimizing and almost-minimizing, with
minimal regular part. Assume that $g_0$ has a H\"older Riemannian
extension across $\Sigma$, is smooth in the interior and up to
$\mathcal R(\Sigma)$ from $M$, and satisfies the local rescaled metric
bounds of Corollary \ref{cor: separation perturbed minimizers}.
Assume also that $R_{g_0}\in L^\infty$.
The metric variations and approximations at infinity below preserve
the agreement near $\mathcal S(\Sigma)$.

These assumptions hold for flow exteriors by taking $g_0=g_t$.
The local metric bounds were proved after Corollary
\ref{cor: separation time slices}. For the scalar-curvature bound,
$\Delta_g u_t=0$ on $E_t$ gives
\[
R_{g_t}=u_t^{-4/(n-2)}R_g,\qquad
|R_{g_t}|\,d\mu_{g_t}=u_t^2|R_g|\,d\mu_g.
\]
Since $e^{-t}\le u_t\le1$ and the initial scalar curvature belongs
to $L^1\cap L^\infty$, the same holds on every flow exterior.

Set
\[
    a_n=\frac{4(n-1)}{n-2}.
\]
Define
\[
\begin{split}
Y_N(g)
:=
\inf_{u\geq 0,\ u\to1}
\bigg[
\int_M
\left(
a_n|\nabla_g u|^2+R_g u^2
\right)d\mu_g
+
2\int_\Sigma H_g u^2\,d\sigma_g
\bigg]
\end{split}
\]
and 
\[
Y_D(g)
:=
\inf_{\substack{u\geq 0,\ u\to1\\ u|_\Sigma=0}}
\int_M
\left(
a_n|\nabla_g u|^2+R_g u^2
\right)d\mu_g ,
\]
where \(H_g\) is the mean curvature with respect to the outward unit normal of $M$ along $\mathcal R(\Sigma)$.
The infima are taken in $1+D^{1,2}(M)$ and
$\chi+D^{1,2}_0(M)$, respectively, where $\chi$ is smooth, zero near
$\Sigma$, and one near infinity. These spaces are defined by completing
compactly supported smooth functions on $\overline M$ and in
$\mathring M$, respectively, in the Dirichlet norm.
The almost-minimizing property of \(\Sigma\) gives a relative isoperimetric
inequality, and hence the corresponding Sobolev and trace inequalities. 

\begin{lemma}\label{Lem: conformal minimizers}
    There exists $\delta>0$, depending on the geometric constants in the Sobolev and trace inequalities, such that if
\begin{equation}\label{Eq: small Rg and Hg}
    \|R_g^-\|_{L^{n/2}(M)}+\|H_g^-\|_{L^{n-1}(\Sigma)}<\delta,
\end{equation}
then the functionals $Y_N(g)$ and $Y_D(g)$ are bounded from below and attain their minimum at unique nonnegative functions $u_N$ and $u_D$, respectively. Moreover, $u_N>0$ on $\mathring M\cup \mathcal R(\Sigma)$, while $u_D>0$ on $\mathring M$, and both functions tend to $1$ at infinity. The function $u_N$ extends positively and H\"older continuously
to the whole boundary. With respect to the outward unit normal $\eta$ of $M$, the minimizers satisfy
\begin{equation}\label{Eq: u_N}
    \left(-a_n\Delta_g+R_g\right)u_N=0\quad \text{in}\quad \mathring M,\quad \left(a_n\frac{\partial}{\partial\eta}+2H_g\right)u_N=0\quad\text{on}\quad \mathcal R(\Sigma),
\end{equation}
and 
$$\left(-a_n\Delta_g+R_g\right)u_D=0\quad \text{in}\quad \mathring M,\quad u_D=0\quad\text{on}\quad \mathcal R(\Sigma).$$
\end{lemma}
\begin{proof}
    By H\"older's inequality together with the Sobolev and trace inequalities, \eqref{Eq: small Rg and Hg} allows the terms involving $R_g^-$ and $H_g^-$ to be absorbed into the Dirichlet term. The quadratic forms defining $Y_N(g)$ and $Y_D(g)$ are therefore coercive. The direct method gives unique minimizers. Replacing a minimizer by its absolute value and applying the strong maximum principle and the Hopf lemma gives the asserted positivity. The displayed equations follow from the first variation. The limit at infinity follows from the asymptotic estimates.
Proposition~\ref{prop:neumann-conformal-minimizer} gives the asserted boundary regularity and positive lower bound.
\end{proof}

\begin{proposition}\label{Prop: mass Yamabe}
Suppose that \eqref{Eq: small Rg and Hg} holds. Then we have
\[
m(g)\geq
\frac{1}{4n(n-1)\omega_n}
\left(
Y_N(g)+Y_D(g)
\right).
\]
\end{proposition}
\begin{proof}
We first assume that $R_g$ is compactly supported. By Lemma \ref{Lem: conformal minimizers}, \(Y_N(g)\) admits a positive minimizer $u_N$. Set
\[
    \hat g=u_N^{\frac{4}{n-2}}g .
\]
The conformal transformation formulas give
$$R_{\hat g}=u_N^{-\frac{n+2}{n-2}}(-a_n\Delta_g u_N+R_gu_N)=0$$
in $\mathring M$, and
$$H_\Sigma^{\hat g}=u_N^{-\frac{2}{n-2}}\left(H_g+\frac{a_n}{2}u_N^{-1}\partial_\eta u_N\right)=0\quad\text{on}\quad \mathcal R(\Sigma).$$

By Proposition~\ref{prop:neumann-conformal-minimizer},  \(\hat g\) extends to a H\"older metric across
\(\Sigma\) and satisfies the hypotheses of Proposition~\ref{Prop: mass capacity}.

Since $R_g$ is compactly supported, $u_N$ is $g$-harmonic near infinity. Standard harmonic asymptotics on an asymptotically flat end imply that $\hat g$ is asymptotically flat and justify the conformal ADM mass formula
\[
    m(g)-m(\hat g)
    =
    \frac{2}{n(n-2)\omega_n}
    \lim_{\rho\to\infty}
    \int_{S_\rho}u_N\partial_r u_N\,d\sigma .
\]
We compute the boundary term at infinity. Let \[M_\rho=M\cap\{|x|\leq\rho\}.\]
Integration by parts from \eqref{Eq: u_N} gives
\[
\begin{split}
&\int_M
\left(
a_n|\nabla_g u_N|^2+R_g u_N^2
\right)d\mu_g
+
2\int_\Sigma H_g u_N^2\,d\sigma_g  \\
&\qquad =
a_n\lim_{\rho\to\infty}
\int_{S_\rho}u_N\partial_r u_N\,d\sigma .
\end{split}
\]
Therefore, we obtain
\begin{equation}\label{eq:mass-drop-uN}
    m(g)-m(\hat g)
    =
    \frac{1}{2n(n-1)\omega_n}Y_N(g).
\end{equation}

Let \(v\) be the capacitary potential of \(\Sigma\) with respect to
\(\hat g\), normalized by
\[
    v=0\quad\text{on }\Sigma,
    \qquad
    v\to1\quad\text{at infinity}.
\]
The mass-capacity inequality applied to \((M,\hat g,\Sigma)\) gives
\[
    m(\hat g)\geq \mathfrak c(\Sigma,\hat g).
\]
By conformal covariance of the Dirichlet energy,
\[
\begin{split}
\mathfrak c(\Sigma,\hat g)
&=
\frac{1}{n(n-2)\omega_n}
\int_M |\nabla_{\hat g}v|_{\hat g}^2\,d\mu_{\hat g}  \\
&=
\frac{1}{n(n-2)\omega_n}
\int_M u_N^2|\nabla_g v|_g^2\,d\mu_g .
\end{split}
\]
Set
\[
    w=u_Nv .
\]
Then \(w\to1\) at infinity and \(w=0\) on \(\Sigma\). We use the identity
\[
    u_N^2|\nabla_g v|^2
    =
    |\nabla_g(u_Nv)|^2
    -
    \operatorname{div}_g(v^2u_N\nabla_g u_N)
    +
    v^2u_N\Delta u_N .
\]
Integrating over \(M_\rho\), the boundary integration on \(\Sigma\) from the divergence term vanishes because
\(v=0\) on \(\Sigma\). Letting \(\rho\to\infty\), and using
\eqref{Eq: u_N}, we obtain
\[
\begin{split}
a_n\int_M u_N^2|\nabla_g v|^2\,d\mu_g
&=
\int_M
\left(
a_n|\nabla_g w|^2+R_g w^2
\right)d\mu_g
-
Y_N(g).
\end{split}
\]
Since \(w\) is admissible for \(Y_D(g)\), it follows that
\[
    a_n\int_M u_N^2|\nabla_g v|^2\,d\mu_g
    \geq
    Y_D(g)-Y_N(g).
\]
Thus
\begin{equation}\label{eq:capacity-yamabe-lower}
\mathfrak c(\Sigma,\hat g)
\geq
\frac{1}{4n(n-1)\omega_n}
\left(
Y_D(g)-Y_N(g)
\right).
\end{equation}
Combining \eqref{eq:mass-drop-uN}, the mass-capacity inequality for
\(\hat g\), and \eqref{eq:capacity-yamabe-lower}, we get
\[
\begin{split}
m(g)
&=
m(\hat g)
+
\frac{1}{2n(n-1)\omega_n}Y_N(g)  \\
&\geq
\frac{1}{4n(n-1)\omega_n}
\left(
Y_D(g)-Y_N(g)
\right)
+
\frac{1}{2n(n-1)\omega_n}Y_N(g) \\
&=
\frac{1}{4n(n-1)\omega_n}
\left(
Y_N(g)+Y_D(g)
\right).
\end{split}
\]
This proves the proposition.

For a general asymptotically flat metric, take a standard cut-off approximation $g_j$ which agrees with $g$ on an increasing sequence of compact sets containing $\Sigma$ and is Schwarzschild outside a larger compact set. The agreement near $\mathcal S(\Sigma)$ and coercivity are
preserved for large $j$. According to Lemma \ref{Lem: approximation 1}, the approximation may be chosen so that
$$\mathrm{supp} R_{g_j}\Subset M,\qquad m(g_j)\to m(g),$$
and 
$$Y_N(g_j)\to Y_N(g),\qquad Y_D(g_j)\to Y_D(g).$$
This proves the proposition.
\end{proof}

\begin{proposition}\label{Prop: mass capacity rigidity}
Assume \(R_g\geq0\) and \(H_g\geq0\). Assume also the tangent-cone metric bounds in Theorem \ref{Thm: intro main}. If
\[
    m(g)=\mathfrak c(\Sigma,g),
\]
then
\[
    R_g\equiv0,
    \qquad
    H_g\equiv0.
\]
Moreover, \((M,g)\) is isometric to a Riemannian Schwarzschild exterior.
\end{proposition}

\begin{proof}
Since \(R_g\geq0\) and \(H_g\geq0\), we have
\[
    Y_N(g)\geq0
\]
and
\[
Y_D(g)
\geq
a_n
\inf_{\substack{u\to1\\ u|_\Sigma=0}}
\int_M|\nabla_g u|^2\,d\mu_g
=
4n(n-1)\omega_n\,\mathfrak c(\Sigma,g).
\]
Proposition \ref{Prop: mass Yamabe} therefore gives
\[
m(g)
\geq
\frac{Y_N(g)+Y_D(g)}{4n(n-1)\omega_n}
\geq
\mathfrak c(\Sigma,g).
\]
If \(m(g)=\mathfrak c(\Sigma,g)\), then the equality holds above. In particular, we have $Y_N(g)=0$.
Since the terms in the definition of \(Y_N(g)\) are nonnegative, this forces
\[
    R_g\equiv0,
    \qquad
    H_g\equiv0,
\]
and that the Neumann minimizer is \(u_N\equiv1\).

 Let \(g_s=g+sk\), where \(k\) is
a smooth compactly supported symmetric two-tensor supported away from
\(\mathcal S(\Sigma)\). The ADM mass is unchanged and \eqref{Eq: small Rg and Hg} holds under such variations.
These variations preserve the standing assumptions near
$\mathcal S(\Sigma)$, so Proposition \ref{Prop: mass Yamabe}
applies to $g_s$. Since equality attains at \(s=0\), the function
\[
    s\longmapsto Y_N(g_s)+Y_D(g_s)
\]
has a local maximum at \(s=0\), and its first variation vanishes.

Because \(u_N\equiv1\), a direct computation of the first variation gives
\[
    \frac{d}{ds}Y_N(g_s)\Big|_{s=0}
    =
    -\int_M \langle k,\operatorname{Ric}_g\rangle\,d\mu_g
    -
    \int_\Sigma \langle k,A\rangle\,d\sigma_g .
\]
Let \(u=u_D\) be the Dirichlet minimizer and set $\tilde g=u^{\frac{4}{n-2}}g$. The first variation of \(Y_D\) gives
\[
    \frac{d}{ds}Y_D(g_s)\Big|_{s=0}
    =
    -\int_M
    u^2\langle k,{\operatorname{Ric}_{\tilde g}}\rangle\,d\mu_g .
\]
Therefore, for every such \(k\),
\[
\int_M
\left\langle
k,
\operatorname{Ric}_g+u^2{\operatorname{Ric}_{\tilde g}}
\right\rangle\,d\mu_g
+
\int_\Sigma \langle k,A\rangle\,d\sigma_g
=0.
\]
By the arbitrariness of \(k\), we first obtain
\[
    \operatorname{Ric}_g+u^2{\operatorname{Ric}_{\tilde g}}=0
\]
in the interior, and then
\[
    A=0
\]
on \(\mathcal R(\Sigma)\). Thus the regular part of \(\Sigma\) is totally
geodesic.

We claim that \(\Sigma\) has no singular points. Suppose that
\(p\in\mathcal S(\Sigma)\). By the almost-minimizing property, a blow-up of
\(\Sigma\) at \(p\) is an area-minimizing cone \(\mathfrak C\). By the local metric bounds and continuity, the rescaled metrics
converge locally in $C^{2,\alpha}$ to $g(p)$ near
$\mathcal R(\mathfrak C)$, after decreasing the exponent.
Theorem \ref{Thm: ABT regularity} and the minimal graph equation
then give local $C^{2,\alpha}$ graphical convergence there. Since \(A=0\) on
\(\mathcal R(\Sigma)\), the regular part of \(\mathfrak C\) is totally geodesic. Since
\(\mathcal R(\mathfrak C)\) is connected and dense in \(\mathfrak C\), the cone \(\mathfrak C\) must be a hyperplane. The
\(\varepsilon\)-regularity theorem then implies that \(p\) is regular, a
contradiction. Hence \(\Sigma\) is smooth and totally geodesic.

We are now in the smooth-boundary equality case.  The doubling and
conformal compactification construction from Lemma
\ref{Lem: smooth mass capacity}, together with the fact that \(\Sigma\) is
totally geodesic, puts us in the rigidity case of the McFeron--Sz\'ekelyhidi
corner theorem \cite{McFeronSzekelyhidi} for Miao's positive mass theorem
\cite{Miao}. So the compactified double is Euclidean, and undoing the
compactification shows that \((M,g)\) is a Riemannian Schwarzschild exterior.
\end{proof}

\subsection{Proof of the monotonicity}

We are ready to complete the proof of the monotonicity of the ADM mass along the
conformal flow.

\begin{proof}[Proof of Proposition \ref{Prop: mass monotonicity}]
By Corollary \ref{Cor: E_t AF}, the metric estimates following
Corollary \ref{cor: separation time slices}, and Proposition
\ref{Prop: mass capacity}, applied to $(\overline{E_t},g_t)$,
we have
\[
    m(t)=m_{ADM}(\overline {E_t},g_t)\geq \mathfrak c(\Sigma_t,g_t).
\]
For almost every $t$, Lemma
\ref{Lem: derivative mass function} gives
\[
    m'(t)
    =
    -2\left(m(t)-\mathfrak c(\Sigma_t,g_t)\right)
    \leq0.
\]
Since \(m(t)\) is locally Lipschitz, it follows that \(m(t)\) is
non-increasing.
\end{proof}

\section{The convergence of a conformal flow}

In this section, we establish the convergence of the conformal flow in the {\it harmonically
flat} case. Namely, on the chosen asymptotically flat end we assume in addition that, in
coordinates for \(|x|>r_0\),
\begin{equation}\label{Eq: conformal flat}
    g=U^{\frac{4}{n-2}}g_{euc},
\end{equation}
where \(U>0\) is harmonic and \(U\to1\) at infinity. This property is preserved
along the conformal flow. For convenience, we extend $U$ to a smooth positive function on $M$, agreeing with the original harmonic factor near infinity.

The arguments in this section are adapted from Bray--Lee \cite{Bray-Lee}, with
minor modifications. We first prove the vanishing mass
property, which is the basic input for the subsequent convergence analysis.

We introduce the auxiliary conformal factor
\[
    \widetilde u_t=\frac12(u_t-v_t),
\]
and set
\[
    \widetilde g_t=\widetilde u_t^{\frac{4}{n-2}}g .
\]
We denote
\[
    \widetilde m(t):=m_{\mathrm{ADM}}(M,\widetilde g_t).
\]
With our normalization of ADM mass and capacity, the expansion at infinity gives
\[
    \widetilde m(t)
    =
    m(t)-\mathfrak c(\Sigma_t,g_t).
\]

\begin{lemma}\label{lem:vanishing_mass}
\[
 \lim_{t\to\infty}\tilde{m}(t)=0.
\]
\end{lemma}

\begin{proof}
We begin with the following asymptotic expansions
\[
v_t=-e^{-t}+b(t)|x|^{2-n}+O(|x|^{1-n}),\]
and
\[
u_t=e^{-t}+B(t)|x|^{2-n}+O(|x|^{1-n}),
\]
where \(B(t)=\int_0^t b(s)\,ds\). From these we obtain
\[
\tilde{m}(t)=e^{-2t}m(0)+e^{-t}(B(t)-b(t)),\]
and
\[
m(t)=e^{-2t}m(0)+2e^{-t}B(t).
\]
Consequently, we have
$$\int_0^t\tilde m(s)\,ds=\frac{m(0)-m(t)}{2}.$$
By Proposition \ref{Prop: mass capacity} and Proposition \ref{Prop: mass monotonicity}, we have
$$0\leq \tilde m(t)\leq m(t)\leq m(0).$$
In particular, $\tilde m$ is integrable on $[0,\infty)$.

The crucial observation is that \(e^tv_t(x)\) is nondecreasing in \(t\), by the nesting of the hypersurfaces and the maximum principle. Comparing the expansions at infinity shows that \(e^tb(t)\) is also nondecreasing. Since
\[
e^tb(t)=\tfrac{1}{2}e^{2t}(m(t)-2\tilde{m}(t))+\tfrac{1}{2}m(0),
\]
the quantity \(e^{2t}(m(t)-2\tilde{m}(t))\) is nondecreasing. In particular, the quantity
\[
m(t)-2\tilde{m}(t)+2\int_0^t(m(s)-2\tilde{m}(s))\,ds,
\]
or equivalently, the quantity
\[
3m(t)-2\tilde{m}(t)+2\int_0^tm(s)\,ds
\]
is nondecreasing. Its monotonicity, together with the monotonicity of $m$, gives, for $s<t$
$$\tilde m(t)\leq \tilde m(s)+m(0)(t-s).$$
Fix $h>0$. Averaging this inequality over $s\in [t-h,t]$, we obtain
$$0\leq \tilde m(t)\leq \frac{1}{h}\int_{t-h}^t\tilde m(s)\,ds+\frac{m(0)h}{2}.$$
Since $\tilde m$ is integrable on $[0,\infty)$, the integral tends to zero as $t\to\infty$, and we have
$$0\leq \limsup_{t\to\infty}\tilde m(t)\leq \frac{m(0)h}{2}.$$
Letting $h\to 0$ proves $\tilde m(t)\to 0$ as $t\to\infty$.
\end{proof}

The aim is to show that the hypersurfaces \(\Sigma_t\) become increasingly
round under blow-down. As a first step, we establish a uniform radius bound.

\begin{lemma}\label{lem:radius_bound}
There exists \(R_{\max}>0\) such that the hypersurface \(\Sigma_t\) remains enclosed within the coordinate sphere of radius \(e^{\frac{2t}{n-2}}R_{\max}\) for all \(t\geq0\).
\end{lemma}

\begin{proof}
Choose $R_{\mathrm{max}}\geq \max\{1,r_0\}$ large enough and constants $a, C>0$ such that
\begin{equation}\label{Eq: U}
    a\leq U(x)\leq 1+C|x|^{2-n}\quad\text{for}\quad |x|\geq R_{\mathrm{max}},
\end{equation}
and that $\Sigma_1$ is enclosed by the coordinate sphere of radius $R_{\mathrm{max}}$.
Set $f_t=Uu_t$, then $g_t=f_t^{\frac{4}{n-2}}g_{euc}$ on this end. Put
$$R(s)=e^{\frac{2s}{n-2}}R_{\mathrm{max}} .$$

We claim that, after increasing $R_{\mathrm{max}}$ independent of $T$, the assumption $\Sigma_s\subset B_{R(s)}$ for all $0\leq s\leq T$ implies $\Sigma_{T+1}\subset B_{R(T+1)}$.
For every \(s<T\), the function $Uv_s$ is Euclidean harmonic outside $B_{R(s)}$, is nonpositive on its boundary, and tends to $-e^{-s}$ at infinity. The maximum principle therefore gives
$$U(x)v_s(x)\leq e^{-s}\left[\left(\frac{R(s)}{|x|}\right)^{n-2}-1\right]\quad\text{for}\quad |x|\geq R(s).$$
Since $-e^{-s}\leq v_s\leq 0$ and \eqref{Eq: U}, for $|x|\geq R(T)$ we obtain $U(x)\leq 1+Ce^{-T}$,
\[
\begin{split}
    \int_0^{T+1}U(x)v_s(x)\,ds&\leq \int_0^Te^{-s}\left[\left(\frac{R(s)}{|x|}\right)^{n-2}-1\right]\,ds\\
    &\leq e^{-T}-1+(e^T-1)R_{\mathrm{max}}^{n-2}R(T)^{2-n}\\
    &\leq 2e^{-T}-1
\end{split}
\]
and
$$U(x)u_{T+1}(x)\geq a\left(1-\int_0^{T+1}e^{-s}\,ds\right)=ae^{-T-1}.$$
In particular, we have
$$ae^{-T-1}\leq f_{T+1}(x)\leq (C+2)e^{-T}.$$

Suppose for contradiction that \(\Sigma_{T+1}\) contains a point \(p\) outside \(B_{R(T+1)}\).
Since $$R(T+1)=e^{\frac{2}{n-2}}R(T),$$ the coordinate ball $B_\rho(p)$ lies outside $B_{R(T)}$, where
$$\rho:=\frac{1}{2}\left(e^{\frac{2}{n-2}}-1\right)R(T).$$
Note that $\Sigma_{T+1}$ satisfies the almost-minimizing inequality of Definition \ref{Defn: almost minimizing} with $\beta=0$, with respect to $e^{\frac{4T}{n-2}}g_{T+1}$ (uniformly equivalent to $g_{euc}$ independent of $T$ and $R_{\mathrm{max}}$), for variations compactly supported in $B_\rho(p)$. The same comparison and isoperimetric argument used in Lemma \ref{Lem: sphere lower area} yields
$$|\Sigma_{T+1}\cap B_\rho(p)|_{g_{euc}}\geq c(n,a,C)\rho^{n-1}.$$
We remark that this estimate holds even when $p$ is singular. It follows that
$$|\Sigma_{T+1}|_{g_{T+1}}\geq \left(a e^{-T-1}\right)^{\frac{2(n-1)}{n-2}}\cdot c(n,a,C)\rho^{n-1}\geq c_*R_{\mathrm{max}}^{n-1}$$
for some positive constant $c_*>0$ independent of $T$ and $R_{\mathrm{max}}$.
Choosing \(R_{\max}\) sufficiently large contradicts 
\(
\mathscr A\equiv|\Sigma_t|_{g_t}.
\)
This completes the proof for the claim.

The lemma follows from the claim by induction.
\end{proof}

\begin{remark}\label{rem:asymptotic_bounds}
With the constants $a,C$ from \eqref{Eq: U}, for any $0\leq \tau\leq T$, the function $f_T$ satisfies  
\[
ae^{-T}\leq f_T(x)\leq (C+2)e^{-(T-\tau)}
\]
for \(|x|\geq R(T-\tau)\).
\end{remark}

To analyze the manifold convergence as \(t\to\infty\), we employ dynamically rescaled coordinates below. For this purpose, we introduce the smooth vector field
\[
X = \frac{2}{n-2}r\frac{\partial}{\partial r} \quad \text{on } \{|x|>r_0\}\quad\text{with}\quad r=|x|,
\]
extended smoothly to all of \(M\), which generates a flow \(\Phi_t\). In the chosen asymptotic coordinates, the flow $\Phi_t$ acts as a Euclidean dilation:
$$\Phi_t(x)=e^{\frac{2t}{n-2}}x\quad\text{for}\quad |x|>r_0,\quad t\geq 0.$$

\begin{definition}\label{def:normalized_flow}
The normalized conformal flow \((M,G_t,\Sigma^*_t)\) is defined by
\[
G_t := \Phi_t^*g_t,\qquad E^*_t=\Phi_t^{-1}(E_t), \qquad \Sigma^*_t := \Phi_t^{-1}(\Sigma_t).
\]
\end{definition}

Let $U$ be the function in \eqref{Eq: conformal flat}, extended to the whole $M$.
In the following, we use the rescaled functions
\[
U_t(x) := e^t (U u_t)(\Phi_t(x)), \qquad V_t(x) := e^t (Uv_t)(\Phi_t(x)),
\]
and the background metric
\[
(G_0)_t := e^{-\frac{4t}{n-2}}\Phi_t^*\left(U^{-\frac{4}{n-2}}g\right).
\]
Then
$$G_t = U_t^{\frac{4}{n-2}}(G_0)_t.$$
Recall that \(v_t\) is harmonic with respect to \(g\) in $E_t$, so the function \(V_t\) satisfies
\[
\begin{cases}
\Delta_{(G_0)_t}V_t = \frac{\Delta_{(G_0)_t}(U\circ\Phi_t)}{U\circ \Phi_t}V_t, & \text{in } E^*_t, \\
V_t = 0, &\text{on }\Sigma^*_t, \\
\lim_{|x|\to\infty} V_t(x) = -1.
\end{cases}
\]
Moreover, $V_t$ vanishes inside $\Sigma^*_t$.

Differentiating the definition of \(U_t\), we obtain, for almost every $t$, on the exterior,
\[
\partial_tU_t=V_t+U_t+XU_t.
\]
Define $\widetilde G_t:=\Phi_t^*\tilde g_t$. Then we have 
\[
\widetilde{G}_t := W_t^{\frac{4}{n-2}}(G_0)_t\quad\text{with}\quad W_t := \tfrac{1}{2}(U_t - V_t).
\]
In particular, $\widetilde G_t$ is isometric to $\tilde g_t$ and has ADM mass $\widetilde m(t)$.

On the dilation region, whenever $\Phi_t(x)$ lies in the harmonically flat region of \eqref{Eq: conformal flat}, we have $(G_0)_t=g_{euc}$. Moreover, $U_t,V_t,W_t$ are Euclidean harmonic, and
$$G_t=U_t^{\frac{4}{n-2}}g_{euc},\qquad \widetilde G_t=W_t^{\frac{4}{n-2}}g_{euc}.$$
The conformal factors $U_t$ and $W_t$ are positive and tend to $1$ at infinity.

\medskip
We recall the quantitative positive mass theorem due to Lee \cite{Lee}.

\begin{theorem}\label{thm:quantitative positive mass}
For every \(n\geq 3\), \(\alpha>1\), and \(\varepsilon>0\), there exists
\(\delta>0\) such that the following holds.
Let \((M^n,g)\) be complete asymptotically flat with nonnegative scalar curvature,
and suppose that the positive mass theorem holds on \(M\). Assume that the
chosen end is harmonically flat, so that for \(|x|>R\),
\[
    g=U^{\frac{4}{n-2}}g_{euc},
\]
with \(U>0\) harmonic and \(U\to1\) at infinity. If
\[
    m(g)<\delta R^{n-2},
\]
then, for every \(|x|\geq\alpha R\),
\[
    |U(x)-1|
    \leq
    \varepsilon\left(\frac{R}{|x|}\right)^{n-2}.
\]
\end{theorem}

We next establish the following convergence result for the normalized conformal factors as $t\to \infty$.

\begin{lemma}\label{lem:asymptotic_expansion}
Let
\[
\mathscr M:=\lim_{t\to\infty}m(t).
\]
Uniformly for \(|x|\geq4R_{\max}\), the following pointwise limits hold:
\[
\lim_{t\to\infty} U_t(x) = 1 + \frac{\mathscr M}{2}|x|^{2-n},\]
\[\lim_{t\to\infty} V_t(x) = -1 + \frac{\mathscr M}{2}|x|^{2-n},\] 
and
\[
\lim_{t\to\infty} W_t(x) = 1.
\]
\end{lemma}

\begin{proof}
By Lemma \ref{lem:radius_bound} and \eqref{Eq: conformal flat}, we can enlarge $R_{\mathrm{max}}$ so that for every $t\geq 0$ we have
$$\Sigma^*_t\subset B_{R_{\mathrm{max}}/2},$$
and
$$(G_0)_t=g_{euc},\quad X=\frac{2}{n-2}r\partial_r\quad \text{on}\quad\{|x|>R_{\mathrm{max}}\}.$$
Then $U_t,V_t,W_t$ are Euclidean harmonic on this fixed exterior region.
Define
\[
\bar{U}_t(x) := U_t(x) - \left(1 + \tfrac{m(t)}{2}|x|^{2-n}\right)\]
and
\[
\bar{W}_t(x) := W_t(x) - \left(1 + \tfrac{\widetilde{m}(t)}{2}|x|^{2-n}\right).
\]
The harmonic expansions and the conformal mass formula give, for each fixed $t$,
$$\bar U_t(x)=O(|x|^{1-n}),\qquad \bar W_t(x)=O(|x|^{1-n}).$$
Using $m'(t)=-2\widetilde m(t)$, the evolution equation becomes
\[
\frac{d}{dt}\bar{U}_t = 2\left[\bar{U}_t - \bar{W}_t + \tfrac{1}{n-2}r\partial_r\bar{U}_t\right]
\]
for almost every $t$.

By the smoothing argument in the proof of Proposition \ref{Prop: mass capacity}, using the outermost property of $\Sigma_t^*$, we may approximate $(E_t^*,G_t)$ by smooth asymptotically flat manifolds with nonnegative scalar curvature and smooth mean-convex boundary, with convergent masses and capacities and a common harmonically flat exterior. Applying the compactification and approximation argument of \cite[pp. 95--96]{Bray-Lee}, we obtain complete approximating metrics whose masses converge to $\tilde m(t)$ and whose conformal factors converge to $W_t$. Theorem \ref{thm:quantitative positive mass} and Lemma \ref{lem:vanishing_mass} then imply
$$W_t\to 1\quad \text{uniformly for}\quad |x|\geq 2R_{\mathrm{max}}.$$

Given \(\varepsilon>0\), choose \(t_0\) sufficiently large so that
\[
\sup_{|x|=2R_{\mathrm{max}}}|\bar{W}_t(x)|<2\varepsilon\quad \text{for every}\quad t\geq t_0.
\]
Standard decay estimates for harmonic functions then yield
\[
|\bar{W}_t(x)| \leq C\varepsilon|x|^{2-n-\alpha} \quad \text{for}\quad t\geq t_0,\, |x|>4R_{\max},
\]
where $C$ depends only on $n$ and $R_{\mathrm{max}}$.
Along the characteristics of the evolution equation, we have
\[
  \frac{d}{ds}\left[e^{-2s}\bar{U}_s\left(e^{\frac{2(t-s)}{n-2}}x\right)\right]
  =-2e^{-2s}\bar{W}_s\left(e^{\frac{2(t-s)}{n-2}}x\right).
\]
Integrating from \(t_0\) to \(t\), for $|x|\geq 4R_{\mathrm{max}}$, we obtain
\[
\begin{split}
    |\bar U_t(x)|&\leq e^{2(t-t_0)}\left|\bar{U}_{t_0}\left(e^{\frac{2(t-t_0)}{n-2}}x\right)\right|+2C\ve |x|^{1-n}\int_{t_0}^te^{-\frac{2(t-s)}{n-2}}\,ds\\
    &\leq \left[C_{t_0}e^{-\frac{2(t-t_0)}{n-2}}+(n-2)C\ve\right]|x|^{1-n}.
\end{split}
\]
Letting $t\to \infty$ and then $\ve\to 0$, we conclude that
$$\bar U_t\to 0\quad \text{uniformly for}\quad |x|\geq 4R_{\mathrm{max}}.$$
It follows that
$$U_t(x)\to 1+\frac{\mathscr M}{2}|x|^{2-n}\quad\text{and so}\quad V_t=U_t-2W_t\to -1+\frac{\mathscr M}{2}|x|^{2-n}$$
uniformly on this region.
This completes the proof.
\end{proof}

We next identify a subsequential Hausdorff limit of $\Sigma_t^*$ as the zero set of the corresponding limiting function $V_\infty$. In the following argument, we use the Euclidean dilations
$$\Psi_t(x)=e^{\frac{2t}{n-2}}x,\quad x\in \mathbb R^n.$$
Note that the pullback metric $\Psi_t^*g_t$ is defined on $$\left\{|x|>e^{-\frac{2t}{n-2}}r_0\right\}.$$

Let $t_i\to \infty$ and write
$$E_i=\Psi_{t_i}^{-1}(E_{t_i})\quad \text{and}\quad \Sigma_i=\Psi_{t_i}^{-1}(\Sigma_{t_i}).$$
Denote
$$U_i=e^{t_i}(Uu_{t_i})\circ \Psi_{t_i}\quad\text{and}\quad V_i=e^{t_i}(Uv_{t_i})\circ \Psi_{t_i}.$$
Here the pullbacks refer to the coordinate end, and local limits are taken in
$\mathbb{R}^n\setminus\{0\}$. Write
$r_\infty=(\mathscr{M}/2)^{1/(n-2)}$.

\begin{lemma}\label{lem:zero_set_limit}
We have $\mathscr M>0$. Moreover, we have $E_i\to E_\infty=\{|x|>r_\infty\}$ locally in $L^1$, $\Sigma_i\to S_{r_\infty}$ locally in the Hausdorff sense, and 
$$V_i\to V_\infty=-1+\frac{\mathscr M}{2}|x|^{2-n}\quad\text{in}\quad C^0_{\mathrm{loc}}(E_\infty).$$
\end{lemma}

\begin{proof}
On every fixed compact annulus away from the origin, Remark \ref{rem:asymptotic_bounds} gives uniform bounds
$$0<c\leq U_i\leq C,\qquad -U_i\leq V_i\leq 0$$
for sufficiently large $i$. On such an annulus, we have $$\Psi_{t_i}^*g_{t_i}=U_i^{\frac{4}{n-2}}g_{euc}$$ and $\Sigma_i$ is locally minimizing with respect to this metric. The constant-area identity gives locally uniform area bounds. The arguments of Lemma \ref{Lem: sphere lower area}--\ref{Lem: Holder v_i} and Lemma \ref{Lem: density lower bound} therefore give uniform two-sided density estimates and uniform local H\"older bounds for $V_i$, extended by zero outside $E_i$. We also have uniform local H\"older estimates for $U_i$ by integrating $V_i$.

Consequently, we have $E_i\to E_\infty$ in $L^1_{\mathrm{loc}}$, $\Sigma_i\to \Sigma_\infty$ locally in Hausdorff sense, $U_i\to U_\infty$, and $V_i\to V_\infty$ locally uniformly.
The density estimates give
$$\partial E_\infty=\Sigma_\infty\quad\text{in}\quad \mathbb R^n\setminus\{0\}.$$
They also imply that compact subsets of $E_\infty$ eventually lie in $E_i$. Therefore, $U_\infty$ and $V_\infty$ are harmonic in $E_\infty$,
and $V_\infty$ is nonpositive there. Passing the H\"older estimates to the limit shows that $V_\infty$ extends continuously up to $\Sigma_\infty$ with zero value.

On the unbounded component of $E_\infty$, Lemma \ref{lem:asymptotic_expansion} and unique continuation imply
$$U_\infty(x)=1+\frac{\mathscr M}{2}|x|^{2-n}\quad\text{and}\quad V_\infty(x)=-1+\frac{\mathscr M}{2}|x|^{2-n}.$$
The strong maximum principle gives strict negativity on this component.

We claim $\mathscr M>0$. Otherwise, we have $\mathscr M=0$ and $V_\infty\equiv -1$ on the unbounded component, which therefore has no boundary in $\mathbb R^n\setminus\{0\}$. This implies $E_\infty=\mathbb R^n\setminus\{0\}$ and $\Sigma_\infty\cap (\mathbb R^n\setminus\{0\})=\emptyset$. By Lemma \ref{lem:radius_bound} and local Hausdorff convergence, for every fixed $r>0$, the image of $S_r$ under $\Psi_{t_i}$ eventually encloses $\Sigma_{t_i}$. By Lemma \ref{lem:radius_bound} and locally uniform H\"older estimates of $U_i$, we conclude that $U_i\to U_\infty\equiv 1$ on $S_r$. The minimizing property gives
$$\mathscr A\leq \lim_{i\to\infty}\int_{S_r}U_i^{\frac{2(n-1)}{n-2}}\,dA_{g_{euc}}=n\omega_n r^{n-1}\to 0\quad\text{as}\quad r\to 0,$$
where $\mathscr A$ is the area constant from Corollary \ref{Cor: constant area}.
This contradicts to the positivity of $\mathscr A$.

Put
$$r_\infty:=\left(\frac{\mathscr M}{2}\right)^{\frac{1}{n-2}}.$$
The formula of $V_\infty$, its negativity on $E_\infty$, and its zero boundary value show that the unbounded component of $E_\infty$ is exactly $\{|x|>r_\infty\}$. In particular, $\Sigma_\infty$ contains $S_{r_\infty}$ and has no point outside $S_{r_\infty}$. An area comparison argument with coordinate spheres
$S_{(1+\varepsilon)r_\infty}$ with $\varepsilon\to 0$, combined with
weighted perimeter lower semicontinuity and the two-sided density
estimates, gives
\[
n\omega_n(2\mathscr{M})^{\frac{n-1}{n-2}}
=
\int_{S_{r_\infty}}
U_\infty^{\frac{2(n-1)}{n-2}}\,dA_{g_{\mathrm{euc}}}
\leq \mathscr{A}
\leq n\omega_n(2\mathscr{M})^{\frac{n-1}{n-2}}.
\]
The density lower bound rules out any additional point of $\Sigma_\infty$
away from the sphere and the origin. Thus $\Sigma_\infty=S_{r_\infty}$
and $E_\infty=\{|x|>r_\infty\}$. Since every convergent subsequence has
these limits, the conclusion holds for the full sequence.
We complete the proof.
\end{proof}

Finally, we can establish the following theorem.

\begin{theorem}\label{thm:sphere_convergence}
Set $$\mathscr M=\lim_{t\to\infty} m(t)\quad\text{and}\quad r_\infty=\left(\frac{\mathscr M}{2}\right)^{\frac{1}{n-2}}.$$
As $t\to\infty$, we have
$$(\Psi_t^{-1}(\overline E_t),\Psi_t^{-1}(\Sigma_t),\Psi_t^*g_t)\to \left(\mathbb R^n\setminus B_{r_\infty},S_{r_\infty},\left(1+\frac{\mathscr M}{2}|x|^{2-n}\right)^{\frac{4}{n-2}}g_{euc}\right)$$
smoothly, meaning that the boundary hypersurfaces converge smoothly and the metrics converge smoothly on compact subsets. In particular, we have
$$\mathscr M=\frac12
    \left(
        \frac{\mathscr A}{n\omega_n}
    \right)^{\frac{n-2}{n-1}} .$$
\end{theorem}
\begin{proof}
Let $t_i\to \infty$, and use the notation $E_i, \Sigma_i, U_i$ as above. By Lemma \ref{lem:zero_set_limit}, we have $E_i\to E_\infty=\{|x|>r_\infty\}$ in $L^1_{\mathrm{loc}}$ and $\Sigma_i\to S_{r_\infty}$ in the local Hausdorff sense. The uniform positivity and H\"older estimates for $U_i$ in the proof of Lemma \ref{lem:zero_set_limit} imply
uniform Euclidean almost-minimizing estimates for $\Sigma_i$, by the argument of Lemma \ref{Lem: almost minimizing limit}. The local minimizing property and the perimeter-convergence argument of Corollary \ref{Cor: convergence} then give the varifold convergence of $\Sigma_i$ to $S_{r_\infty}$ with multiplicity one. Theorem 2.27 then improves the convergence to graphical
$C^{1,\alpha}$-convergence near $S_{r_\infty}$. For large $i$, this graph
is a closed sphere whose image under $\Psi_{t_i}$ encloses the original
compact core; it is an admissible enclosure by itself, so minimizing
area rules out any additional component of $\Sigma_{t_i}$. Since
$t_i\to\infty$ is arbitrary, we conclude that $\Psi_t^{-1}(\Sigma_t)$
converges to $S_{r_\infty}$ in $C^{1,\alpha}$-graphical sense as
$t\to\infty$. Boundary Schauder estimates, the same fixed-length
integration along characteristics, and nesting give uniform one-sided
$C^{1,\alpha}$ estimates for $V_i$ and $U_i$. By the same bootstrap
argument as in the proof of Proposition 2.31, applied uniformly at all
sufficiently large times, the convergence can be upgraded to be smooth
as desired.
\end{proof}

\section{Proof of the Penrose inequality}
\label{sec: proof penrose}

We now prove Theorem~\ref{Thm: intro main}.
\begin{proof}
Write $\Sigma_0=\partial M$ and $\mathscr A=\mathcal H^{n-1}_g(\Sigma_0)$.
Assume first that the asymptotic end is harmonically flat. By Corollary \ref{Cor: constant area}, the conformal flow satisfies $\mathcal H^{n-1}_{g_t}(\Sigma_t)=\mathscr A$. Theorem \ref{thm:sphere_convergence} gives
$$\lim_{t\to\infty} m(t)=\frac12
    \left(
        \frac{\mathscr A}{n\omega_n}
    \right)^{\frac{n-2}{n-1}} .$$
    Since $m(t)$ is non-increasing by Proposition \ref{Prop: mass monotonicity}, we obtain
    $$m_{\mathrm{ADM}}(M,g)\geq \frac{1}{2}\left(
        \frac{\mathcal H^{n-1}_g(\partial M)}{n\omega_n}
    \right)^{\frac{n-2}{n-1}} .$$

   For a general asymptotically flat end, apply the harmonically flat
case to the approximations in Lemma \ref{Lem: approximation penrose}
and pass to the limit in their masses and boundary areas.

It remains to prove the rigidity. Suppose equality holds. By
Corollary \ref{Cor: E_t AF} and the metric estimates following
Corollary \ref{cor: separation time slices}, each exterior
$(\overline{E_t},g_t)$ satisfies the hypotheses of
Lemma \ref{Lem: approximation penrose}. Applying the inequality to
its approximations and passing to the limit, we obtain
\[
m(0)\geq m(t)\geq
\frac{1}{2}\left(
    \frac{\mathcal H^{n-1}_{g_t}(\Sigma_t)}{n\omega_n}
\right)^{\frac{n-2}{n-1}}
=
\frac{1}{2}\left(
    \frac{\mathscr A}{n\omega_n}
\right)^{\frac{n-2}{n-1}}
=m(0).
\]
Hence $m(t)\equiv m(0)$.

By Corollary \ref{Cor: convergence} and the outermostness of
$\Sigma_0$, the enclosed compact regions decrease to the region
enclosed by $\Sigma_0$ as $s\downarrow0$, also in Hausdorff distance.
Every smooth capacity test function equal to one near the latter
region is therefore admissible for all sufficiently small $s$.
Capacity monotonicity gives the reverse inequality, so
\[
\mathfrak c(\Sigma_s,g)\rightarrow\mathfrak c(\Sigma_0,g).
\]
Subtracting $m(0)$ from the mass identity in Lemma
\ref{Lem: derivative mass function}, dividing by $t$, and letting
$t\downarrow0$ now gives
\[
m(0)=\mathfrak c(\Sigma_0,g).
\]
Proposition \ref{Prop: mass capacity rigidity} implies that $(M,g)$
is isometric to a Riemannian Schwarzschild exterior.
Conversely, the Schwarzschild exterior realizes equality by direct
computation.
\end{proof}

\appendix
\section{Mean-convex smoothing}

Let $\Omega \subset M$ be an open subset and $\Sigma$ be a union of some of the connected components of $\partial \Omega$. For each constant $\rho > 0$, we define 
$$\mathbf\Lambda_\rho(\Omega, \Sigma)$$
 to be the set of points $x \in \Sigma$, for which there exists a unit-speed minimizing geodesic $\gamma: [0, \rho] \to M$ such that 
\[
\gamma(\rho)=x,\, d(\gamma(0),\Sigma)=\rho\mbox{ and }\gamma([0,\rho))\subset\Omega.
\]
Denote 
\begin{equation}\label{Eq: projection set}
\mathbf\Lambda(\Omega,\Sigma):=\bigcup_{\rho>0}\mathbf\Lambda_\rho(\Omega,\Sigma).
\end{equation}

First let us recall the following notion of  $C^k$-quasi-regular boundary.
\begin{definition}
   The boundary portion $\Sigma$ is called  
   {\it $C^k$-quasi-regular}
   relative to $\Omega$ if $\mathbf\Lambda(\Omega,\Sigma)$ is a
    $C^k$-hypersurface.  
\end{definition}

\begin{proposition}\label{Prop: mean-convex-smoothing}
Suppose that $\Sigma$ is compact and 
$C^2$-quasi-regular relative to $\Omega$ and that the mean curvature of $\mathbf\Lambda(\Omega, \Sigma)$ with respect to the outward-pointing unit normal satisfies $H \geq \delta$ for some constant $\delta > 0$.
Then there exists a sequence of open subsets $\Omega_l\subset \Omega$ such that
\begin{itemize}
\item we have 
$$\partial\Omega_l\cap \partial\Omega=\partial\Omega\setminus \Sigma,$$
and $\Sigma_l:= \partial \Omega_l\setminus \partial\Omega$ is a smooth hypersurface whose mean curvature with respect to the outward-pointing unit normal satisfies $$H\geq \delta-l^{-1};$$
\item $\Omega\setminus\Omega_l$ lies in $l^{-1}$-neighborhood of $\Sigma$, and in particular, $\Sigma_l \to \Sigma$ in the sense of Hausdorff distance as $l \to \infty$;
\item we also have $\mathcal H^{n-1}(\Sigma_l)\to \mathcal H^{n-1}(\mathbf \Lambda(\Omega,\Sigma))$ as $l\to \infty$.
\end{itemize}
\end{proposition}

The proof can be found in \cite{GromovPlateauStein, BiZhuVolumeGrowth}. For completeness, we include Gromov's proof.
 Denote
$$d_\Sigma(q)=\dist(q,\Sigma),\,q\in \Omega.$$
We use
$$s^*=\dist(\Sigma,\partial\Omega\setminus\Sigma),\,F_s=\{q\in\Omega:d_\Sigma(q)\geq s\}\mbox{ and }\Sigma_s=\partial F_s\cap\Omega.$$

For a locally semi-concave function $u$, its distributional covariant Hessian $\nabla^2 u$ is a locally finite symmetric $2$-tensor-valued Radon measure. We write
$$\nabla^2 u\leq Cg\mathcal H^n\mbox{ on }U$$
if, for every continuous symmetric $2$-tensor $T\geq 0$ compactly supported on $U$, we have
$$\langle \nabla^2u,T\rangle_g\leq C\int_M\tr_g T\,\dd\mathcal H^n.$$

\begin{lemma}\label{Lem: control A}
    There are constants \(s_0\in(0,s^*/4)\) and \(C_0<+\infty\), depending only on a compact neighborhood of \(\Sigma\), with the following properties.  

If \(0<a<b<2s_0\), then \(d_\Sigma\) is semiconcave on
\[
 U_{a,b}:=\{q\in\Omega:a<d_\Sigma(q)<b\}.
\]
We have the decomposition
\begin{equation}\label{Eq: decomposition A}
    \nabla^2d_\Sigma=A\,\mathcal H^n+S
\end{equation}
on \(U_{a,b}\), where $S\leq 0$ and
\[
 A\leq \frac{C_0}{a} g
\]
almost everywhere on $U_{a,b}$. Moreover, we have
\[
 |\nabla d_\Sigma|=1\mbox{ and }
 A(\nabla d_\Sigma,\cdot)=0
\]
almost everywhere on $U_{a,b}$.
\end{lemma}
\begin{proof}
   It is easy to verify that $d_\Sigma$ is  semi-concave and satisfies the estimate
    $$\nabla^2d_\Sigma\leq \frac{C_0}{a}g\mbox{ on }U_{a,b},$$
    in the viscosity sense, for some universal constant $C_0$. Since the distribution Hessian $\nabla^2 d_\Sigma$ is a locally finite symmetric $2$-tensor-valued Radon measure, we can make the decomposition
    $$\nabla^2 d_\Sigma=A\mathcal H^n+S,\mbox{ where }A\in L^1_{loc}\mbox{ and } S\bot \mathcal H^n.$$
    We obtain 
    $$A\leq \frac{C_0}{a} g$$
    almost everywhere on $U_{a,b}$, restricted to the intersection of the Alexandrov points of $d_\Sigma$ and the Lebesgue points of $A$. The local semi-concavity of $d_\Sigma$ yields $S\leq 0$. Since $d_\Sigma$ is a distance function, it is differentiable almost everywhere, and we have $|\nabla d_\Sigma|=1$ at all differentiable points. Again, from the  local semi-concavity of $d_\Sigma$, we have $\nabla d_\Sigma\in BV_{loc}$ and $A=D_{ap}(\nabla d_\Sigma)$ almost everywhere, where $D_{ap}(\nabla d_\Sigma)$ is the approximate differential. From
    $$D_{ap}(|\nabla d_\Sigma|^2)=2\langle D_{ap}(\nabla d_\Sigma)(\cdot),\nabla d_\Sigma\rangle=2A(\nabla d_\Sigma,\cdot),$$
    we have $A(\nabla d_\Sigma,\cdot)=0$ almost everywhere.
\end{proof}

\begin{lemma}
    After decreasing $s_0$, there is a constant $K_0<+\infty$ such that, for almost every $x\in \Omega$ with $0<d_\Sigma(x)<2s_0$, if $e=\nabla d_\Sigma(x)$, then we have
    $$\tr_{e^\bot}A(x)\leq -\delta+K_0d_\Sigma(x).$$
\end{lemma}
\begin{proof}
Exclude the nondifferentiability set of $d_\Sigma$ and the critical
values of the normal exponential maps on countably many coordinate
neighborhoods of $\Lambda(\Omega,\Sigma)$. Both sets have zero
$\mathcal H^n$-measure, the latter by the area formula.

Take $x$ outside these sets. In particular, $x$ has a unique nearest
point $p$ in $\Lambda(\Omega,\Sigma)$. Let
\[
\gamma:[0,t]\to(\Omega,g)
\]
be the minimizing geodesic connecting $p$ and $x$, then $d_\Sigma$
is $C^2$ near $x$ and along the interior of $\gamma$, since the nearest
point is unique and the minimizing normal segment has no focal or
cut points. The inverse function theorem applies to the normal
exponential map. From the Riccati equation we can derive
$$\tr_{e^\bot}\nabla^2d_\Sigma(x)\leq -\delta+K_0d_\Sigma(x),$$
where $K_0\geq 0$ is a constant such that $\Ric_g\geq -K_0 g$ in a neighborhood of $\Sigma$. As before, we have $A=\nabla^2 d_\Sigma$ almost everywhere among the differentiable points of $d_\Sigma$, the desired estimate follows.
\end{proof}

Fix $s\in (0,s_0)$ and $h\in (0,s/4)$. Define
$$U_{s,h}:=\{q\in \Omega: |d_\Sigma(q)-s|<h\}.$$
Let 
$$\delta_s:=\delta-K_0(s+h)\mbox{ and }C_s=\frac{C_0}{s-h}.$$
We always take $s$ small enough such that $\delta_s>0$. Set
$$L_s:=(n-2)C_s+\delta_s,\,\alpha_s(t)=-L_s^{-1}e^{-L_s(t-s)},\mbox{ and }w_s=\alpha_s\circ d_\Sigma.$$
Clearly we have 
$$\alpha_s'(t)=e^{-L_s(t-s)}\mbox{ and }\alpha_s''(t)=-L_s\alpha_s'(t).$$
For simplicity, we denote
$$a_s^-=e^{-L_s h}\mbox{ and }a_s^+=e^{L_sh}.$$
\begin{lemma}
    On $U_{s,h}$, in the sense of tensor-valued Radon measures, the function $w_s$ satisfies
    $$\tr(P\nabla^2w_s)\leq -\delta_s a_s^-\mathcal H^n $$
    for every measurable field $P$ of rank-$(n-1)$ $g$-orthogonal projection. It also satisfies
    $$|\nabla w_s|\leq a_s^+$$
    almost everywhere, and
    $$\nabla^2 w_s\leq a_s^+C_s g\mathcal H^n.$$
\end{lemma}
\begin{proof}
    From the chain rule we have
    $$\nabla w_s=\alpha_s'(d_\Sigma)\nabla d_\Sigma$$
    and
    $$\nabla^2 w_s=\alpha_s'(d_\Sigma)\nabla^2 d_\Sigma+\alpha_s''(d_\Sigma)\nabla d_\Sigma\otimes \nabla d_\Sigma \mathcal H^n.$$
    Using \eqref{Eq: decomposition A} we have the Lebesgue decomposition
    $$\nabla^2 w_s=B\mathcal H^n+\alpha_s'(d_\Sigma)S,$$
    where
    $$B=\alpha_s'(d_\Sigma)A+\alpha_s''(d_\Sigma)e\otimes e\mbox{ with }e=\nabla d_\Sigma.$$
    The last two estimates follow from Lemma \ref{Lem: control A}.
    
    Again, from Lemma \ref{Lem: control A} we know that $A$ has eigenvalue $\mu_1,\ldots,\mu_{n-1},0$ with $\mu_i\leq C_s$, where the eigenvalue $0$ corresponds to the eigenvector $e$. Therefore, the eigenvalues of $B$ are
    $$\alpha_s'(d_\Sigma)\mu_1,\ldots,\alpha_s'(d_\Sigma)\mu_{n-1},\alpha_s''(d_\Sigma).$$
  Therefore, for every rank-$(n-1)$ orthogonal projection $P$,
    $$\tr(PB)\leq \max\left\{\alpha_s'(d_\Sigma)\sum_{i=1}^{n-1}\mu_i,\alpha_s''(d_\Sigma)+(n-2)C_s\alpha_s'(d_\Sigma)\right\}\leq -\delta_s a_s^-.$$
    Combined with the non-positivity of the singular part, we obtain the first estimate.
\end{proof}

\begin{lemma}\label{Lem: smooth mollification}
    Let $V'\Subset V$ be open subsets of $(M,g)$ and let $u$ be locally semi-concave on $V$. Suppose that $u$ satisfies
    $$|\nabla u|\leq L,\,\nabla^2 u\leq Cg\mathcal H^n,$$
    and
    $$\tr(P\nabla^2u)\leq -c\mathcal H^n$$
    on $V$, for constants $L,C<+\infty$ and $c>0$ and for every measurable field $P$ of rank-$m$ orthogonal projection. Then there are smooth functions $u_\varepsilon$ on $V'$ such that, as $\varepsilon\to 0$,
    $$u_\varepsilon\to u\mbox{ uniformly and in }W^{1,1}(V'),$$
    and
    $$|\nabla u_\varepsilon|\leq L+o(1)\mbox{ and }\tr(P\nabla^2 u_\varepsilon)\leq -c+o(1)$$
    uniformly on compact subsets of $V'$, simultaneously for all rank-$m$ orthogonal projections $P$.
\end{lemma}
\begin{proof}
    This follows from the mollification
    $$u_\varepsilon(x)=\int_{T_xM}u(\exp_x(\varepsilon z))\varphi(z)\,\dd z_x,$$
    where $\dd z_x$ is the Euclidean measure induced by $g_x$, and $\varphi$ is the standard Euclidean mollifier.
\end{proof}

\begin{proof}[Gromov's bending and smoothing proof]
    We apply Lemma \ref{Lem: smooth mollification} to $w_s$ for approximation functions $w_{s,\ve}$ such that $w_{s,\ve}$ are smooth on 
    $$U_{s,h}'=\{q\in\Omega:|d_\Sigma(q)-s|<3h/4\},$$
    and satisfy
    $$\|w_{s,\ve}-w_s\|_{L^\infty}\to 0,\,|\nabla w_{s,\ve}|\leq a_s^++o(1)\mbox{ and }\tr(P\nabla^2 w_{s,\ve})\leq -\delta_sa_s^-+o(1)$$
   as $\ve \to 0$. For a regular value $\tau$ of $w_{s,\ve}$ sufficiently close to $\alpha_s(s)$, we define
   $$\Omega_{s,h,\ve,\tau}:=\{q\in \Omega:d_\Sigma(q)>s+h/2\}\cup \left(\{q\in \Omega:w_{s,\ve}(q)>\tau\}\cap U_{s,h}'\right)$$
   and $\Sigma_{s,h,\ve,\tau}:=\partial\Omega_{s,h,\ve,\tau}\cap \Omega$. Let $P$ be the orthogonal projection onto the tangent space of $\Sigma_{s,h,\ve,\tau}$, then the mean curvature of $\Sigma_{s,h,\ve,\tau}$ satisfies
   $$H_{s,h,\ve,\tau}=-\frac{\tr(P\nabla^2 w_{s,\ve})}{|\nabla w_{s,\ve}|}\geq \frac{\delta_s a_s^--o(1)}{a_s^++o(1)}.$$
   For any integer $l>1$, we can take $s_l$, $h_l$ and $\ve_l$ small enough, and $\tau_l$ sufficiently close to $\alpha_{s_l}(s_l)$ such that $\Sigma_{s_l,h_l,\ve_l,\tau_l}$ is the desired smooth approximation of $\Sigma$ with $H_{s_l,h_l,\ve_l,\tau_l}\geq \delta-l^{-1}$.

   To conclude that $\Sigma_{s_l,h_l,\ve_l,\tau_l}$ converges to $\Sigma$ in the Hausdorff sense, it suffices to show $\Sigma_*=\Sigma$ for the Hausdorff limit $\Sigma_*$ of $\Sigma_{s_l,h_l,\ve_l,\tau_l}$. On the one hand, we can guarantee that $\Sigma_{s_l,h_l,\ve_l,\tau_l}$ lies in the $l^{-1}$-neighborhood of $\Sigma$, so we must have $\Sigma_*\subset \Sigma$. On the other hand, for each $p\in\Lambda_\rho(\Omega,\Sigma)$, the
inward normal segment has $d_\Sigma$ equal to its parameter up to
$\rho$. At parameters $s_l-h_l/2$ and $s_l+h_l/2$, the values of
$w_{s_l,\varepsilon_l}$ lie on opposite sides of $\tau_l$ when the
smoothing and level errors are sufficiently small. Thus the segment
meets $\Sigma_{s_l,h_l,\varepsilon_l,\tau_l}$ at a point tending to $p$,
and $\Lambda(\Omega,\Sigma)\subset\Sigma_*$. Since $\Sigma$ consists of boundary components of $\Omega$, we have $\overline{\mathbf\Lambda(\Omega,\Sigma)}=\Sigma$, and the previous discussion yields $\Sigma_*=\Sigma$ as desired.

   For the measure convergence, we claim
   $$\mathcal H^{n-1}(\Sigma_{s_l})\to \mathcal H^{n-1}(\mathbf \Lambda(\Omega,\Sigma))\mbox{ as }s_l\to 0.$$
   Since $\Sigma_{s_l}$ is a subset of the image under the exponential map from a smooth section of the normal bundle of $\mathbf\Lambda(\Omega,\Sigma)$, $\Sigma_{s_l}$ is $(n-1)$-rectifiable. On one hand, portions of $\Sigma_{s_l}$ can be written as smooth graphs over $\mathbf \Lambda_{\rho}(\Omega,\Sigma)$ for $l$ large enough, and so we have 
$$\liminf_{l\to \infty} \mathcal H^{n-1}(\Sigma_{s_l})\geq \mathcal H^{n-1}(\mathbf \Lambda_{\rho}(\Omega,\Sigma))\to \mathcal H^{n-1}(\mathbf \Lambda(\Omega,\Sigma)).$$
On the other hand, we have from the Riccati equation, the mean curvature of $\mathbf \Lambda_{\rho}(\Omega,\Sigma)$, and the area formula that
$$\mathcal H^{n-1}(\Sigma_{s_l})\leq e^{C_*s_l^2}\cdot\mathcal H^{n-1}(\mathbf \Lambda_{s_l}(\Omega,\Sigma)),$$
where $C_*$ denotes the bound of the Ricci curvature around $\Sigma$. Clearly, we have 
$$\limsup_{l\to \infty} \mathcal H^{n-1}(\Sigma_{s_l})\leq \mathcal H^{n-1}(\mathbf \Lambda(\Omega,\Sigma)),$$
and the claim follows. 
   
   It remains to choose regular values $\tau=\tau_\ve\to \alpha_{s_l}(s_l)$ such that
   $$\mathcal H^{n-1}(\Sigma_{s_l,h_l,\ve,\tau})\to \mathcal H^{n-1}(\Sigma_{s_l})\mbox{ as }\ve\to 0,\,\tau\to \alpha_{s_l}(s_l),$$
   and then the desired measure convergence follows from a diagonal argument. Recall that we have $w_{s_l,\ve}\to w_{s_l}$ in $W^{1,1}(U_{s_l,h_l}')$. Take $\beta_\ve\to 0$ slowly enough so that
   \begin{equation}\label{Eq: error w_s}
       \|w_{s_l,\ve}-w_{s_l}\|_\infty+\|\nabla w_{s_l,\ve}-\nabla w_{s_l}\|_{L^1}=o(\beta_\ve)\mbox{ as }\ve\to 0.
   \end{equation}
   Set $$I_\ve=(\alpha_{s_l}(s_l)-\beta_\ve,\alpha_{s_l}(s_l)+\beta_\ve).$$
   For simplicity, we denote
\[
A_\varepsilon(t)
:=\mathcal H^{n-1}\bigl(\{w_{s_l,\varepsilon}=t\}
\cap U'_{s_l,h_l}\bigr).
\]
Then we have
\[
\begin{aligned}
\int_{I_\varepsilon} A_\varepsilon(t)\,dt
&=\int_{I_\varepsilon} A_0(t)\,dt \\
&\quad+\int_{U'_{s_l,h_l}}
\left(
|\nabla w_{s_l,\varepsilon}|
\mathbf{1}_{\{w_{s_l,\varepsilon}\in I_\varepsilon\}}
-
|\nabla w_{s_l}|
\mathbf{1}_{\{w_{s_l}\in I_\varepsilon\}}
\right)d\mathcal H^n \\
&=\int_{I_\varepsilon} A_0(t)\,dt+o(\beta_\varepsilon),
\end{aligned}
\]
where $s_l$ is chosen so that its level has perimeter
$\mathcal H^{n-1}(\Sigma_{s_l})$ and $\alpha_{s_l}(s_l)$ is a
Lebesgue point of $A_0$. In estimating the second term, (A.3)
controls the gradient difference. The indicator difference lies in
intervals of width $o(\beta_\varepsilon)$ about the endpoints of
$I_\varepsilon$; coarea and the Lebesgue-point property bound its
contribution by $o(\beta_\varepsilon)$. Then we have
\[
\int_{I_\varepsilon} A_\varepsilon(t)\,dt
=
|I_\varepsilon|\cdot\mathcal H^{n-1}(\Sigma_{s_l})
+o(\beta_\varepsilon),
\]
which yields the existence of a regular value
$\tau_\varepsilon\in I_\varepsilon$ with
\[
\mathcal H^{n-1}(\Sigma_{s_l,h_l,\varepsilon,\tau_\varepsilon})
\leq \mathcal H^{n-1}(\Sigma_{s_l})+o(1).
\] From $\Omega_{s_l,h_l,\ve,\tau_\ve}\to F_{s_l}$ in $L^1$-sense, we have $$\liminf_{\ve\to 0}\mathcal H^{n-1}(\Sigma_{s_l,h_l,\ve,\tau_\ve})\geq \mathcal H^{n-1}(\Sigma_{s_l}).$$
   This completes the proof.
\end{proof}

\section{Approximation lemmas}

\begin{lemma}\label{Lem: approximation 1}
In the setting of Lemma \ref{Lem: conformal minimizers}, there exist asymptotically flat metrics \(g_j\), agreeing with \(g\) on an exhaustion of \(M\) by compact sets containing \(\Sigma\), such that \(g_j\) is Schwarzschild outside a compact set,
\[
m(g_j)=m(g),\qquad
Y_N(g_j)\rightarrow Y_N(g),\qquad
Y_D(g_j)\rightarrow Y_D(g).
\]
Moreover, the quadratic forms in Lemma \ref{Lem: conformal minimizers} remain coercive for \(j\) large enough.
\end{lemma}

\begin{proof} Write \(m=m(g)\), and let
\[
g_{\mathrm S}
=\left(1+\frac{m}{2|x|^{n-2}}\right)^{\frac4{n-2}}
g_{\mathrm{euc}}
\]
on a sufficiently far portion of the asymptotically flat end. Choose a smooth cutoff \(\chi_j\) equal to one for \(|x|\le j\) and zero for \(|x|\ge2j\), with
\[
|\partial\chi_j|\le Cj^{-1},
\qquad
|\partial^2\chi_j|\le Cj^{-2},
\]
and set
\[
g_j=\chi_jg+(1-\chi_j)g_{\mathrm S},
\]
extending \(g_j=g\) over the compact part of \(M\). Then \(g_j\to g\) uniformly as bilinear forms, \(R_{g_j}\) vanishes outside a compact set, and
\[
m(g_j)=m.
\]
The boundary metric and mean curvature are unchanged.

Choose
\[
\frac{n-2}{2}<q<\min\{\tau,n-2\},
\]
and define
\[
f_j=R_{g_j}\frac{d\mu_{g_j}}{d\mu_g}-R_g.
\]
The asymptotic estimates give \(f_j=0\) for \(|x|\le j\) and
\[
|f_j(x)|\le C|x|^{-q-2}
\qquad (|x|\ge j).
\]
Consequently,
\[
\|f_j\|_{L^{n/2}(M,g)}
+
\|f_j\|_{L^{2n/(n+2)}(M,g)}
\rightarrow0.
\]
We also have
\[
\int_M f_j\,d\mu_g\rightarrow0.
\]
Indeed, scalar curvature is its linear divergence term plus an error of order \(O(|x|^{-2q-2})\). The integrated divergence terms cancel because \(g_j=g\) near \(\{|x|=j\}\) and their ADM masses agree. The remaining integral has the order \(O(j^{n-2-2q})\), which tends to zero.

For either $Y_N$ or $Y_D$, denote the corresponding energies by
$\mathcal E_j$ and $\mathcal E$.
Write an admissible function as $u=1+w$.
H\"older's inequality gives
\[
\begin{aligned}
\left|\int_M f_ju^2\,d\mu_g\right|
&\le
\left|\int_M f_j\,d\mu_g\right|
+2\|f_j\|_{L^{2n/(n+2)}(M,g)}
  \|w\|_{L^{2n/(n-2)}(M,g)}
\\
&\quad
+\|f_j\|_{L^{n/2}(M,g)}
  \|w\|_{L^{2n/(n-2)}(M,g)}^2.
\end{aligned}
\]
The Sobolev inequality bounds the norms of $w$ by
$C\|\nabla w\|_{L^2(M,g)}$.
Moreover, uniform convergence of the metric coefficients gives
\[
\left|
\int_M|\nabla u|_{g_j}^2\,d\mu_{g_j}
-
\int_M|\nabla u|_g^2\,d\mu_g
\right|
\le o(1)\|\nabla w\|_{L^2(M,g)}^2.
\]
Since the boundary term is unchanged, there exist
$\varepsilon_j\to0$, independent of $u$, such that
\begin{equation}\label{eq:approx-energy-comparison}
|\mathcal E_j(u)-\mathcal E(u)|
\le
\varepsilon_j
\left(1+\|\nabla w\|_{L^2(M,g)}^2\right).
\end{equation}

By the coercivity argument of Lemma \ref{Lem: conformal minimizers}, we have 
\[
\mathcal E(u)
\ge c\|\nabla w\|_{L^2(M,g)}^2-C
\]
for constants $c>0$ and $C$.
Hence, for sufficiently large $j$,
\[
\mathcal E_j(u)
\ge
\frac c2\|\nabla w\|_{L^2(M,g)}^2-C-1.
\]

Write $Y_j=\inf\mathcal E_j$ and $Y=\inf\mathcal E$.
Testing $\mathcal E_j$ on a minimizer of $\mathcal E$ and using
\eqref{eq:approx-energy-comparison} gives
\[
\limsup_{j\to\infty}Y_j\le Y.
\]
Choose admissible functions $u_j=1+w_j$ such that
\[
\mathcal E_j(u_j)\le Y_j+\frac1j.
\]
The preceding coercivity estimate implies that
$\|\nabla w_j\|_{L^2(M,g)}$ is uniformly bounded.
Applying \eqref{eq:approx-energy-comparison} again, we obtain
\[
Y
\le \mathcal E(u_j)
\le \mathcal E_j(u_j)+o(1)
\le Y_j+\frac1j+o(1).
\]
Therefore, we have 
$$Y\le\liminf_{j\to\infty}Y_j,$$ proving the convergence of both
$Y_N(g_j)$ and $Y_D(g_j)$.
\end{proof}

\begin{lemma}\label{Lem: approximation penrose}
Let $(M,g)$ satisfy the hypotheses of Proposition
\ref{Prop: mass capacity}, and suppose that its boundary $\Sigma$
is nonempty and its own outermost minimizing enclosure. Then there
exist asymptotically flat exterior regions $(M_j,g_j)$, with
nonnegative scalar curvature and harmonically flat ends, whose
compact boundaries $\Sigma_j$ are their own outermost minimizing
enclosures and have minimal regular parts, such that
\[
m(g_j)\to m(g),
\qquad
\mathcal H^{n-1}_{g_j}(\Sigma_j)
\to \mathcal H^{n-1}_g(\Sigma).
\]
Moreover, $g_j$ extends smoothly across $\Sigma_j$.
\end{lemma}

\begin{proof}
Set $\mathscr A=\mathcal H^{n-1}_g(\Sigma)$.
For $0<\varepsilon<1$, use the conformal metric
$\widetilde g_\varepsilon$ from Lemma \ref{Lem: positive mean curvature replacement}, which satisfies
\[
R_{\widetilde g_\varepsilon}\ge0,
\qquad
g\le\widetilde g_\varepsilon
\le(1-\varepsilon)^{-\frac4{n-2}}g,
\]
and
\[
m(\widetilde g_\varepsilon)\rightarrow m(g).
\]
The tangent-cone metric assumption in
Theorem \ref{Thm: intro main} verifies the additional
hypothesis of Subsection \ref{Sec: 3.2}.
By
Lemma \ref{Lem: positive mean curvature replacement}, choose the minimizing hypersurface
$\Sigma_{\varepsilon,\delta}$ strictly outside $\Sigma$.
Comparing its defining functional with the original exterior
$M$ gives
\[
\mathcal H^{n-1}_{\widetilde g_\varepsilon}
(\Sigma_{\varepsilon,\delta})
\le
\mathcal H^{n-1}_{\widetilde g_\varepsilon}(\Sigma)
=
(1-\varepsilon)^{-\frac{2(n-1)}{n-2}}\mathscr A.
\]
By Lemma \ref{Lem: mean convex smoothing}, choose a smooth hypersurface $\Gamma_\varepsilon$
enclosing $\Sigma$, with strictly positive mean curvature with
respect to the outward normal of its exterior region
$E_\varepsilon$, such that
\[
\mathcal H^{n-1}_{\widetilde g_\varepsilon}(\Gamma_\varepsilon)
\le
(1-\varepsilon)^{-\frac{2(n-1)}{n-2}}\mathscr A+\varepsilon.
\]

Apply the harmonic-flat approximation of \cite[Theorem 1.3, Remark 1.4, and Section 3]{Bray-Lee} to the smooth manifold $(\overline {E_\varepsilon},\widetilde g_\varepsilon)$
with boundary.
Using the homogeneous Neumann boundary conditions in the conformal
deformation, we can construct a smooth metric $h_\varepsilon$ satisfying
\[
R_{h_\varepsilon}\ge0,
\qquad
(1-\varepsilon)\widetilde g_\varepsilon
\le h_\varepsilon
\le(1+\varepsilon)\widetilde g_\varepsilon,
\]
\[
|m(h_\varepsilon)-m(\widetilde g_\varepsilon)|<\varepsilon,
\]
and having a harmonically flat end.
The conformal mean-curvature formula preserves the strict
positivity of the mean curvature of $\Gamma_\varepsilon$.

Let $\Sigma_\varepsilon$ be the outermost
$h_\varepsilon$-minimizing enclosure of $\Gamma_\varepsilon$.
The strict mean-curvature barrier implies that
$\Sigma_\varepsilon$ is disjoint from $\Gamma_\varepsilon$.
Hence it is almost-minimizing, its regular part is minimal,
and the ambient metric is smooth across it.

Since $\Sigma_\varepsilon$ encloses $\Sigma$, the
outer-minimizing property of $\Sigma$ gives
\[
\mathcal H^{n-1}_{h_\varepsilon}(\Sigma_\varepsilon)
\ge
(1-\varepsilon)^{\frac{n-1}{2}}\mathscr A.
\]
Conversely, its minimizing property gives
\[
\begin{aligned}
\mathcal H^{n-1}_{h_\varepsilon}(\Sigma_\varepsilon)
&\le
\mathcal H^{n-1}_{h_\varepsilon}(\Gamma_\varepsilon)\\
&\le
(1+\varepsilon)^{\frac{n-1}{2}}
\left[
(1-\varepsilon)^{-\frac{2(n-1)}{n-2}}\mathscr A+\varepsilon
\right].
\end{aligned}
\]
So its area tends to $\mathscr A$, while its mass tends to
$m(g)$.
Restricting $h_\varepsilon$ to the exterior of
$\Sigma_\varepsilon$ and choosing $\varepsilon\downarrow0$
prove the lemma.
\end{proof}

\section{Boundary regularity of the Neumann minimizer}
\label{App: Neumann regularity}

Let \((M^n,g)\) and \(\Sigma\) be as in Subsection~\ref{sec:rigidity-mass-capacity}.

\begin{lemma}\label{lem:neumann-sobolev-trace}
There is \(C>0\) such that every \(v\in C_c^\infty(\overline M)\) satisfies
\begin{equation}\label{eq:neumann-sobolev-trace}
\|v\|_{L^{\frac{2n}{n-2}}(M)}+\|v\|_{L^{\frac{2(n-1)}{n-2}}(\Sigma)}\le C\|\nabla_gv\|_{L^2(M)}.
\end{equation}
\end{lemma}
\begin{proof}
    For smooth boundary these are the usual Sobolev and trace inequalities. Otherwise, outer minimality gives for almost every \(t>0\)
    \[
    \mathcal H_{g_0}^{n-1}(\Sigma\cap\{|v|>t\})\le\mathcal H_{g_0}^{n-1}(\mathring M\cap\{|v|=t\}).
    \]
   The coarea formula then gives 
   \[
   \int_\Sigma|v|d\sigma_{g_0}\le\int_M|\nabla_{g_0}v|d\mu_{g_0}.
   \]
Choose a cutoff function \(\chi\) supported in a compact neighborhood \(K\subset M\) of \(\Sigma\) such that \(1-\chi\) is supported in the asymptotically flat end. Then
   \[
   \begin{aligned}
       \|v\|_{L^{\frac{n}{n-1}}(M,g_0)}\le&\|\chi v\|_{L^{\frac{n}{n-1}}(M,g_0)}+\|(1-\chi)v\|_{L^{\frac{n}{n-1}}(M,g_0)}\\
    \le&C\left(\|\nabla_{g_0}v\|_{L^1(M,g_0)}+\|v\|_{L^1(\Sigma,g_0)}+\|v\|_{L^1(K,g_0)}\right)\\
    \le&C\left(\|\nabla_{g_0}v\|_{L^1(M,g_0)}+\|v\|_{L^1(K,g_0)}\right).
   \end{aligned}
   \]
Now we show 
\[
\|v\|_{L^1(K,g_0)}\le C\|\nabla_{g_0}v\|_{L^1(M,g_0)}
\]
for some uniform \(C>0\). Otherwise, there are \(v_j\in C_c^\infty(\overline M)\) such that
\[
\|v_j\|_{L^1(K,g_0)}=1,\qquad \|\nabla_{g_0}v_j\|_{L^1(M,g_0)}\longrightarrow 0.
\]
A subsequence therefore converges locally in \(L^1\) to a constant \(c\). On the other hand, \(\|v_j\|_{L^{n/(n-1)}(M,g_0)}\) are uniformly bounded. So \(\|c\|_{L^{n/(n-1)}}<\infty\)  and thus \(c=0\), contradicting \(\|c\|_{L^1(K,g_0)}=1\). 

Since \(g\) and \(g_0\) are uniformly equivalent, applying the previous estimates to \(|v|^{2(n-1)/(n-2)}\) gives \eqref{eq:neumann-sobolev-trace}.
\end{proof}

Put \(D_r(x)=\mathring M\cap B_r^g(x)\).
\begin{lemma}\label{lem:neumann-truncation}
For each \(a>0\) there are \(C>0\), \(L>1\) and \(r_1>0\) such that for \(x\in\Sigma\) and \(0<r<r_1\), every \(v\in W^{1,1}(D_{Lr}(x))\) satisfying
\[
\mu_g(D_r(x)\cap\{v\le 0\})\ge ar^n,\qquad \mu_g(D_r(x)\cap\{v\ge 1\})\ge ar^n
\]
also satisfies
\begin{equation}\label{eq:neumann-truncation}
r^{n-1}\le C\int_{D_{Lr}(x)\cap\{0<v<1\}}|\nabla_gv|d\mu_g.
\end{equation}
\end{lemma}
\begin{proof}
    We argue by contradiction. Suppose that for some \(a>0\) there are \(x_j\to x\in\Sigma\), \(r_j>0\) with \(jr_j\to 0\), and \(w_j\in W^{1,1}(D_{2jr_j}(x_j))\) satisfying
    \[
    \mu_g(D_{r_j}(x_j)\cap\{w_j\le 0\})\ge ar_j^n,\qquad \mu_g(D_{r_j}(x_j)\cap\{w_j\ge 1\})\ge ar_j^n,
    \]
    and 
    \[
    r_j^{1-n}\int_{D_{2jr_j}(x_j)\cap\{0<w_j<1\}}|\nabla_gw_j|d\mu_g\longrightarrow 0.
    \]
    Set \(v_j:=\min\{1,\max\{0,w_j\}\}\). After rescaling by \(r_j^{-1}\) we obtain a local BV function \(v_\infty\) and a perimeter-minimizing Caccioppoli set \(E_\infty\). As in the proof of Corollary~\ref{Cor: discrete compactness}, the boundaries converge locally in Hausdorff distance. Moreover,
    \[
   0\in\partial E_\infty,\qquad  |Dv_\infty|(E_\infty)=0,
    \]
    and
    \[
    |E_\infty\cap\{v_\infty\le 0\}|>0,\qquad |E_\infty\cap\{v_\infty\ge 1\}|>0.
    \]
    By \cite[Theorem 1]{BombieriGiusti} (see also \cite[Theorem 5.1]{JerisonTwoHyperplanes}), \(E_\infty\) is connected. So \(v_\infty=c\) almost everywhere in \(E_\infty\) for some constant \(c\), a contradiction.
\end{proof}

\begin{proposition}\label{prop:neumann-conformal-minimizer}
Let \(u_N\in 1+D^{1,2}(M)\) be a nonnegative weak solution of \eqref{Eq: u_N} tending to \(1\) at infinity. Then for some \(\alpha\in(0,1)\), \(u_N\in C_{\mathrm{loc}}^{0,\alpha}(M)\). Moreover, there exists \(c>0\) such that \(u_N\ge c\) on \(M\).
\end{proposition}
\begin{proof}
Choose \(c_0>0\) such that \(c_0s^n\le \mu_g(D_s(y))\le Cs^n\) uniformly for \(y\in M\) and sufficiently small \(s>0\), and set \(a=2^{-n-1}c_0\). After approximation, Lemma~\ref{lem:neumann-truncation} gives some \(\lambda>1\) such that whenever \(\mu_g(A\cap D_s(y))\ge as^n\) and \(\mu_g(D_s(y)\setminus A)\ge as^n\), one has
\[
s^{n-1}\le C\mathcal H_g^{n-1}(\partial A\cap D_{\lambda s}(y)).
\]
Set \(L=2\lambda+1\) and assume \(\mu_g(A\cap D_r(x))\le\frac{1}{2}\mu_g(D_r(x))\). For almost every \(y\in A\cap D_r(x)\), \(\mu_g(D_s(y)\setminus A)=o(s^n)\) as \(s\downarrow 0\). Since \(D_r(x)\subset D_{2r}(y)\),
\[
\mu_g(D_{2r}(y)\setminus A)\ge\mu_g(D_r(x)\setminus A)\ge\frac{1}{2}c_0r^n=a(2r)^n.
\]
By continuity, choose \(0<s_y<2r\) with \(\mu_g(D_{s_y}(y)\setminus A)=as_y^n\). Then \(\mu_g(A\cap D_{s_y}(y))\ge(c_0-a)s_y^n\ge as_y^n\). Then we can choose disjoint balls \(B_{\lambda s_i}^g(y_i)\) with \(s_i=s_{y_i}\) such that \(A\cap D_r(x)\subset \bigcup_iB_{5\lambda s_i}^g(y_i)\). Since \(D_{\lambda s_i}(y_i)\subset D_{Lr}(x)\),
\[
\mu_g(A\cap D_r(x))^{\frac{n-1}{n}}\le C\left(\sum_is_i^n\right)^{\frac{n-1}{n}}\le C\sum_is_i^{n-1}\le C\mathcal H_g^{n-1}(\partial A\cap D_{Lr}(x)).
\]
Thus for \(v\in W^{1,1}(D_{Lr}(x))\),
\[
\|v-\bar v\|_{L^{\frac{n}{n-1}}(D_r(x))}\le C\int_{D_{Lr}(x)}|\nabla_gv|d\mu_g.
\]
With this estimate and \eqref{eq:neumann-sobolev-trace}, the Moser iteration argument gives Harnack and local H\"older estimates up to \(\mathcal S(\Sigma)\). Consequently, \(u_N\) cannot vanish at a (singular) boundary point. Together with the strong maximum principle, this shows that \(u_N>0\) on \(M\). Since \(u_N\to 1\) at infinity, this implies \(\inf_Mu_N\ge c>0\).
\end{proof}

\printbibliography

\end{document}